\documentclass[12pt,amstex]{amsart}

\usepackage{mathptmx}
\usepackage{mathrsfs}

\usepackage{verbatim}
\usepackage{url}
\usepackage[all]{xy}
\usepackage{color}

\usepackage[colorlinks=true,citecolor=blue]{hyperref}

\usepackage{stmaryrd}
\usepackage{epsfig}
\usepackage{amsmath}
\usepackage{amssymb}

\usepackage{amscd}
\usepackage{graphicx}
\usepackage{pstricks}

\theoremstyle{plain}
\newtheorem{thm}{Theorem}[section]
\newtheorem{lem}[thm]{Lemma}

\newtheorem{prop}[thm]{Proposition}
\newtheorem{ques}{Question}
\newtheorem{cor}[thm]{Corollary}
\theoremstyle{definition}
\newtheorem{de}[thm]{Definition}
\newtheorem{exam}[thm]{Example}
\theoremstyle{remark}
\newtheorem{rem}[thm]{Remark}

\numberwithin{equation}{section}

\def \N {\mathbb{N}}

\def \Z {\mathbb{Z}}

\def \O {\mathcal{O}}

\def \A {\mathcal A}
\def \F {\mathcal F}
\def \G {\mathcal{G}}

\def \P {\mathcal P}

\def \X {\mathcal{X}}

\def \RP {{\bf RP}}
\def \M {{\bf M}}

\def \id {{\rm id}}

\def \a {\alpha }

\def \ep {\epsilon}
\def \d {\delta}
\def \D {\Delta}
\def \c {\circ}
\def \w {\omega}

\def \lra {\longrightarrow}

\begin{document}
\title[Topological dynamical systems induced by polynomials]{Topological dynamical systems induced by polynomials and combinatorial consequences}
\author{Wen Huang}
\author{Song Shao}
\author{Xiangdong Ye}

\address{School of Mathematical Sciences, University of Science and Technology of China, Hefei, Anhui, 230026, P.R. China}

\email{wenh@mail.ustc.edu.cn}
\email{songshao@ustc.edu.cn}
\email{yexd@ustc.edu.cn}


\thanks{This research is supported by NNSF of China (12031019, 12090012, 11971455, 11731003).}

\date{2021.9.19}
\date{2022.5.27}
\date{2022.6.13}
\date{July 3, 2022}
\date{July 23, 2022}
\date{Aug. 16, 2022}
\date{Nov. 12, 2022}

\begin{abstract}
 Let $d\in\N$ and $p_i$ be an integral polynomial with $p_i(0)=0$, $1\le i\le d$.
It is shown that if $S$ is piecewise syndetic in $\Z$,
then $$\{(m,n)\in\Z^2: m+p_1(n),\ldots,m+p_d(n)\in S\}$$ is piecewise syndetic in $\Z^2$, which extends the result by Glasner
and Furstenberg for linear polynomials. 
Our result is obtained by showing the density of minimal points of a dynamical system of $\Z^2$ action associated with
the piecewise syndetic set $S$ and the polynomials $\{p_1,\ldots,p_d\}$.

Moreover, it is proved that if $(X,T)$ is minimal, then
for each non-empty open subset $U$ of $X$, there is $x\in U$ with $\{n\in \Z: T^{p_1(n)}x\in U, \ldots, T^{p_d(n)}x\in U\}$
piecewise syndetic.



\end{abstract}

\maketitle
\tableofcontents





\section{Introduction}

In the article, integers, nonnegative integers and natural numbers
are denoted by $\Z$, $\Z_+$ and $\N$ respectively. An {\em integral polynomial} is a polynomial
taking integer values at the
integers. The polynomials $p(n)$ and $q(n)$ are {\em essentially distinct} if $p(n)-q(n)$ is not a constant function.

\subsection{Combinatorial motivation}\
\medskip

A subset $S$ of $\Z$ is {\em piecewise syndetic} if there exists $N\in \N$ such that
$\bigcup_{i=1}^N (S-i)$ contains arbitrarily long intervals. In general, a subset $S$ of a countable abelian group $(G,+)$ is {\em piecewise syndetic} if there exists a finite subset $F$ of $G$ such that $\bigcup_{i\in F} (S-i)$ is {\em thick} in $G$, i.e. it contains a shifted copy of any finite subset of $G$. The set of all piecewise syndetic subsets of $G$ will be denoted by $\F_{ps}(G)$, or simply $\F_{ps}.$

Van der Waerden's theorem states that each piecewise syndetic subset of $\Z$ contains arbitrarily long arithmetic progressions. That is, if $S\subseteq \Z$ is piecewise syndetic, then for all $d\in \N$, the set $\{(m, n)\in \Z^2: m,m+n,\ldots, m+(d-1)n\in S, n\neq 0\}$ is not empty.
Furstenberg and Glasner \cite{FG98} obtained the following beautiful result using the Stone-\v{C}ech compactification of $\Z$.

\medskip
\noindent{\bf Theorem} (Furstenberg-Glasner)
{\em
Let $d\in \N$ and $S$ be piecewise syndetic in $\Z$, then $$\{(m, n)\in \Z^2: m,m+n,\ldots, m+(d-1)n\in S\}$$}
is piecewise syndetic in $\Z^2$.

\medskip

Later Beiglb\"{o}ck \cite{Beiglbock} provided a combinatorial proof for the Furstenberg-Glasner's result just using van der Waerden's theorem.
Bergelson and Hindman  \cite{BH01} extended this result to apply to
many notions of largeness in arbitrary semigroups and to partition regular
structures other than arithmetic progressions. One of the questions asked in \cite[Question 4.7]{BH01} is as follows:

\medskip

\begin{ques}\label{q-bl} 
Let $d\in \N$ and $p_i$ be an integral polynomial with $p_i(0)=0$, $1\le i\le d$.
Is it true that whenever $S$ is piecewise syndetic in $\Z$, then $$\{(m,n)\in \Z^2: m+p_1(n),m+p_2(n),\ldots, m+p_d(n)\in S\}$$
is piecewise syndetic in $\Z^2$?
\end{ques}

\medskip

\subsection{Dynamical motivation}\
\medskip

By a {\em topological dynamical system} (t.d.s. or system for short) we mean a pair $(X,T)$, where $X$ is a compact metric space (with a metric $\rho$) and $T:X \to X$ is a homeomorphism.
For $x\in X$, the {\it orbit of $x \in X$} is defined by $\O(x,T)=\{T^nx: n\in \Z\}$.
A t.d.s. $(X,T)$ is {\it transitive} if for any pair of non-empty open subsets $U$ and $V$ of $X$, there is $n\in\Z$ with $U\cap T^{-n}V\not=\emptyset$;
it is {\it weakly mixing} if 
$(X\times X,T\times T)$ is transitive; and
it is {\it minimal} if $\O(x,T)$ is dense in $X$ for every $x\in X$. A point $x \in X $
is	{\it minimal} or {\it uniformly recurrent} if the subsystem $(\overline{\O (x,T)},T)$ is minimal.
Recall that a subset $S$ of $\Z$ is {\it syndetic} if it has a bounded gap,
i.e. there is $N\in \N$ such that $\{i,i+1,\cdots,i+N\} \cap S \neq
\emptyset$ for every $i \in {\Z}$. $S$ is {\it thick} if it
contains arbitrarily long runs of integers. It is easy to verify that
a subset $S$ of $\Z$ is piecewise syndetic if and only if it is an
intersection of a syndetic set with a thick set. Let $(X,T)$ be a t.d.s. and $x\in X$. A classic result states that $x$ is a minimal
point if and only if $N_T(x,U)\triangleq\{n\in \Z: T^nx\in U\}$ is syndetic for any neighborhood
$U$ of $x$ (\cite[Chapter 4]{GH}).

\medskip

The Birkhoff recurrence theorem claims that any t.d.s. $(X,T)$ has a recurrent point $x$, that is, there is some increasing sequence $\{n_k\}_{k=1}^\infty$ of $\N$ such that $T^{n_k}x\to x,$ as $k\to \infty$.
Birkhoff recurrence theorem has the following generalization: for any $d\in \N$, there exist some $x\in X$ and some increasing sequence $\{n_k\}_{k=1}^\infty$ of $\N$ such that $T^{in_k}x\to x,$ as $ k\to \infty$
simultaneously for $i=1,\ldots,d$. We refer to this theorem as {\it the multiple Birkhoff recurrence theorem}, which is equivalent to van der Waerden's Theorem (\cite[Theorem 2.6]{F},
\cite[Theorem 1.4]{FW}). The multiple Birkhoff recurrence theorem can be deduced from the multiple recurrence theorem of Furstenberg \cite[Theorem 1.5]{F77} which was proved by using deep measure theoretic tools. It is Furstenberg and Weiss who presented a topologically dynamical proof  of the theorem (\cite[Theorem 1.4]{FW}). In the theorem,  such a point $x$ is called a {\em multiply recurrent point}, i.e., $x^{\otimes d}\triangleq (x,x,\ldots, x)\in X^d$ is a recurrent point of $\tau_d\triangleq T\times T^2\times \cdots \times T^d$.

\medskip

Furstenberg \cite{F77} proved Szemer\'edi's theorem via the multiple Poincar\'{e} recurrence theorem, which can be used to prove
 that for a t.d.s. $(X,T)$ and any $T$-invariant Borel probability measure $\mu$
on $X$, for $\mu$ almost all $x\in X$, every $d\in \N$ and
every neighborhood $U$ of $x$ the set $N_{\tau_d}(x^{\otimes d}, U^d)=\{n\in \Z: (\tau_d)^nx^{\otimes d}\in U^d=U\times \cdots \times U \}$
has positive upper density. Can we say more properties related to $N_{\tau_d}(x^{\otimes d}, U^d)$?
In \cite{F81} Furstenberg asked a question about multiply uniform recurrence:
Does there always exist a point $x$ such that $x^{\otimes d}$ is a minimal point for $\tau_d$, i.e. $N_{\tau_d}(x^{\otimes d}, U^d)$ is syndetic
for each neighborhood $U$ of $x$?
Recall that a t.d.s. $(X,T)$ is {\it distal} if for any $x\not=y\in X$, $\liminf_{n\in \Z}\rho(T^nx, T^ny)>0$.
For a distal t.d.s. $(X,T)$  and any $x\in X$, $x^{\otimes d}$ is not only a recurrent point of $\tau_d$,
but also a minimal point of $\tau_d$ (\cite[Proposition 5.11]{F81}). In \cite{HSY-21} Huang, Shao and Ye demonstrated that there exists a t.d.s. $(X,T)$ without any multiply
minimal points, which gives a negative answer to the question by Furstenberg.

\medskip

According to this negative answer, one can ask  a weaker question:
 Let $(X,T)$ be a t.d.s. and $d\in \N$.
Is there a point $x\in X$ such that $x^{\otimes d}$ is piecewise syndetic recurrent? That is, for any neighborhood
$U$ of $x$, is the set
$$N_{\tau_d}(x^{\otimes d}, U^d)=\{n\in \Z: (\tau_d)^nx^{\otimes d}\in U\times \cdots \times U \}$$
piecewise syndetic?

For weakly mixing minimal t.d.s. there do exist such points (see Remark \ref{rem-2.8} below). In general, 
we conjecture that it is not true. Nervelessness, we have the following result:

\medskip
{\noindent \bf Theorem} \cite{HSY-21}\quad {\em Let $(X,T)$ be a minimal system and $d\in \N$. Then for any non-empty open
subset $U$ of $X$, there is some $x\in U$ such that
$N_{\tau_d}(x^{\otimes d}, U^d)$ is piecewise syndetic.}

\medskip
Moreover, we think that this result is sharp, since it was shown in \cite{HSY-21} for a minimal PI system $(X,T)$ and $x\in X$, $(x,x)$ is minimal under $T\times T^2$ if and only if for each neighborhood $U$ of $x$, $\{n\in\Z: T^nx\in U, T^{2n}x\in U\}\in \F_{ps}$, and we
conjecture that there is a minimal PI system $(X,T)$ which has no point $x\in X$ such that $x$ is multiply minimal,
for more discussions on this aspect see \cite{HSY-21}.
According to the result obtained in \cite{HSY-21}, 
it is natural to ask the following question. 



\begin{ques}\label{ques2}
Let $(X,T)$ be a minimal t.d.s. and $d\in \N$.  Let $p_{i}$ be an integral polynomial with $p_{i}(0)=0$, $1\le i\le d$.
Is it true that for each non-empty open set $U$, there is a point $x\in U$ such that 
$$N_{\{p_1,\ldots, p_d\}}(x,U):=\{n\in\Z: T^{p_1(n)}x\in U, \ldots, T^{p_d(n)}x\in U\}$$
is piecewise syndetic?
\end{ques}

\subsection{Main results of the paper}\
\medskip

In this paper, we give affirmative answers to Questions \ref{q-bl} and \ref{ques2}. Particularly, we have a stronger result under
some mild conditions for Question \ref{ques2}. That is, we have


\medskip
\noindent{\bf Theorem A:}\quad
{\em
Let $d\in \N$ and $p_i$ be an integral polynomial with $p_i(0)=0$, $1\le i\le d$. If $S$ is piecewise syndetic in $\Z$, then $$\{(m,n)\in \Z^2: m+p_1(n),m+p_2(n),\ldots, m+p_d(n)\in S \}$$}
is piecewise syndetic in $\Z^2$.

\medskip
We remark that Theorem A deepens some result of Bergelson and Leibman \cite{BL96} which is equivalent to the following statement: under the same conditions as in Theorem A,
$\{(m,n)\in \Z^2: m+p_1(n),m+p_2(n),\ldots, m+p_d(n)\in S, n\not=0 \}\not=\emptyset$.

\medskip
Meanwhile, we  prove the following result:

\medskip
\noindent{\bf Theorem B:}\quad
{\em
Let $d\in \N$ and $p_i$ be an integral polynomial with $p_i(0)=0$, $1\le i\le d$. If $S\in\F_{ps}(\Z)$, then there is $A\in\F_{ps}(\Z)$ such that for any $N\in \N$, there is some ${a}_N\in \Z$ with 
$$A\cap [-N,N]\subseteq \{n\in\Z:  a_N+p_1(n), a_N+p_2(n), \ldots, a_N+p_d(n)\in S\}.$$}

In particular, we have that
$$\{n\in \Z: \exists \ m\in \Z \ \text{such that}\ m+p_1(n),m+p_2(n),\ldots, m+p_d(n)\in S\}$$
is piecewise syndetic in $\Z$, which is known \cite{BL96} \footnote{By \cite[Corollary 1.9]{BL96},
for any t.d.s. $(X,T)$, all integral polynomials $p_1,\ldots, p_d$ vanishing at zero,  the set $\{n\in Z: \forall \ep>0, \exists x\in X \ \text{such that }\ \forall i \in \{1,2,\ldots,d\}, \rho(T^{p_i(n)}x,x)<\ep\}$ is an IP$^*$ set. By this result, one can show that for a piecewise syndetic subset $S$ of $\Z$, the set $\{n\in \Z: \exists \ m\in \Z \ \text{such that}\ m+p_1(n),m+p_2(n),\ldots, m+p_d(n)\in S\}$ is an IP$^*$ subset and it is syndetic.}.

\medskip

Theorem A and Theorem B will be immediate consequences of dynamical results developed in the current
paper, i.e. they follow by the following Theorems E and D respectively. First we state Theorem C which shows stronger answer to Question \ref{ques2} can be obtained
under additional conditions.

\medskip
\noindent{\bf Theorem C:}\quad
{\em
Let $(X,T)$ be a minimal t.d.s. and $d\in \N$.  Let $p_{i}$ be an integral polynomial with $p_{i}(0)=0$, $1\le i\le d$. If one of the following conditions is satisfied,
\begin{itemize}
  \item $(X,T)$ is weakly mixing; 
  \item $(X,T)$ is distal;
  \item $\deg (p_i)\ge 2, 1\le i \le d$,
\end{itemize}
then there is a dense $G_\delta$ subset $X_0$ such that for each $x\in X_0$
and each neighbourhood $U$ of $x$
$$N_{\{p_1,\ldots,p_d\}}(x,U)=\{n\in\Z: T^{p_1(n)}x\in U, \ldots, T^{p_d(n)}x\in U\}$$
is piecewise syndetic.
}
\medskip

Generally, we have the following result which implies Theorem B.
We believe that it is the best possible one that people can expect,


\medskip
\noindent{\bf Theorem D:}\quad
{\em
Let $(X,T)$ be a minimal t.d.s. and $d\in \N$. Let $p_i$
be an integral polynomial with $p_i(0)=0$, $1\le i\le d$.  Then for any non-empty open subset $U$ of $X$, there is $x\in U$ such that
$$N_{\{p_1,\ldots,p_d\}}(x,U)=\{n\in\Z: T^{p_1(n)}x\in U, \ldots, T^{p_d(n)}x\in U\}$$
is piecewise syndetic.
}
\medskip

Finally, Theorem A is a corollary of the following Theorem E.

\medskip
\noindent{\bf Theorem E:}\quad
{\em
Let $(X,T)$ be a minimal t.d.s. and $d\in \N$. Let $p_i$
be an integral polynomial with  $p_i(0)=0$, $1\le i\le d$.
Then for each $x\in X$ and each neighbourhood $U$ of $x$
$$N^{\Z^2}_{\{p_1,\ldots,p_d\}}(x,U)=\{(m,n)\in\Z^2: T^{m+p_1(n)}x\in U, \ldots, T^{m+p_d(n)}x\in U\}$$
is piecewise syndetic in $\Z^2$.
}

\subsection{The main ideas of the proofs}\
\medskip

During the past several decades, one of the most important achievements in ergodic theory is that for a measure preserving system
$(X,\mathcal{X},\mu, T)$, each $f\in L^\infty (X, \mu)$ can be decomposed into three parts, i.e. $f=f_{\rm unif}+f_{\rm nil}+f_{\rm sml}$, where $f_{\rm nil}$ is the structured
one, $f_{\rm unif}$ is the random one and $f_{\rm sml}$ is the error one, see for example \cite[Theorem 16.10]{HK18}. The structured part $f_{\rm nil}$ depends on the algebraic
structure of the $d$-step pro-nilfactor of $(X, \mathcal{X},\mu, T)$.

\medskip

Searching for the corresponding result in topological dynamics (minimal systems) has a long history, starting from the early work
of Glasner in 1994 \cite{G94}. Finally, the problem was settled down in the recent work of Glasner, Huang, Shao, Weiss and Ye
\cite{GHSWY} by showing a saturation theorem for linear polynomials, which opens a door for possible study of many questions. In ergodic theory, when dealing with
questions related to polynomials, a powerful tool is the van der Corput lemma. Unfortunately, there is no similar tool in topological dynamics.
Thus, the saturation theorem for polynomials in topological dynamics can not be deduced from the linear one directly and one needs extra work to
complete it.  In the recent nice work by Qiu \cite{Qiu}, the author proved a weak version of the saturation theorem for polynomials, which allows him
to give a complete answer to a well known question concerning the density of  polynomial orbits in a totally minimal systems (for some previous
progress see \cite{GHSWY}). One of the main contributions and tools in the current work is 
Theorems \ref{thm-poly-sat1} and \ref{thm-poly-sat2} (strengthening the result of Qiu) which we will refer as the saturation theorems for polynomials.

\medskip
By Furstenberg's correspondence principle, we can associate each piecewise syndetic subset $S$ of $\Z$ with a minimal system $(X,T)$, whose dynamical properties are related to combinatorial properties of $S$. The other main contribution of the current
paper is that we further associate $(X,T)$ and given polynomials $\mathcal{A}=\{p_1,\ldots,p_d\}$ two dynamical systems:
one is $\Z$ action $(W_x^\A,\sigma)$, and the other one is $\Z^2$ action $(M_\infty(X,\A), \langle T^{\infty},\sigma\rangle)$
(see Section \ref{Section-Def-M}).
For example, when $\mathcal{A}=\{n^2\}$ and $x\in X$ we define
$$W_x^\A=\overline{\O}((T^{n^2}x)_{n\in\Z},\sigma)=\overline{\{(\ldots, T^{(n-1)^2}x, \underset{\bullet}T^{n^2}x, T^{(n+1)^2}x,\ldots):n\in\Z\}} \subset X^{\Z},$$
where $`` \ \underset{\bullet } \ "$ means the $0$-th coordinate, and $\sigma$ is the left shift;
and
\begin{equation*}
  \begin{split}
  M_\infty(X,\A)&=\overline{\O}((T^{n^2}x)_{n\in\Z}, \langle T^{\infty},\sigma\rangle)\\
  &= \overline{\{(\ldots, T^{m+(n-1)^2}x, \underset{\bullet}T^{m+n^2}x, T^{m+(n+1)^2}x,\ldots):n,m\in\Z\}}\subset X^{\Z},
  \end{split}
\end{equation*}
where $T^\infty=\cdots\times T\times T\times\cdots$, and $\langle T^{\infty},\sigma\rangle$
is the group generated by $T^\infty$ and $\sigma$.
We believe that such a consideration is new, which is not only very powerful in the current paper, but also provides a useful tool to study some other questions.
Note that the idea was inspired by some consideration carried out in \cite[Chapter 14]{HK18} for nilmanifolds and the shift $\sigma$. 
Moreover, in a companion paper
\cite{HSY-new}
by the same authors, we will construct measure preserving transformations associated with a given measure preserving transformation and
given polynomials, and derive some applications.

By showing the density of minimal points in $M_\infty(X,\A)$ (Theorem \ref{thm-poly-M-t.d.s}), we obtain our Theorem E (using Theorems \ref{Thm-equivalence1} and \ref{thm-poly-M-t.d.s}), and by lifting the property
from $\infty$-step pro-nilfactors, we show Theorem D. And then, Theorems A and B follow from Theorems E and D immediately.
We remark that the proof of Theorem D can be argued by the same method used in the proof of Theorem E, i.e. by showing that
if $\deg (p_i)\ge 2,\ 1\le i\le d$, then the minimal points are dense in $W_x^\A$ under $\sigma$, and then dealing further when there
are linear terms. Here we adopt another proof
to show how to use the saturation theorem for polynomials in a different way.

\medskip
Precisely, to show Theorem D we use the saturation theorem for polynomials and Leibman's result on the polynomial orbits on
nilmanifolds. When there are no linear terms in $\A$ (resp. $(X,T)$ is weakly mixing or distal), the situation is relatively
simple and we can get a stronger result (Theorem C). But
when there are linear terms in $\A$, things are much involved and we need some extra discussion, see the Claim in the proof
of Theorem C.

\medskip
Unfortunately, we do not know how to show Theorem E by using the idea of the proof of Theorem D.
To illustrate the idea and the difficult we face in the proof of Theorem E,
we first give the proof when $(X,T)$ is distal (see Theorem \ref{thm-distal}). From the proof of the distal case, we realize that 
we need suitable elements in ${\rm Aut} (X,T)$, the group of automorphisms of $(X,T)$.
To find the elements that we need, we pass our discussion to the universal minimal flow to find rich elements in ${\rm Aut} (X,T)$ for our purpose,
i.e. we use the regularizer from the theory of minimal flows 
to prove our theorem. Since our approach to show Theorem E uses heavy tools, it will be nice if one finds some elementary arguments to prove it.

\medskip
To finish the subsection we state the following corollary of Theorem E.
Note that a t.d.s. is an $M$-system if it is transitive and the set of minimal points is dense.

\medskip
\noindent{\bf Corollary F}: {\it Let $(X,T)$ be a t.d.s.  and $d\in \N$. Let $\A=\{p_1,\ldots, p_d\}$
be an family of integral polynomials with $p_i(0)=0$, $1\le i\le d$. Then

\begin{enumerate}
\item $(X,T)$ has dense minimal points if and only if so  does $(M_\infty(X,\A), \langle T^{\infty},\sigma\rangle)$.

\item  $(X,T)$ is an $M$-system if and only if so is $(M_\infty(X,\A), \langle T^{\infty},\sigma\rangle)$.
\end{enumerate}
}
\medskip

We remark that generally the minimality of $(X,T)$ does not imply the minimality of $(M_\infty(X,\A), \langle T^{\infty},\sigma\rangle)$, see
Section \ref{Section-Def-M}.

\subsection{Organization of the paper}

We organize the paper as follows. In Section \ref{Section-pre} we state some necessary notions and some known facts used in the paper. In Section \ref{section-saturate}, we give some saturation theorems for polynomials, which provide important tools in the paper. In Section \ref{Section-Def-M}, we define systems $N_\infty(X,\A)$ and $M_\infty(X,\A)$ induced by a t.d.s. $(X,T)$ and a family of integral polynomials $\A$, whose dynamical properties play very important roles in our study. After the preparation in the previous sections, in Sections \ref{Section-nil}, \ref{Section-Z-poly-rec} and \ref{Section-Z2-recurrence},
we prove Theorems C, D, E and Corollary F. In Section \ref{Section-combinatorial}, we deduce the combinatorial consequences and thus prove Theorems A and B. In the final section we state some open questions.

\medskip
We thank Dr. Jiahao Qiu and Hui Xu for careful readings and corrections.


\section{Preliminaries}\label{Section-pre}

In this section we give some necessary notions and some known facts used in the paper.
Note that, instead of just considering a single transformation $T$,
we will consider commuting
transformations $T_1$, $\ldots$ , $T_k$ of $X$. We only recall some basic
definitions and properties of systems for
one transformation. Extensions to the general case are
straightforward.

\subsection{Topological dynamical systems}\

\subsubsection{}
By a {\em topological dynamical system} (t.d.s.) we mean a pair $(X,T)$, where
$X$ is a compact Hausdorff space (in the paper we mainly deal with a compact metric space $X$ with a metric $\rho$ except in
Section \ref{Section-Z2-recurrence}) and $T:X\to X$ is a
homeomorphism. Let $(X, T)$ be a t.d.s. and $x\in X$. Then $\O(x,T)=\{T^nx: n\in \Z\}$ denotes the
{\em orbit} of $x$. A subset $A\subseteq X$ is called {\em invariant} (or {$T$-invariant}) if $TA= A$. When $Y\subseteq X$ is a closed and
invariant subset of the system $(X, T)$, we say that the system
$(Y, T|_Y)$ is a {\em subsystem} of $(X, T)$. Usually we will omit the subscript, and denote $(Y, T|_Y)$ by $(Y,T)$.
If $(X, T)$ and $(Y, S)$ are two t.d.s., their {\em product system} is the
system $(X \times Y, T\times S)$.

\medskip

When there are more than one t.d.s. involved, usually we should use different symbols to denote different transformations on different spaces, for example, $(X,T), (Y,S), (Z,H)$ etc. But when no confusing, it is convenient to use only one symbol $T$ for all transformations in all t.d.s. involved,  for example, $(X,T), (Y,T), (Z,T)$ etc.
In this paper, we {\em  use the same symbol $T$ for the transformations in all t.d.s.}

\subsubsection{}
Let $X, Y$ be compact metric spaces and $\phi: X \to Y$ be a map.  For $n \geq 2$ let
\begin{equation}\label{ite-n}
\phi^{(n)}=\underbrace{\phi\times \cdots \times \phi }_{\text{($n$ times)}}: X^n\rightarrow Y^n.
\end{equation} Thus
we write $(X^n,T^{(n)})$ for the $n$-fold product system $(X\times	\cdots \times X,T\times \cdots \times T)$.
The diagonal of $X^n$ is $$\Delta_n(X)=\{(x,\ldots,x)\in X^n: x\in X\}.$$
When $n=2$ we write	$\Delta(X)=\Delta_2(X)$.

\subsubsection{}
A t.d.s. $(X,T)$ is called {\em minimal} if $X$ contains no proper non-empty
closed invariant subsets. It is easy to verify that a t.d.s. is
minimal if and only if every orbit is dense. In a t.d.s. $(X, T)$ we say that a point $x\in X$ is {\em minimal} if $(\overline{\O(x,T)}, T)$ is minimal.
When $X$ is a metric compact space, $(X,T)$ is transitive if and only if there is some point $x\in X$ such that $\overline{\O(x,T)}=X$, which is called a {\em transitive point}.

\subsubsection{}
Let $(X,T)$ be a t.d.s. A pair $(x,y)\in X^2$ is {\it proximal} if $\inf_{n\in \Z} \rho(T^nx,T^ny)=0$; and it is {\it distal} if it is not proximal. Denote by ${\bf P}(X,T)$ the set of all proximal pairs of $(X,T)$.
A t.d.s. $(X,T)$ is
called {\it distal} if $(x,x')$ is distal whenever $x,x'\in X$ are
distinct.


\subsubsection{}
For a t.d.s. $(X,T)$, $x\in X$ and $U\subseteq X$ let
$$N_T (x,U)=\{n\in \Z: T^nx\in U\}.$$
A point $x\in X$ is said to be {\em recurrent} if for every
neighborhood $U$ of $x$, $N_T (x,U)$ is infinite. Equivalently, $x\in
X$ is recurrent if and only if there is a sequence $n_i\to\infty$ such that
$T^{n_i}x\to x,$ as $ i\to\infty$. 

\subsubsection{}
Generally, a {\em $G$-system} is a triple $\X=(X, G, \Pi)$ or just $(X,G)$, where $X$ is a compact Hausdorff space, $G$ is a Hausdorff topological group with the unit $e$ and $\Pi: G\times X\rightarrow X$ is a continuous map such that $\Pi(e,x)=x$ and $\Pi(s,\Pi(t,x))=\Pi(st,x)$ for all $s,t\in G$ and $x\in X$. We shall fix $G$ and suppress the action symbol. 
An analogous definition can be given if  $G$ is a semigroup. Also, the notions of transitivity, minimality and weak mixing are naturally generalized to group actions.

\subsubsection{}
The following lemma is easy to be verified, and one can find a proof in \cite[Chapter 4]{GH}
or \cite[Chapter IV, (1.2)]{Vr}.

\begin{lem}\label{lem-minimal-point}
Let $(X,G)$ be a $G$-system and let $H$ be a syndetic \footnote{A subset $H$ of  a group $G$ is {\em syndetic} if there exists a compact subset $K$ of $G$ such that $G=KH$.} normal subgroup of $G$. Then a point $x$ is $G$-minimal if and only if it is $H$-minimal.
\end{lem}

\subsection{Factor maps}\

\subsubsection{}
A {\em factor map} $\pi: X\rightarrow Y$ between two t.d.s. $(X,T)$
and $(Y, T)$ is a continuous onto map which intertwines the
actions (i.e. $\pi\circ T= T\circ \pi$); one says that $(Y, T)$ is a {\it factor} of $(X,T)$ and
that $(X,T)$ is an {\it extension} of $(Y,T)$. The systems are said to be {\em isomorphic} if $\pi$ is bijective.

\subsubsection{}
The proof of the following lemma is easy.

\begin{lem}\label{den-minimal}
We have the following
\begin{enumerate}
\item Let $(X,T)$ be minimal and $k\in\N$. Then minimal points of $T^{(k)}$ is dense in $X^k$.

\item Let $(X, \langle T,S\rangle)$ be a minimal system with $T\circ S=S\circ T$, where $T,S:X\rightarrow X$ are homeomorphisms. Then
the set of minimal points of $T$ (resp.  of $S$) is dense in $X$.

\item Let $\pi: (X, G)\rightarrow (Y, G)$ be a factor map. If $y\in Y$ is minimal, then there is
a minimal point $x\in \pi^{-1}(y)$.
\end{enumerate}
\end{lem}
\begin{proof} (1) and (3) are standard facts. It remains to show (2). Let $x$ be a minimal point of $T$ and $U$ be a non-empty open
set of $X$. By the minimality of $(X, \langle T,S\rangle)$, there are $n_0,m_0\in\Z$ with $y=T^{n_0}S^{m_0}x\in U$. We claim that $y$ is a
minimal point of $T$. To see this, we show that $N_T(y,V)$ is syndetic for each neighbourhood $V$ of $y$. Let $V$ be a neighborhood of $y$. We have $x\in T^{-n_0}S^{-m_0}V$, which implies that $\{n\in\Z:T^nx\in T^{-n_0}S^{-m_0}V\}$
is syndetic, and so is $\{n\in\Z:T^ny\in V\}$. This shows that the minimal points of $T$ is dense in $X$. Similarly, the minimal points of $S$ is dense in $X$ too.
\end{proof}

\subsubsection{}
Let $\pi: (X,T)\rightarrow (Y, T)$ be a factor map. Then
$$R_\pi=\{(x_1,x_2):\pi(x_1)=\pi(x_2)\}$$
is a closed invariant equivalence relation, and $Y=X/ R_\pi$.

Let $(X,T)$ and $(Y,T)$ be t.d.s. and let $\pi: (X,T) \to (Y,T)$ be a factor map.
One says that:
\begin{itemize}
  \item $\pi$ is an {\it open} extension if it sends open sets to open sets; 
  \item $\pi$ is a {\it proximal} extension if
$\pi(x_1)=\pi(x_2)$ implies $(x_1,x_2) \in {\bf P} (X,T)$;
  \item $\pi$ is  a {\it distal} extension if $\pi(x_1)=\pi(x_2)$ and $x_1\neq x_2$ implies $(x_1,x_2) \not\in {\bf P} (X,T)$;
  \item $\pi$ is an {\it almost one to one} extension  if there exists a dense $G_\d$ set $X_0\subseteq X$ such that $\pi^{-1}(\{\pi(x)\})=\{x\}$ for any $x\in X_0$;
\item $\pi$ is an {\it equicontinuous or isometric} extension if for any $\ep >0$ there exists $\d>0$
such that $\pi(x_1)=\pi(x_2)$ and $\rho(x_1,x_2)<\d$ imply $\rho(T^n x_1,T^n x_2)<\ep$ for any $n\in \Z$.

\end{itemize}

One can use the following
construction to modify a factor map to be open via almost one to one maps, which is due originally to Veech
(see \cite[Theorem 3.1]{Veech}).

\begin{thm}\label{thm-O}
Given a factor map $\pi:X\rightarrow Y$ between minimal t.d.s.
$(X,T)$ and $(Y,T)$, there exists a commutative diagram of factor
maps (called {\em O-diagram})
\[
\begin{CD}
X @<{\varsigma^*}<< X^*\\
@V{\pi}VV      @VV{\pi^*}V\\
Y @<{\varsigma}<< Y^*
\end{CD}
\]
such that

\noindent (a) $\varsigma$ and $\varsigma^*$ are almost one to one
extensions;

\noindent (b) $\pi^*$ is an open extension;

\noindent (c) $X^*$ is the unique minimal set in $R_{\pi
\varsigma}=\{(x,y)\in X\times Y^*: \pi(x)=\varsigma (y)\}$ and
$\varsigma^*$ and $\pi^*$ are the restrictions to $X^*$ of the
projections of $X\times Y^*$ onto $X$ and $Y^*$ respectively.
\end{thm}

We note that when $\pi$ is open, we have $X^*=X,\ Y^*=Y$ and $\pi^*=\pi$.


\subsection{Furstenberg families}\
\medskip

Let us recall some notions related to Furstenberg families (for
details see \cite[Chapter 2]{Ak97} or \cite[Chapter 9]{F}). We only discuss subsets of $\Z$ here, and it is similar for subsets of $\Z^d$.

\subsubsection{}
Let $\P=\P({\Z})$ be the collection
of all non-empty subsets of $\Z$. A subset $\F$ of $\P$ is a {\em
(Furstenberg) family}, if it is hereditary upwards, i.e., $F_1
\subseteq F_2$ and $F_1 \in \F$ imply $F_2 \in \F$. A family $\F$ is
{\it proper} if it is a proper subset of $\P$, i.e. neither empty
nor all of $\P$. If a proper family
$\F$ is closed under finite intersections, then $\F$ is called a {\it
filter}. For a family $\F$, the {\it dual family} is
$$\F^*=\{F\in\P: {\Z} \setminus F\notin\F\}=\{F\in \P:F \cap F' \neq
\emptyset \ for \ all \ F' \in \F \}.$$
Denote by $\F_{inf}$ the family consisting of all infinite subsets of $\Z$.
The
collection of all syndetic (resp. thick) subsets of $\Z$ is denoted by
$\F_s$ (resp. $\F_t$). Note that $\F_s^*=\F_t$ and $\F_t^*=\F_s$. Denote the set of
all piecewise syndetic sets by $\F_{ps}$. A set $F$ is called {\em thickly syndetic} if for every $N\in \N$ the positions where intervals with length $N$ runs begin form
a syndetic set. Denote the set of
all thickly syndetic sets by $\F_{ts}$, and we have $\F_{ts}$ is a filter and $\F^*_{ps}=\F_{ts}, \F^*_{ts}=\F_{ps}$.

\subsubsection{}
Let $\F$ be a family and $(X,T)$ be a t.d.s. We say $x\in X$ is
$\F$-{\it recurrent} if for each neighborhood $U$ of $x$, $N(x,U)\in
\F$. So the usual recurrent point is just $\F_{inf}$-recurrent one.
The following proposition characterizes piecewise syndetic recurrence (see for example \cite[Lemma 2.1]{HY05} or \cite[Theorem 3.1]{Ak97}).

\begin{prop}\label{prop-M}
Let $(X,T)$ be a transitive t.d.s. Then the following are equivalent:
\begin{enumerate}
  \item $(X,T)$ is an $M$-{system}, i.e., it is transitive and the minimal points are dense in $X$;
  \item each transitive point $x$ is $\F_{ps}$-recurrent; 
  \item there is a transitive point $x$ that is $\F_{ps}$-recurrent.
\end{enumerate}
\end{prop}

\begin{rem}
Proposition \ref{prop-M} still holds for t.d.s. under $\Z^k$-actions.
Let $T_1,T_2,\ldots, T_k: X\rightarrow X$ be commuting homeomorphisms ($k\in \N$). Let $G=\langle T_1,T_2,\ldots, T_d\rangle$ be the group generated by $T_1,T_2,\ldots, T_k$. Then $(X,G)$ is an $M$-system if and only if for each (some) transitive point $x\in X$ and any neighbourhood $U$ of $x$, $$N_G(x,U)=\{(n_1,n_2,\ldots, n_k)\in \Z^k: T_1^{n_1}\cdots T_k^{n_k}x\in U\}$$ is a piecewise syndetic subset of $\Z^k$.
For a proof, we refer to \cite[Theorem 3.1]{Ak97}.
\end{rem}


\subsection{Some facts about hyperspaces} \label{sub:ellis} \

\subsubsection{}
Let $X$ be a compact metric space. Let $2^X$ be the collection of nonempty closed subsets of $X$.
Let $\rho$ be the metric on $X$,
then one may define a metric on $2^X$ as follows:
\begin{equation*}
\begin{split}
 \rho_H(A,B)& = \inf \{\ep>0: A\subseteq B_\ep(B), B\subseteq B_\ep(A)\}\\
 &= \max \{\max_{a\in A} \rho(a,B),\max_{b\in B} \rho(b,A)\},
\end{split}
\end{equation*}
where $\rho(x,A)=\inf_{y\in A} \rho(x,y)$ and $B_\ep (A)=\{x\in X: \rho(x, A)<\ep\}$.
The metric $\rho_H$ is called the {\em Hausdorff metric} of $2^X$.

Let $\{A_i\}_{i=1}^\infty$ be an arbitrary sequence of subsets of $X$. Define
$$\liminf A_i=\{x\in X: \text{for any neighbourhood $U$ of $x$, $U\cap A_i\neq \emptyset$ for all but finitely many $i$}\};$$
$$\limsup A_i=\{x\in X: \text{for any neighbourhood $U$ of $x$, $U\cap A_i\neq \emptyset$ for infinitely many $i$}\}.$$
We say that $\{A_i\}_{i=1}^\infty$ converges to $A$, denoted by $\lim_{i\to \infty} A_i=A$, if
$$\liminf A_i=\limsup A_i=A.$$
Now let $\{A_i\}_{i=1}^\infty\subseteq 2^X$ and $A\in 2^X$. Then $\lim_{i\to\infty} A_i=A$ if and only if $\{A_i\}_{i=1}^\infty $ converges to $A$ in $2^X$ with respect to the Hausdorff metric.

\subsubsection{}
Let $X,Y$ be two compact metric spaces. Let $F: Y\rightarrow 2^X$ be a map and $y\in Y$.
We say that $F$ is {\em upper semi-continuous (u.s.c.)} at $y$ if whenever $\lim y_i=y$, one has that $\limsup F(y_i)\subseteq F(y)$. If $F$ is u.s.c. at every point of $Y$, then we say that $F$ is u.s.c.
We say $F$ is {\em lower semi-continuous (l.s.c.)} at $y$ if whenever $\lim y_i=y$, one has that $\liminf F(y_i)\supset F(y)$. If $F$ is l.s.c. at every point of $Y$, then we say that $F$ is l.s.c.


It is easy to verify that $F: Y\rightarrow 2^X$ is u.s.c. at $y\in Y$ if and only if for each $\ep>0$ there exists a neighbourhood $U$ of $y$ such that $F(U)\subseteq B_\ep(F(y))$; and $F: Y\rightarrow 2^X$ is l.s.c. at $y\in Y$ if and only if for each $\ep>0$ there exists a neighbourhood $U$ of $y$ such that $F(y)\subseteq B_\ep(F(y'))$ for all $y'\in U$.

If $f: X\rightarrow Y$ is a continuous map, then it is easy to verify that
$$F=f^{-1}: Y\rightarrow 2^X, y\mapsto f^{-1}(y)$$ is u.s.c. Let $(X,T)$ be t.d.s. Then the map $$F: X\rightarrow 2^X, x\mapsto \overline{\O(x,T)}$$ is l.s.c. (see Subsection \ref{subsec-lsc}).

\medskip

We have the following well known result, for a proof see \cite[p.70-71]{Kura2} and \cite[p.394]{Kura1}, or \cite{Fort}.
\begin{thm}\label{thm-Fort}
Let $X,Y$ be compact metric spaces. If $F: Y\rightarrow 2^X$ is u.s.c. (or l.s.c.),
then the points of continuity of $F$ form a dense $G_\delta$ set in $Y$.
\end{thm}

\subsection{$N_d(X)$ and Glasner's theorem}\
\medskip

Let $x\in X$. 
For convenience, sometimes we denote the orbit closure of $x\in X$ under $T$ by  $\overline{\O}(x,T)$ instead of $\overline{\O(x,T)}$. Let $A\subseteq X$, the  orbit of $A$ is defined by
$\O(A,T)=\{T^nx: x\in A,n\in \Z\}$, and its closure is denoted by $\overline{\O}(A,T)= \overline{\O(A,T)}$.

\medskip

Let $(X,T)$ be a t.d.s., and $d\in \N$. Recall that $\tau_d=T\times T^2 \times \cdots \times T^{d},$
$\Delta_d(X)=\{(x,x,\ldots,x): x\in X\}\subseteq X^d,$ and $T^{(d)}=T\times \ldots\times T \ (d \ \text{times})$.
Let
$$N_d(X)=\overline{\O}(\D_d(X), \tau_d).$$
If $(X,T)$ is transitive and $x\in X$ is a transitive point, then
$N_d({X})=\overline{\O}(x^{\otimes d},
\langle\tau_d, T^{(d)}\rangle),$
is the orbit closure of
$x^{\otimes d}=(x,\ldots,x)$ ($d$ times) under the action of the group
$\langle\tau_d, T^{(d)}\rangle$ generated by $\tau_d$ and $T^{(d)}$.

\medskip

The following result was obtained by Glasner.

\begin{prop} \cite[Theorem 5.1, Corollary 2.5]{G94}\label{thm-Glasner}
Let $(X,T)$ be a minimal t.d.s. Then
\begin{enumerate}
\item $(N_d(X), \langle\tau_d, \sigma_d\rangle)$ is minimal for each $d\in\N$.

\item If in addition $(X,T)$ is weakly mixing, then there is a dense $G_\delta$ set $X_0$ of $X$ such that
for each $x\in X_0$ and each $d\in\N$, $\overline{\O}(x^{\otimes d},\tau_d)=X^d$.
\end{enumerate}
\end{prop}

\begin{rem}\label{rem-2.8}
By Proposition \ref{thm-Glasner}-(2), if $(X,T)$ is a weakly mixing minimal t.d.s., then there is a dense $G_\d$ subset $X_0$ of $X$ such that for all $x\in X_0$, $x^{\otimes d}$ is a transitive point of $(X^d, \tau_d)$. Note that by Proposition \ref{thm-Glasner}-(1) and Lemma \ref{den-minimal}, the set of $\tau_d$-minimal points is dense in $(X^d, \tau_d)$. Thus $(X^d,\tau_d)$ is an $M$-system. By Proposition \ref{prop-M}, for any neighbourhood $U$ of $x\in X_0$, $N_{\tau_d}(x^{\otimes d},U^d)=\{n\in \Z: \tau_d x\in U^d=U\times U\times \cdots \times U\}$ is piecewise syndetic.
\end{rem}

\subsection{Nilmanifolds and nilsystems}\

\subsubsection{}
Let $G$ be a group. For $g, h\in G$ and $A,B \subseteq G$, we write $[g, h] =
ghg^{-1}h^{-1}$ for the commutator of $g$ and $h$ and
$[A,B]$ for the subgroup spanned by $\{[a, b] : a \in A, b\in B\}$.
The commutator subgroups $G_j$, $j\ge 1$, are defined inductively by
setting $G_1 = G$ and $G_{j+1} = [G_j ,G]$. Let $d \ge 1$ be an
integer. We say that $G$ is {\em $d$-step nilpotent} if $G_{d+1}$ is
the trivial subgroup.

\subsubsection{}
Let $d\in\N$, $G$ be a $d$-step nilpotent Lie group and $\Gamma$ be a discrete
cocompact subgroup of $G$. The compact manifold $X = G/\Gamma$ is
called a {\em $d$-step nilmanifold}. The group $G$ acts on $X$ by
left translations and we write this action as $(g, x)\mapsto gx$.
The Haar measure $\mu$ of $X$ is the unique Borel probability measure on
$X$ invariant under this action. Fix $t\in G$ and let $T$ be the
transformation $x\mapsto t x$ of $X$, i.e. $t(g\Gamma)=(tg)\Gamma$ for each $g\in G$. Then $(X, \mu, T)$ is
called a {\em $d$-step nilsystem}. In the topological setting we omit the measure
and just say that $(X,T)$ is a $d$-step nilsystem. For more details on nilsystems, refer to \cite[Chapter 11]{HK18}.
Here are some basic properties of nilsystems.

\begin{thm}\cite{Leibman051, P}\label{thm-ParryLeibman}
Let $(X = G/\Gamma,\mu , T )$ be a $d$-step nilsystem with $T$ the
translation by the element $t\in G$. Then:

\begin{enumerate}
\item $(X, T )$ is uniquely ergodic if and only if $(X,\mu , T )$ is
ergodic if and only if $(X, T )$ is minimal if and only if $(X, T )$
is transitive.

\item Let $Y$ be the closed orbit of some point $x\in X$. Then $Y$ can
be given the structure of a nilmanifold, i.e. $Y = H/\Lambda$, where $H$
is a closed subgroup of $G$ containing $t$ and $\Lambda$ is a closed
cocompact subgroup of $H$.

\end{enumerate}
\end{thm}

One can generalize the above results to the action of several translations. For example,
let $X= G/\Gamma$ be a nilmanifold with the Haar measure $\mu$ and let
$t_1,\cdots , t_k$ be commuting elements of $G$. If the group
spanned by the translations $t_1, \cdots , t_k$ acts ergodically on
$(X,\mu)$, then $X$ is uniquely ergodic for this group. For more details, refer to \cite{Leibman051, Leibman052, HK18}.

\subsubsection{}
We will need to use inverse limits of nilsystems, so we recall the
definition of a sequential inverse limit of systems. If
$(X_i,T_i)_{i\in \N}$ are systems 
and $\pi_i: X_{i+1}\rightarrow X_i$ is a factor map for all $i\in\N$, the {\em inverse
limit} of these systems is defined to be the compact subset of
$\prod_{i\in \N}X_i$ given by $\{ (x_i)_{i\in \N }: \pi_i(x_{i+1}) =
x_i, i\in\N\}$, and we denote it by
$\lim\limits_{\longleftarrow}(X_i,T_i)_{i\in\N}$.
It is a compact metric space.
We note that the maps $\{T_i\}_{i\in \N}$ induce naturally a transformation $T$ on the inverse
limit such that $T(x_1,x_2,\ldots)=(T_1(x_1),T_2(x_2),\ldots)$.




\begin{de} \cite[Definition 1.2]{HKM} 
For $d\in \N$, a minimal t.d.s. $(X,T)$ is called  a {\em $d$-step pro-nilsystem} or {\it system of order $d$} if $X$ is an inverse limit of $d$-step minimal nilsystems.
\end{de}

\subsection{Regionally proximal relation
of order $d$}\
\medskip

\begin{de}\cite[Definition 3.2]{HKM} 
Let $(X, T)$ be a t.d.s. and let $d\in \N$. The points $x, y \in X$ are
said to be {\em regionally proximal of order $d$} if for any $\d  >
0$, there exist $x', y'\in X$ and a vector ${\bf n} = (n_1,\ldots ,
n_d)\in\Z^d$ such that $\rho (x, x') < \d, \rho (y, y') <\d$, and $$
\rho (T^{{\bf n}\cdot \ep}x', T^{{\bf n}\cdot \ep}y') < \d\
\text{for every  $\ep=(\ep_1,\ldots,\ep_d)\in \{0,1\}^d\setminus\{00\cdots 0\}$},$$
where ${\bf n}\cdot \ep=\sum_{i=1}^d n_i\ep_i$.
The set of regionally proximal pairs of
order $d$ is denoted by $\RP^{[d]}$ (or by $\RP^{[d]}(X,T)$ in case of
ambiguity), and is called {\em the regionally proximal relation of
order $d$}.
\end{de}

It is easy to see that $\RP^{[d]}$ is a closed and invariant
relation, and 
\begin{equation*}
    {\bf P}(X,T)\subseteq  \ldots \subseteq \RP^{[d+1]}\subseteq
    \RP^{[d]}\subseteq \ldots \RP^{[2]}\subseteq \RP^{[1]}.
\end{equation*}

The following theorems were proved in \cite{HKM} (for minimal distal systems) and
in \cite{SY} (for general minimal systems).

\begin{thm}\label{thm-1}
Let $(X, T)$ be a minimal t.d.s. and let $d\in \N$. Then
\begin{enumerate}

\item $\RP^{[d]}$ is an equivalence relation.

\item $(X,T)$ is a $d$-step pro-nilsystem if and only if $\RP^{[d]}=\Delta_X$.
\end{enumerate}
\end{thm}

The regionally proximal relation of order $d$ allows us to construct the maximal $d$-step
pro-nilfactor of a system.

\begin{thm}\label{thm0}
Let $\pi: (X,T)\rightarrow (Y,T)$ be a factor map between minimal t.d.s.
and let $d\in \N$. Then,
\begin{enumerate}
  \item $\pi\times \pi (\RP^{[d]}(X,T))=\RP^{[d]}(Y,T)$.
  \item $(Y,T)$ is a $d$-step pro-nilsystem if and only if $\RP^{[d]}(X,T)\subseteq R_\pi$.
\end{enumerate}
In particular, $X_d\triangleq X/\RP^{[d]}(X,T)$, the quotient of $(X,T)$ under $\RP^{[d]}(X,T)$, is the
maximal $d$-step pro-nilfactor of $X$. 
\end{thm}

\subsection{$\infty$-step pro-nilsystems}\
\medskip

By Theorem \ref{thm-1} for any minimal t.d.s. $(X,T)$,
$$\RP^{[\infty]}=\bigcap\limits_{d\ge 1} \RP^{[d]}$$
is a closed invariant equivalence relation (we write $\RP^{[\infty]}(X,T)$ in case of ambiguity). Now we formulate the
definition of $\infty$-step pro-nilsystems. 

\begin{de}\cite[Definition 3.4]{D-Y}
A minimal t.d.s. $(X, T)$ is an {\em $\infty$-step
pro-nilsystem} or {\em a system of order $\infty$}, if the equivalence
relation $\RP^{[\infty]}$ is trivial, i.e., coincides with the
diagonal.
\end{de}

\begin{rem}
\begin{enumerate}
  \item Similar to Theorem \ref{thm0}, one can show that the quotient of a
minimal system $(X,T)$ under $\RP^{[\infty]}$ is the maximal
$\infty$-step pro-nilfactor of $(X,T)$. We denote the maximal
$\infty$-step pro-nilfactor of $(X,T)$ by $X_\infty$.

  \item A minimal system is an $\infty$-step pro-nilsystem if and only if it is
an inverse limit of minimal nilsystems \cite[Theorem 3.6]{D-Y}.

\end{enumerate}
\end{rem}

\section{Saturation theorems for polynomials}\label{section-saturate}

In this section, we generalize the saturation theorems in \cite{GHSWY} and \cite{Qiu} to polynomial version,
which will be used in the proofs of Theorems C, D and E.

Let $X, Y$ be sets, and let $\phi : X\rightarrow Y$ be a map. A subset $L$ of $X$ is called
{\em $\phi$-saturated} if $$\{x\in L: \phi^{-1}(\phi(x))\subseteq L\}=L,$$ i.e., $L=\phi^{-1}(\phi(L))$.

\subsection{Topological characteristic factors }\
\medskip

Given a factor map $\pi: (X,T)\rightarrow (Y,T)$ and $d\ge 2$,
the t.d.s. $(Y,T)$ is said to be a {\em $d$-step topological
characteristic factor
(for $\tau_d=T\times T^2\times \cdots \times T^d$)
of  $(X,T)$}, if there exists a dense $G_\d$
subset $\Omega$ of $X$ such that for each $x\in \Omega$ the orbit
closure $L_x=\overline{\O}((x, \ldots,x), \tau_d)$ is $\pi^{(d)}=\pi\times \cdots \times
\pi$ ($d$-times) saturated. That is, $(x_1,x_2,\ldots, x_d)\in L_x$
if and only if $(x_1',x_2',\ldots, x_d')\in L_x$, where $\pi(x_i)=\pi(x_i')$ for all $i\in \{1,2,\ldots, d\}$.
The following theorem was proved in \cite{GHSWY}.
\begin{thm}\cite[Theorem A]{GHSWY}\label{thm-GHSWY}
Let $(X,T)$ be a minimal t.d.s., and $\pi:X\rightarrow X_\infty$ be the factor map from $X$ to its maximal $\infty$-step pro-nilfactor $X_\infty$. Then there are minimal t.d.s. $X^*$ and $X_\infty^*$ which are almost one to one
extensions of $X$ and $X_\infty$ respectively, and a commuting diagram below such that $\pi^*: X^*\rightarrow X^*_\infty$ is open, and $X_\infty^*$ is a $d$-step topological characteristic factor of $X^*$ for all $d\ge 2$.
\[
\begin{CD}
X @<{\varsigma^*}<< X^*\\
@VV{\pi}V      @VV{\pi^*}V\\
X_\infty @<{\varsigma}<< X_\infty^*
\end{CD}
\]
\end{thm}

We remark that the almost one to one modifications in the above theorem are needed, see for example \cite{G94, WXY}. Based on Theorem \ref{thm-GHSWY}, using PET induction and elaborate constructions, it was shown in \cite{Qiu}

\begin{thm}\cite[Theorem B]{Qiu}\label{Thm-Qiu}
Let notation be as in Theorem \ref{thm-GHSWY} and $d\in \N$. Then for any non-empty open subsets $V_0,V_1,\ldots, V_d$ of $X^*$ with $\bigcap _{i=0}^d \pi^*(V_i)\neq \emptyset$ and essentially distinct non-constant integral polynomials $p_1,p_2,\ldots, p_d$ with $p_{i}(0)=0$, $i=1,2,\ldots, d$, there is some $n\in \N$ such that
$$V_0\cap T^{-p_1(n)}V_1\cap T^{-p_2(n)}V_2\cap \ldots \cap T^{-p_d(n)}V_d\neq \emptyset.$$
\end{thm}

\begin{rem}\label{rem-3.3}
In \cite{GHSWY}, to get Theorem \ref{thm-GHSWY}, the authors proved the following result: Let $\pi:(X,T)\rightarrow (Y,T)$ be an extension of minimal t.d.s.
Assume that $\pi$ is open and $X_{\infty}$ is a factor of $Y$.
$$
\xymatrix{
                &         X \ar[d]^{\pi} \ar[dl]_{\pi_\infty}    \\
  X_{\infty} & Y   \ar[l]_{\phi}          }
$$
Then $Y$ is a
$d$-step topological characteristic factor of $X$ for all $d\ge 2$, that is, for all $d\in \N$ there exists a dense $G_\d$ subset $\Omega_d$ of $X$ such that for each $x\in \Omega_d$ the orbit closure $L_x=\overline{\O}(x^{(d)}, \tau_{d})$ is $\pi^{(d)}$-saturated \cite[Theorem 4.2]{GHSWY}. Using this result and O-diagram (Theorem \ref{thm-O}), one has Theorem \ref{thm-GHSWY}. Thus in Theorem \ref{thm-GHSWY}, when $\pi: X\rightarrow X_\infty$ is open, then $X^*=X$ and $X^*_\infty=X_\infty$. The same thing happens in Theorem \ref{Thm-Qiu}.

\end{rem}

\subsection{A l.s.c. map}\label{subsec-lsc}\
\medskip

Recall that for metric spaces $X$ and $Y$, the map $F: Y\rightarrow 2^X$ is l.s.c. at $y\in Y$ if and only if for each $\ep>0$ there exists a neighbourhood $U$ of $y$ with $F(y)\subseteq B_\ep(F(y'))$ for all $y'\in U$.

Let $(X,T)$ be a t.d.s. and $d\in \N$. Let $\A=\{p_1,\ldots, p_d\}$ be a set of integral polynomials and define
$$\O_\A((x_1,\ldots,x_d))=\O_{\{p_1,\ldots,p_d\}}((x_1,\ldots,x_d))=\{(T^{p_1(n)}x_1, \ldots, T^{p_d(n)}x_d): n\in \Z \},$$
and
$$\overline{\O}_\A((x_1,\ldots,x_d))=\overline{\{(T^{p_1(n)}x_1, \ldots, T^{p_d(n)}x_d): n\in \Z \}}.$$

\begin{lem}
Let $(X,T)$ be a t.d.s. and $d\in \N$. Let $\A=\{p_1,\ldots, p_d\}$ be a set of integral polynomials and define
$$G_{\A}: X^d\rightarrow 2^{X^d}, (x_1,\ldots, x_d)\mapsto \overline{\O}_\A((x_1,\ldots,x_d)).$$
Then $G_\A$ is l.s.c.
\end{lem}

\begin{proof}
Let ${\bf x}=(x_1,\ldots, x_d)\in X^d$ and $\ep>0$. Let $g(n)({\bf x})=(T^{p_1(n)}x_1, \ldots, T^{p_d(n)}x_d)$ for $n\in \Z$. Then $G_\A({\bf x})=\overline{\O}_\A({\bf x})=\overline{\{g(n){\bf x}: n\in \Z\}}$. Since $G_\A({\bf x})$ is compact, there exist $m_1,m_2\ldots, m_k\in \Z$ such that $\{g(m_1)({\bf x}), g(m_2)({\bf x}),\ldots, g(m_k)({\bf x})\}$ is $\ep/2$-dense in $G_\A({\bf x})$.

Since $T$ is continuous, there is some $\d>0$ such that whenever $\rho_d({\bf x},{\bf x'})<\d$, $$\rho_d(g(m_i)({\bf x}), g(m_i)({\bf x'}))<\ep/2, \ 1\le i\le k,$$
where $\rho_d$ is the metric of $X^d$. Now we show for each ${\bf x'}\in X^d$, if $\rho_d({\bf x}, {\bf x'})<\d$, then $G_\A({\bf x})\subseteq B_{\ep}(G_\A({\bf x'}))$.
In fact, let ${\bf y}\in G_\A({\bf x})$. Then there is some $i\in \{1,2\ldots, k\}$ such that $\rho_d(g(m_i)({\bf x}),{\bf y})<\ep/2$. As $\rho_d (g(m_i)({\bf x}), g(m_i)({\bf x'}))<\ep/2$, we have that $\rho_d(g(m_i)({\bf x'}),{\bf y})<\ep$. That means ${\bf y}\in B_\ep(G_\A({\bf x'}))$.
The proof is complete.
\end{proof}

With the same proof, we have the following lemma.

\begin{lem}\label{cont-points}
Let $(X,T)$ be a t.d.s. and $d\in \N$. Let $\A=\{p_1,\ldots, p_d\}$ be a set of integral polynomials and define
$$F_{\A}: X \rightarrow 2^{X^d},  x \mapsto \overline{\O}_\A(x^{\otimes d}).$$
Then $F_\A$ is l.s.c.
\end{lem}

Let $(X,T)$ be a t.d.s. By Theorem \ref{thm-Fort}, for any finite set $\A$ of integral polynomials,
the continuous points of $F_\A$ (in Lemma \ref{cont-points}) form a residual subset $X_\A$ of $X$. Let
\begin{equation}\label{cont-poly}
X_{pol}=\bigcap X_\A,
\end{equation}
where $\A$ ranges over all finite sets consisting of integral polynomials.
Since the set of such $\A$ is countable, $X_{pol}$ is still a residual subset of $X$. It is easy to verify  $T(X_{pol})=X_{pol}$.

\subsection{Polynomial saturation theorems}\
\medskip

Now we are able to prove polynomial saturation theorems in different settings, namely Theorem \ref{thm-poly-sat1} and Theorem \ref{thm-poly-sat2}.

\subsubsection{} First we give the polynomial version of Theorem \ref{thm-GHSWY}.

\begin{thm}\label{thm-poly-sat1}
Let $(X,T)$ be a minimal t.d.s., and $\pi:X\rightarrow X_\infty$ be the factor map from $X$ to its maximal $\infty$-step pro-nilfactor $X_\infty$.
Then there are minimal t.d.s. $X^*$ and $X_\infty^*$ which are almost one to one
extensions of $X$ and $X_\infty$ respectively, an open factor map $\pi^*$ and a commuting diagram below
\[
\begin{CD}
X @<{\varsigma^*}<< X^*\\
@VV{\pi}V      @VV{\pi^*}V\\
X_\infty @<{\varsigma}<< X_\infty^*
\end{CD}
\]
such that there is a $T$-invariant residual subset $X^*_0$ of $X^*$ having the following property:  for all $x\in X^*_0$, all sets of essentially distinct
non-constant integral polynomials $\A=\{p_1,\ldots, p_d\}$ and $d\in \N$, $\overline{\O}_\A(x^{\otimes d})$ is ${\pi^*}^{(d)}$-saturated, that is,
$$\overline{\O}_\A(x^{\otimes d})=({\pi^*}^{(d)})^{-1}\Big({\pi^*}^{(d)}(\overline{\O}_\A(x^{\otimes d}))\Big).$$

If in addition $\pi$ is open, then $X^*=X$, $X_\infty^*=X_\infty$ and $\pi^*=\pi$.
\end{thm}

\begin{proof}
Let $X^*, X_\infty^*$ and $\pi^*$ be as in Theorem \ref{thm-GHSWY}. Let $X_0^*=X^*_{pol}$ be defined by \eqref{cont-poly}. It is a $T$-invariant residual subset of $X^*$.
We show that for all $x\in X_0^*$, $d\in\N$ and all sets of integral polynomials $\A=\{p_1,\ldots, p_d\}$, $\overline{\O}_\A(x^{\otimes d})$ is ${\pi^*}^{(d)}$-saturated.

Let $y=\pi^*(x)$. Since $\pi^*$ is open, it suffices to show for any ${\bf y}\in \O_\A(y^{\otimes d})$, $({\pi^*}^{(d)})^{-1}({\bf y})\subset \overline{\O}_\A(x^{\otimes d})$.

\medskip

Now let ${\bf y}\in \O_\A(y^{\otimes d})$. Then there is some $k\in \Z$ such that ${\bf y}=(T^{p_1(k)}y, T^{p_2(k)}y,\ldots, T^{p_d(k)}y)$. Let $(x_1,x_2,\ldots,x_d)\in ({\pi^*}^{(d)})^{-1}({\bf y})=\prod_{i=1}^d(\pi^*)^{-1}(T^{p_i(k)} y)$.

\medskip

Let $\A'=\{p_i(n+k): 1\le i\le d\}$. Then $\A'$ is still a set of essentially distinct non-constant integral polynomials. Since $x\in X_0^*= X^*_{pol}$, it is a continuous point of the map
$$F_{\A'}: X^*\rightarrow 2^{{X^*}^{d}}.$$
Thus for any fixed $\ep>0$, there is some $\d>0$ such that $\rho(x,x')<\d$ and $x,x'\in X^*$ implies that
$$(\rho^{d})_H(F_{\A'}(x), F_{\A'}(x'))<\ep,$$
where $$(\rho^{d})_H\big((z_i)_{i=1}^{d},(y_i)_{i=1}^{d}\big)=\max_{1\le i\le d}\rho(z_i,y_i)$$ for any
$(z_i)_{i=1}^{d},(y_i)_{i=1}^{d}\in (X^*)^{d}$, and $(\rho^{d})_H$ is the corresponding Hausdorff metric on $2^{{X^*}^{d}}$.

Since $(x_1,x_2,\ldots,x_d)\in ({\pi^*}^{(d)})^{-1}({\bf y})=\prod_{i=1}^d(\pi^*)^{-1}(T^{p_i(k)} y)$, it follows that
$$y\in \pi^*(B_\d(x))\cap \bigcap_{1\le i\le d} \pi^*(T^{-p_i(k)}(B_\ep(x_i))). $$

Applying Theorem \ref{Thm-Qiu} to essentially distinct non-constant integral polynomials \break $\{p_i(n+k)-p_i(k)\}_{i=1}^d$, we know that there is some $\widetilde{n}\in \N$ such that
$$B_\d(x)\cap \bigcap_{1\le i\le d} T^{-(p_i(\widetilde{n}+k)-p_i(k))}(T^{-p_i(k)}B_\ep(x_i)) \neq \emptyset.$$
That is,
$B_\d(x)\cap \bigcap_{1\le i\le d} T^{-p_i(\widetilde{n}+k)}B_\ep(x_i) \neq \emptyset.$
Let
\begin{equation}\label{t1}
  x'\in B_\d(x)\cap \bigcap_{1\le i\le d} T^{-p_i(\widetilde{n}+k)}B_\ep(x_i).
\end{equation}
Then $\rho(x,x')<\d$, and hence
\begin{equation}\label{t2}
  \rho^{d}_H(F_{\A'}(x), F_{\A'}(x'))<\ep.
\end{equation}
By \eqref{t1},
$$\Big(T^{p_i(\widetilde{n}+k)}(x')\Big)_{1\le i\le d}\in \prod_{i=1}^dB_\ep(x_i)= B_\ep(x_1)\times B_\ep(x_2)\times \cdots \times B_\ep(x_d).$$
So $F_{\A'}(x')\cap \prod_{i=1}^dB_\ep(x_i)\neq \emptyset.$
Following \eqref{t2}, one has that
$$F_{\A'}(x)\cap \prod_{i=1}^dB_{2\ep}(x_i)\neq \emptyset.$$
In particular, there is some $m\in \Z$ such that
$$\Big(T^{p_1(m+k)}x, T^{p_2(m+k)}x,\ldots, T^{p_d(m+k)}x\Big)\in \prod_{i=1}^dB_{2\ep}(x_i).$$
As $\ep$ is arbitrary, we have that $(x_1,x_2,\ldots, x_d)\in \overline{\O_\A}(x^{\otimes d})$, and hence
$$({\pi^*}^{(d)})^{-1}({\bf y})\subset \overline{\O}_\A(x^{\otimes d}).$$

\medskip
If in addition $\pi: X\rightarrow X_\infty$ is open, then by Remark \ref{rem-3.3},  $X^*=X$, $X_\infty^*=X_\infty$ and $\pi^*=\pi$. The proof is complete.
\end{proof}

It is easy to see that Theorem \ref{thm-GHSWY} is a special case of Theorem \ref{thm-poly-sat1} when $\A=\{n,2n,\ldots, dn\}$. Also it is clear that Theorem \ref{thm-poly-sat1} implies Theorem \ref{Thm-Qiu}.
We can rewrite Theorem \ref{thm-poly-sat1} as follows:

\begin{thm}
Let $(X,T)$ be a minimal t.d.s., and $\pi:X\rightarrow X_\infty$ be the factor map from $X$ to its maximal $\infty$-step pro-nilfactor $X_\infty$.
Then there are minimal t.d.s. $X^*$ and $X_\infty^*$ which are almost one to one
extensions of $X$ and $X_\infty$ respectively, an open factor map $\pi^*$ and a commuting diagram below
\[
\begin{CD}
X @<{\varsigma^*}<< X^*\\
@VV{\pi}V      @VV{\pi^*}V\\
X_\infty @<{\varsigma}<< X_\infty^*
\end{CD}
\]
such that there is a $T$-invariant residual subset $X^*_0$ of $X^*$ having the following property:  for all $x\in X^*_0$, for any non-empty open subsets $V_1,\ldots, V_d$ of $X^*$ with $\pi(x)\in \bigcap _{i=1}^d \pi^*(V_i)$ and essentially distinct non-constant integral polynomials $p_1,p_2,\ldots, p_d$ with $p_{i}(0)=0$, $i=1,2,\ldots, d$, there is some $n\in \N$ such that
$$x\in  T^{-p_1(n)}V_1\cap T^{-p_2(n)}V_2\cap \cdots \cap T^{-p_d(n)}V_d.$$
\end{thm}

\subsubsection{Some special cases}

When $(X,T)$ is distal, then $\pi: X\rightarrow X_\infty$ is open. Thus we have the following corollary.

\begin{cor}
Let $(X,T)$ be a minimal distal t.d.s., and $\pi:X\rightarrow X_\infty$ be the factor map from $X$ to its maximal $\infty$-step pro-nilfactor $X_\infty$.
Then there is a residual subset $X_0$ of $X$ such that for all $x\in X_0$ and all sets of essentially distinct non-constant integral polynomials $\A=\{p_1,\ldots, p_d\}$, $\overline{\O}_\A(x^{\otimes d})$ is ${\pi}^{(d)}$-saturated.
\end{cor}

When $(X,T)$ is a weakly mixing minimal t.d.s., $X_\infty$ is trivial, and we have the following corollary.

\begin{cor}\label{wm-3} \cite[Theorem 1.2]{HSY-19-1}
Let $(X,T)$ be a weakly mixing and minimal t.d.s. Then there is a dense $G_\delta$ subset $X_0$ of $X$ such that for each $x\in X_0$ and  essentially distinct non-constant integral polynomials $p_1, \ldots, p_k$,
$$\{(T^{p_1(n)}x, \ldots, T^{p_k(n)}x): n\in\Z\}$$
is dense in $X^k$.
\end{cor}

Let $(X,T)$ be a t.d.s. and $\A=\{p_1,\ldots, p_d\}$ be a set of integral polynomials. Put
\begin{equation}\label{nn}
  N_\A(X)=\overline{\bigcup_{x\in X}\O_{\A}(x^{\otimes d})}=\overline{\{(T^{p_1(n)}x, \ldots, T^{p_d(n)}x): x\in X, n\in \Z \}}.
\end{equation}
Note that when $\A=\{n,2n, \ldots, dn\}$, $N_{\A}(X)=N_d(X)$. The special case of the following theorem when $\A=\{n,2n, \ldots, dn\}$ has been considered in \cite[Lemma 3.3]{GHSWY}.

\begin{thm}\label{thm-poly-sat2}
Let $(X,T)$ be a minimal t.d.s., and $\pi:X\rightarrow X_\infty$ be the factor map  from $X$ to its maximal $\infty$-step pro-nilfactor $X_\infty$.
Then there are minimal t.d.s. $X^*$ and $X_\infty^*$ which are almost one to one
extensions of $X$ and $X_\infty$ respectively, an open factor map $\pi^*$, and a commuting diagram below
\[
\begin{CD}
X @<{\varsigma^*}<< X^*\\
@VV{\pi}V      @VV{\pi^*}V\\
X_\infty @<{\varsigma}<< X_\infty^*
\end{CD}
\]
such that
for all sets of essentially distinct non-constant integral polynomials $\A=\{p_1,\ldots, p_d\}$, $N_\A(X^*)$ is ${\pi^*}^{(d)}$-saturated, that is,
$$N_\A(X^*)=({\pi^*}^{(d)})^{-1}\Big({\pi^*}^{(d)}(N_\A(X^*))\Big)=({\pi^*}^{(d)})^{-1}
\Big(N_\A(X_\infty^*)\Big).$$

If in addition $\pi$ is open, then $X^*=X$, $X_\infty^*=X_\infty$ and $\pi^*=\pi$.
\end{thm}

\begin{proof}
Let $X^*, X_\infty^*$ and $\pi^*$ be as in Theorem \ref{thm-GHSWY}.
It is clear that $N_\A(X^*) \subseteq ({\pi^*}^{(d)})^{-1} \Big( N_\A(X_\infty^*) \Big)$. 
Let ${\bf y}=(y_1,y_2,\ldots, y_d)\in N_\A(X_\infty^*)$. It remains to show that $({\pi^*}^{(d)})^{-1}({\bf y})\subseteq N_\A(X^*)$.

Put $Y_0=\pi^*(X^*_0)$, where $X^*_0$ is the set defined in Theorem \ref{thm-poly-sat1}.
Then $Y_0$ is a dense subset of $X_\infty^*$.
Since ${\bf y}=(y_1,y_2,\ldots, y_d)\in N_\A(X_\infty^*)$, there is some sequence $\{y^i\}_{i=1}^\infty\subseteq Y_0$ and $\{n_i\}_{i=1}^\infty\subseteq \Z$ such that $$(T^{p_1(n_i)}y^i, T^{p_2(n_i)}y^i,\ldots, T^{p_d(n_i)}y^i)\rightarrow {\bf y}=(y_1,y_2,\ldots, y_d), \ i\to\infty.$$
Since $y^i\in Y_0$, there is some $x^i\in X_0^*$ with $y^i=\pi(x^i)$ and then by Theorem \ref{thm-poly-sat1}, $$({\pi^*}^{(d)})^{-1} (T^{p_1(n_i)}y^i, T^{p_2(n_i)}y^i,\ldots, T^{p_d(n_i)}y^i)\subseteq \overline{\O_\A}((x^i)^{\otimes d})\subseteq N_\A(X^*).$$
As $\pi^*$ is open, $({\pi^*}^{(d)})^{-1}$ is continuous and it follows that
$$(\pi^{(d)})^{-1} (y_1,y_2,\ldots, y_d) \subseteq N_\A(X^*).$$
Thus we conclude that $({\pi^*}^{(d)})^{-1} \Big( N_\A(X_\infty^*) \Big) \subseteq N_\A(X^*)$. The proof is complete.
\end{proof}

\section{The induced systems $N_\infty(X,\A)$ and $M_\infty(X,\A)$ for polynomials}\label{Section-Def-M}

In this section, for a given finite set $\A$ of integral polynomials and a t.d.s. $(X,T)$, we define t.d.s.
$N_\infty(X,\A)$ and $M_\infty(X,\A)$ of $\Z^2$-actions. We will show that in fact they are isomorphic, which enable us to use
them freely in different situations.
Moreover, we will show that Theorem E is equivalent to the denseness of
minimal points of $M_\infty(X,\A)$.

\subsection{The definition of $N_\infty(X,\A)$}\

\subsubsection{Notation}

Let $d\in\N$ and $\A=\{p_1, p_2,\cdots, p_d\}$ be a collection of integral polynomials with $p_i(0)=0$, $1\le i\le d$.
A point of $(X^d)^{\Z}$ is denoted by
$${\bf x}=({\bf x}_n)_{n\in {\Z}}=\Big((x^{(1)}_n, x^{(2)}_n,\cdots, x^{(d)}_n) \Big)_{n\in \Z}.$$

Let $\vec{p}=(p_1,p_2,\cdots, p_d)$ and let $T^{\vec{p}(n)}: X^d\rightarrow X^d$ be defined by
\begin{equation}\label{}
  T^{\vec{p}(n)}(x_1,x_2,\cdots, x_d)=(T^{p_1(n)}x_1, T^{p_2(n)}x_2,\cdots, T^{p_d(n)}x_d).
\end{equation}

Recall that $T^{(d)}=T\times T\times \cdots\times T$ ($d$-times).
Define $T^\infty: (X^d)^{{\Z}}\rightarrow (X^d)^{{\Z}}$ by
$$T^\infty({\bf x}_n)_{n\in {\Z}}=(T^{(d)}{\bf x}_n)_{n\in {\Z}}.$$
Let $\sigma: (X^d)^{{\Z}}\rightarrow (X^d)^{{\Z}}$
be the shift map, i.e.,  for all $({\bf x}_n)_{n\in {\Z}}\in (X^d)^{{\Z}}$
$$(\sigma {\bf x})_n={\bf x}_{n+1}, \ \forall n\in {\Z}. $$

\subsubsection{Definition of $N_\infty(X,\A)$}
Let $x^{\otimes d}=(x,x,\ldots, x)\in X^d$ and
$$\D_{\infty}(X)=\{x^{(\infty)}\triangleq (\ldots, x^{\otimes d}, x^{\otimes d},\ldots ) \in (X^d)^{{\Z}}: x\in X\}.$$
For each $x\in X$, put
\begin{equation}\label{}
  \w_x^\A\triangleq(T^{\vec{p}(n)}x^{\otimes d})_{n\in \Z}=\big ((T^{p_1(n)}x, T^{p_2(n)}x,\ldots, T^{p_d(n)}x) \big)_{n\in \Z}
  \in (X^d)^{{\Z}},
\end{equation}
and set
\begin{equation}\label{}
  N_\infty(X,\A)=\overline{\bigcup\{\O(\w_x^\A,\sigma): x\in X\}}\subseteq (X^d)^{{\Z}}.
\end{equation}

\begin{rem} We have the following
\begin{enumerate}

\item It is clear that $N_\infty(X,\A )$ is invariant under the action of $T^\infty$ and $\sigma$, and $T^\infty\c \sigma=\sigma\c T^\infty$. Thus
$(N_\infty(X,\A), \langle T^\infty, \sigma\rangle)$ is a $\Z^2$-t.d.s.

\item If $(X,T)$ is transitive, then for each transitive point $x$ of $(X,T)$, $N_\infty(X,\A)=\overline{\O(\w_x^\A, \langle T^\infty, \sigma\rangle )}$.

\item Sometimes we identify  points in $(X^{d_1+d_2})^{\Z}$ as $(X^{d_1})^{\Z}\times (X^{d_2})^{\Z}$ as follows: 

{\small $$\Big((x^{(1)}_n, \cdots, x^{(d_1+d_2)}_n) \Big)_{n\in \Z}=\Big(\big((x^{(1)}_n, \cdots, x^{(d_1)}_n) \big)_{n\in \Z}, \big((x^{(d_1+1)}_n, \cdots, x^{(d_1+d_2)}_n) \big)_{n\in \Z}\Big).$$}


\item Similarly, we may define systems in $(X^{d})^{\Z_+}$ associated with the polynomials $\A$. In such a case,
$\sigma$ is a continuous surjective map but not a homeomorphism.
\end{enumerate}
\end{rem}

\subsection{Some examples}
The action $T^\infty$ on $N_\infty(X,\A)$ is clear. To make readers familiar with the definition, we provide below several examples to describe the action under $\sigma$.

\subsubsection{The case $\A=\{p\}$ with $p(n)=n$}   \label{subsection-linear-one}\
\medskip

In this case we show that for a minimal t.d.s. $(X,T)$,
\begin{equation}\label{}
  (N_\infty(X,\A), \sigma)\cong (X, T),
\end{equation}
where $\cong$ means two systems are isomorphic.

\medskip

Now $p(n)=n$, and we have for each $x\in X$
$$\omega^\A_x=(\ldots, T^{-2}x, T^{-1}x, \underset{\bullet}x,Tx,T^2x,\ldots )=(T^nx)_{n\in {\Z}},$$
where $`` \ \underset{\bullet } \ "$ means the $0$-th coordinate.
Note that
$$\sigma \w_x^\A=(\ldots, T^{-1}x, x, \underset{\bullet }Tx, T^2x, T^{3}x, \ldots)=T^\infty \w_x^\A =\w_{Tx}^\A .$$
It follows that $\sigma^n\w_x^\A=\w^\A_{T^n x}$ for all $n\in \Z$. By this fact and the minimality of $(X,T)$ it is easy to verify that
$$N_\infty(X,\A)=\{\omega_x^\A: x\in X\}.$$
Thus if we define $$\phi:(X,T)\rightarrow (N_\infty(X,\A),\sigma), \ x\mapsto \omega_x ^\A,$$
and then it is an isomorphism between them. So
$$ (N_\infty(X,\A),\sigma)\cong (X,T).$$
Also it is easy to see that $\overline{\O}(\w_x^\A,\sigma)=N_\infty(X,\A)$ for each $x\in X$ in this case.

\medskip

\subsubsection{The case for linear polynomials}\ 
\medskip

Let $\A=\{p_1, p_2, \ldots, p_d\}$, where $p_i(n)=a_in,\ 1\le i\le d$ and $a_1,a_2,\ldots, a_d$ are distinct non-zero integers.
In this case we show that for a minimal t.d.s. $(X,T)$,
\begin{equation}\label{}
  (N_\infty(X,\A), \sigma)\cong ({N}_{\A}(X,T), \tau_{\vec{a}}),
\end{equation}
where ${N}_\A(X,T)=\overline{\O}(\Delta_d(X), \tau_{\vec{a}})$ is as defined in \eqref{nn}, and $\tau_{\vec{a}}=T^{a_1}\times T^{a_2} \times \cdots \times T^{a_d}$
with $\vec{a}=(a_1,a_2,\ldots, a_d)$.

\medskip
As $\A=\{p_1, p_2, \ldots, p_d\}$, we have that for each $x\in X$
$$\omega_x^\A=\Big((T^{a_1n}x, T^{a_2n}x, \ldots, T^{a_dn}x)\Big)_{n\in {\Z}}=(\ldots,\tau^{-1}_{\vec{a}}x^{\otimes d}, \underset{\bullet }x^{\otimes},
\tau_{\vec{a}}x^{\otimes d}, \tau_{\vec{a}}^2 x^{\otimes d},\ldots)\in (X^d)^{\Z}.$$
Note that
\begin{equation*}
  \begin{split}
\sigma \omega^\A_x &= \Big((T^{a_1(n+1)}x, T^{a_2(n+1)}x, \ldots, T^{a_d(n+1)}x)\Big)_{n\in {\Z}}\\
&=\Big(\tau_{\vec{a}} (T^{a_1n}x, T^{a_2n}x, \ldots, T^{a_dn}x)\Big)_{n\in {\Z}}=\tau_{\vec{a}}^\infty \omega_x^\A,
\end{split}
\end{equation*}
where $\tau_{\vec{a}}^\infty=\cdots \times \tau_{\vec{a}}\times \tau_{\vec{a}}\times \cdots,$
and for all $n\in {\Z}$
\begin{equation*}
  \begin{split}
    \sigma^n\omega_x^\A &= ({\tau_{\vec{a}}^\infty})^n\w_x^\A\\
     & = \big( \ldots, \tau^{-1}_{\vec{a}} \tau_{\vec{a}}^nx^{\otimes d},  \underset{\bullet } {\tau_{\vec{a}}^nx^{\otimes d}}, \tau_{\vec{a}} \tau_{\vec{a}}^nx^{\otimes d}, \tau_{\vec{a}}^2\tau_{\vec{a}}^nx^{\otimes d},\ldots \big)\\
     & = \widetilde{\tau_{\vec{a}}}\big(\ldots, \tau_{\vec{a}}^nx^{\otimes d}, \underset{\bullet }{\tau_{\vec{a}}^nx^{\otimes d}}, \tau_{\vec{a}}^nx^{\otimes d},\ldots \big),
   \end{split}
\end{equation*}
where $ \widetilde{\tau_{\vec{a}}}=\cdots \times \tau^{-1}_{\vec{a}} \times  \underset{\bullet } {{\rm id}_{X^d}} \times \tau_{\vec{a}}\times \tau^{2}_{\vec{a}}\times \cdots$.
Thus we have that
$$N_\infty(X, \A)=\{(\ldots, \tau^{-1}_{\vec{a}} {\bf y},  \underset{\bullet } {\bf y}, \tau_{\vec{a}} {\bf y}, \tau_{\vec{a}}^2 {\bf y},\ldots ):{\bf y}\in {N}_d(X,T)\}.$$
Now we define
$$\phi:({N}_\A(X,T),\tau_{\vec{a}})\rightarrow (N_\infty(X,\A),\sigma), \ {\bf y}\mapsto (\ldots, \tau^{-1}_{\vec{a}} {\bf y},  \underset{\bullet } {\bf y}, \tau_{\vec{a}} {\bf y}, \tau_{\vec{a}}^2 {\bf y},\ldots ),$$ and it is an isomorphism. Thus
$ (N_\infty(X,\A), \sigma)\cong ({N}_\A (X,T), \tau_{\vec{a}}),$
and $$ (N_\infty(X,\A),\langle T^\infty, \sigma\rangle)\cong ({N}_\A (X,T), \langle T^{(d)},\tau_{\vec{a}} \rangle).$$

Similarly, we have that $\big(\overline{\O}(\w_x^\A,\sigma),\sigma\big)\cong \big(\overline{\O}(x^{\otimes d},\tau_{\vec{a}}),\tau_{\vec{a}}\big)$
for each $x\in X$.

\subsubsection{Weakly mixing systems}\
\medskip

Let $(X, T)$ be a weakly mixing minimal t.d.s. In this case $N_\infty(X,\A)$ will be very big in $(X^d)^\Z$. The following ones are special cases proved
in Section 5 
(see Theorem \ref{thm-wm1}).

\medskip

\paragraph{{\bf Case} (I):\ $\A=\{p\}$ with $p(n)=n^2$}
\
\medskip

In this case for each $x\in X$,
$$\omega_x^\A=(\ldots, T^{p(-1)}x,  \underset{\bullet }x, T^{p(1)}x, T^{p(2)}x, \ldots)=(T^{p(n)}x)_{n\in {\Z}}.$$
We have that
\begin{equation}\label{}
  (N_\infty(X,\A),\sigma)\cong (X^{\Z}, \sigma),
\end{equation}
and there is a residual subset $X_0$ of $X$ such that for all $x\in X_0$, $\overline{\O}(\w_x^\A,\sigma)=X^\Z$.

\medskip

\paragraph{{\bf Case} (II):\ $\A=\{p_1,p_2\}$ with $p_1(n)=n$ and $p_2(n)=n^2$.}
\
\medskip

In this case for each $x\in X$, $\omega_x^\A$ is the point
$$\Big(\ldots, ( T^{p_1(-1)}x, T^{p_2(-1)}x), \underset{\bullet } {(x, x)},( T^{p_1(1)}x, T^{p_2(1)}x), \ldots\Big)=\Big((T^{p_1(n)}x, T^{p_2(n)}x)\Big)_{n\in {\Z}}.$$
We have that
\begin{equation}\label{}
  (N_\infty(X,\A),\sigma)\cong (X\times X^{\Z}, T \times \sigma'),
\end{equation}
where $\sigma': X^{\Z}\rightarrow X^{\Z}$ is the left shift, and there is a residual subset $X_0$ of $X$ such that for all $x\in X_0$, $\overline{\O}(\w_x^\A,\sigma) \cong X\times X^\Z$.

\subsection{The definition of $M_\infty(X,\A)$} \label{subsection-linear-terms}\
\medskip

In the previous subsection we see that when $\A=\{p_1, p_2, \ldots, p_d\}$ with $p_i(n)=a_in, 1\le i\le d$ and $a_1,a_2,\ldots, a_d$
are distinct non-zero integers, $(N_\infty(X,\A),\sigma) \cong ({N}_\A(X,T), \tau_{\vec{a}})$. In this subsection, we define
$M_\infty(X,\A)$, and particularly we pay attention to the case when $\A$ contains at least one linear element,
and give an equivalent way to look at $(N_\infty(X,\A),\sigma)$.

\subsubsection{Definition of $M_\infty(X,\A)$}
Assume that $\A=\{p_1, \ldots, p_{s}, p_{s+1}, \ldots, p_{d}\}$, where $p_i(n)=a_in, 1\le i\le s$ with
$s\ge 0$, $a_1,a_2,\ldots, a_s\in \Z\setminus\{0\}$ are distinct, and $\deg p_{i}\ge 2, s+1\le i\le d$.
Note that $s=0$  means that $\A=\{p_1, \ldots, p_{d}\}$ with $\deg p_i\ge 2$, $1\le i\le d$.

We have that
$$\omega_x^\A =\Big((T^{a_1n}x, \ldots, T^{a_{s}n}x, T^{p_{s+1}(n)}x,\ldots, T^{p_d(n)}x)\Big)_{n\in {\Z}}\in (X^d)^{\Z}.$$

Define $\widetilde{\sigma}: X^s\times (X^{d-s})^{\Z} \rightarrow X^s\times (X^{d-s})^{\Z}$ as follow:
for $((x_1,\ldots, x_s), {\bf x})\in X^s\times (X^{d-s})^{\Z}$, let
$$\widetilde{\sigma}\Big(\big((x_1,x_2, \ldots, x_s), {\bf x}\big)\Big)=\Big(\big((T^{a_1}x_1, T^{a_2}x_2, \ldots, T^{a_s}x_s), \sigma'{\bf x}\big)\Big),$$
where $\sigma'$ is the left shift on $(X^{d-s})^{\Z}$. Recall that $\tau_{\vec{a}}=T^{a_1}\times \ldots \times T^{a_s}$. So, $\widetilde{\sigma}=\tau_{\vec{a}}\times \sigma'$.
Let
\begin{equation}\label{xi}
  \xi_x^\A=\Big((x, \ldots , x), (T^{p_{s+1}(j)}x,\ldots, T^{p_d(j)}x)_{j\in {\Z}}\Big)\in X^s\times (X^{d-s})^{\Z}.
\end{equation}
Then for $n\in {\Z}$,
\begin{equation}\label{w3}
  \widetilde {\sigma}^n\xi_x^\A =\Big((T^{a_1n }x, \ldots , T^{a_sn}x), (T^{p_{s+1}(n+j)}x,\ldots, T^{p_d(n+j)}x)_{j\in {\Z}}\Big)\in X^s\times (X^{d-s})^{\Z}.
\end{equation}

We set
\begin{equation}\label{}
  M_\infty(X,\A)=\overline{\bigcup\{\O(\xi_x^\A,\widetilde{\sigma}): x\in X\}}\subseteq X^s\times (X^{d-s})^{\Z}.
\end{equation}

It is clear that $M_\infty(X,\A )$ is invariant under the action of $T^\infty$ and $\widetilde{\sigma}$, and $(M_\infty(X,\A), \langle T^\infty, \widetilde{\sigma}\rangle)$ is a t.d.s.

\subsubsection{} Now we show that

\begin{lem}\label{M-N-equ}
For any t.d.s. $(X,T)$ we have
$$(M_\infty(X,\A),\langle T^\infty, \widetilde{\sigma}\rangle)\cong (N_\infty(X,\A), \langle T^\infty, \sigma\rangle),$$ and for each $x\in X$
$$\big(\overline{\O}(\w_x^\A,\sigma),\sigma\big)\cong \big(\overline{\O}(\xi_x^\A,\widetilde{\sigma}), \widetilde{\sigma}\big).$$
\end{lem}

\begin{proof}
Note that for all $n\in {\Z}$,
\begin{equation}\label{w4}
  \begin{split}
    \sigma^n \omega_x^\A & = \Big((T^{a_1(n+j)}x, \ldots, T^{a_{s}(n+j)}x, T^{p_{s+1}(n+j)}x,\ldots, T^{p_d(n+j)}x)\Big)_{j\in {\Z}} \\
       &= \Big(\big( \tau_{\vec{a}}^{n+j} (x^{\otimes s})\big)_{j\in {\Z}}, (\sigma')^n\big( T^{p_{s+1}(j)}x,\ldots, T^{p_d(j)}x\big)_{j\in {\Z}}\Big)\\
       &= \Big(\widetilde{\tau_{\vec{a}}}\big(\ldots , \tau_{\vec{a}}^nx^{\otimes s}, \underset{\bullet }{ \tau_{\vec{a}}^nx^{\otimes s}}, \tau_{\vec{a}}^nx^{\otimes s},\ldots \big), (\sigma')^n\big( T^{p_{s+1}(j)}x,\ldots, T^{p_d(j)}x\big)_{j\in {\Z}}\Big) ,
   \end{split}
\end{equation}
where $ \widetilde{\tau_{\vec{a}}}=\ldots \tau^{-1}_{\vec{a}}\times \underset{\bullet }{{\rm id}_{X^s}}\times \tau_{\vec{a}}\times \tau^{2}_{\vec{a}}\times \ldots$. By \eqref{w3} and \eqref{w4}, we see that for sequences $\{x_i\}_{i\in\Z}$ of $X$ and $\{n_i\}_{i\in\Z}$ of $\Z$,
$$\widetilde{\sigma}^{n_i} \xi_{x_i}^\A \to ({\bf y}, {\bf x}) \in M_\infty(X,\A)$$
 if and only if
 $$\sigma^{n_i} \omega_{x_i}^\A\to \Big( (\ldots,  \tau^{-1}_{\vec{a}} {\bf y}, \underset{\bullet } {\bf y}, \tau_{\vec{a}} {\bf y}, \tau_{\vec{a}}^2 {\bf y},\ldots ), {\bf x}\Big)\in N_\infty(X,\A).$$
Now we define
$$\phi:(M_\infty(X,T), \widetilde{\sigma})\rightarrow (N_\infty(X,\A),\sigma), \ ({\bf y}, {\bf x})\mapsto\Big( (\ldots,  \tau^{-1}_{\vec{a}} {\bf y}, \underset{\bullet } {\bf y}, \tau_{\vec{a}} {\bf y}, \tau_{\vec{a}}^2 {\bf y},\ldots ), {\bf x}\Big),$$ and it is an isomorphism. Thus
$$ (M_\infty(X,T), \widetilde{\sigma}) \cong (N_\infty(X,\A), \sigma),$$
and
$$ (M_\infty(X,\A),\langle T^\infty, \widetilde{\sigma}\rangle)\cong (N_\infty(X,\A), \langle T^{\infty}, \sigma \rangle).$$
Similarly, it is easy to verify that for all $x\in X$,
$\big(\overline{\O}(\w_x^\A,\sigma),\sigma\big)\cong \big(\overline{\O}(\xi_x^\A,\widetilde{\sigma}), \widetilde{\sigma}\big).$
\end{proof}

\subsubsection{Compare $N_\infty(X,\A)$ with $M_\infty(X,\A)$}
We have show that $(N_\infty(X,\A), \langle T^\infty, \sigma\rangle)$ is isomorphic to $ (M_\infty(X,\A), \langle T^\infty, \widetilde{\sigma}\rangle).$
Now we compare them. The definition of $N_\infty(X,\A)$ is cleaner than $M_\infty(X,\A)$, as we do not care whether there are linear terms in $\A$ or not. But when we study the dynamical properties and $\A$ contains linear terms, $M_\infty(X,\A)$ is easy to handle than $N_\infty(X,\A)$, which we will see in the sequel.

\subsubsection{}
In the sequel when there is no room for confusion, {\bf we will use $\sigma$ instead of $\widetilde{\sigma}$
when studying $(M_\infty(X,\A),\langle T^\infty, \widetilde{\sigma}\rangle)$.}

\subsubsection{}
Let $(X,T)$ be a transitive t.d.s. and ${\rm Trans}_T$ be the set of transitive points, which is a residual subset of $X$. By definition, for each $x\in {\rm Trans}_T$ we have that $ M_\infty(X,\A)=\overline{\O(\xi_x^\A, \langle T^\infty, {\sigma}\rangle)}\subseteq X^s\times (X^{d-s})^{\Z}.$ Thus
$(M_\infty(X,\A), \langle T^\infty,\sigma \rangle)$ is a transitive t.d.s. and for each $x\in {\rm Trans}_T$, $\xi_{x}^\A$ is a transitive point.
In particular, if $(X,T)$ is minimal, then $(M_\infty(X,\A), \langle T^\infty,\sigma \rangle)$ is a transitive t.d.s. and for each $x\in X$, $\xi_{x}^\A$ is a transitive point.

\subsubsection{A remark}

First we recall the following polynomial multiple recurrence theorem by Bergelson and Leibman.
\begin{thm}\cite[Corollary 1.8]{BL96}\label{BL-polynomial}
Let $X$ be a compact metric space, and let $\Gamma$ be an abelian group of its homeomorphisms such that $(X,\Gamma)$ is minimal. Let $T_1,\ldots ,T_d\in \Gamma$, $k\in \N$ and let $p_{i,j}$ be integral polynomial with $p_{i,j}(0)=0$ for all $ 1\le i\le k$ and $1\le j\le d$. Then there exists
a residual set $X_0\subseteq X$ satisfying that for each $x\in X_0$ there exists an increasing sequence
$\{n_m\}_{m=1}^\infty$ of $\N$ such that for each $i=1,\ldots,k$,
\begin{equation*}
  T_1^{p_{i,1}(n_m)}\cdots T_d^{p_{i,d}(n_m)}x\rightarrow x, \ m\to \infty.
\end{equation*}
\end{thm}




By the above result it is easy to obtain
\begin{rem}\label{afterBL}
Let $(X,T)$ be a minimal t.d.s. There is a residual subset $\widehat{X}_{pol}$ of $X$ such that for each $x\in \widehat{X}_{pol}$ and any finite set $\A=\{p_1,\ldots, p_d\}$ of integral polynomials vanishing at zero, there exists an increasing sequence $\{n_m\}_{m=1}^\infty$ of $\N$ satisfying
$$\sigma^{n_m}\xi_x^\A \to \xi_x^\A, \ m\to\infty.$$
\end{rem}

\subsection{Condition $(\spadesuit)$ and saturated properties of $M_\infty(X,\A)$}\

\subsubsection{}
Note that for a given integral polynomials $p$ with $p(0)=0$, for each $j\in\Z$ the polynomial $p^{ [j]}$ defined by
$$p^{[j]}(n)=p(n+j)-p(j), \ \forall n\in\Z$$ appears naturally in the definition of $M_\infty(X,\A)$ and $N_\infty(X,\A)$. This causes problems when applying the saturation theorems (it requires the polynomials are distinct) to polynomials $\{p_1,\ldots,p_d\}$, it may happen that $$p_i^{[k_1]}=p_j^{[k_2]}$$ for some $1\le i\not=j\le d$ and $k_1,k_2\in\Z$. For example when $p_1(n)=n^2$ and $p_2(n)=n^2+2n$, we have $p_1^{[1]}=p_2^{[0]}$. To overcome this difficulty we introduce the following condition $(\spadesuit)$.

\subsubsection{Condition $(\spadesuit)$}

\begin{de}
Let $\A=\{p_1, p_2, \ldots, p_{d}\}$ be a family of integral polynomials. We say $\A$ satisfies {\em condition $(\spadesuit)$} if $p_1(0)=\ldots =p_d(0)=0$ and
\begin{enumerate}
  \item $p_1(n)=a_1 n, p_2(n)=a_2n, \ldots, p_{s}(n)=a_{s}n$, where $s\ge 0$, and $a_1,a_2,\ldots, a_s$ are distinct non-zero integers;
  \item $\deg p_{j}\ge 2, s+1\le j\le d$;
  \item for each $i\neq j\in \{s+1,s+2,\ldots, d\}$, $p_j^{[k]}\neq p_i^{[t]}$ for any $k,t\in \Z$.

\end{enumerate}
\end{de}

We have the following simple observation which indicates that the condition $(\spadesuit)$ does not affect our discussion.

\begin{lem}\label{ww=dem}
Let $k\in\N$ and $Q=\{q_1,q_2,\ldots,q_k\}$ be integral polynomials with $q_i(0)=0, 1\le i\le k$. Then there is $1\le d\le k$
and $\A=\{p_1,\ldots,p_d\}\subset \{q_1,\ldots,q_k\}$ such that $\A$ satisfies condition $(\spadesuit)$, and $\{q_1, \ldots, q_k\}\subseteq \{p_i^{[m]}: 1\le i\le d, m\in \Z \}$.

\end{lem}

\begin{proof}
If there are $1\le i\not=j\le k$ and $k_1,k_2\in \Z$ such that $q_i^{[k_1]}=q_j^{[k_2]}$ we just remove one of them (say $q_j$) to get a proper subset $Q'=Q\setminus\{q_j\}$ of $Q$. Continuing this process for finitely many steps, the argument will stop, and the resulting $\A$ is what we want.
\end{proof}



\begin{exam}\label{exam1}
Let $\A'=\{n^2, n^2+2n, n^2+6n\}$ and $\A=\{n^2\}$. Then by the definition, $\A'$ does not satisfy the condition $(\spadesuit)$, $\A$ satisfies the condition $(\spadesuit)$ and $\A'\subseteq \{p^{[m]}: p(n)=n^2\in \A, m\in \Z\}$.
Let $(X,T)$ be a t.d.s. Now we show that $(N_\infty(X,\A), \langle T^\infty,\sigma \rangle)$ is isomorphic to $(N_\infty(X,\A'), \langle T^\infty,\sigma \rangle)$.

Recall that
$$\w_x^\A=(T^{n^2}x)_{n\in \Z}\in X^\Z, \ \text{and}\ \w_x^{\A'}=\Big((T^{n^2}x, T^{n^2+2n}x, T^{n^2+6n}x)\Big)_{n\in \Z}\in (X^3)^\Z .$$
Note that $$\sigma^k\w_x^\A=(T^{(n+k)^2}x)_{n\in \Z}=(T^{n^2+2kn+k^2}x)_{n\in Z}=(T^{\infty})^{k^2}(T^{n^2+2kn}x)_{n\in \Z}, \forall\ k\in \Z.$$
We have that $\w_x^{\A'}=\Big(\w_x^\A, (T^{\infty})^{-1}\sigma \w_x^\A, (T^\infty)^{-9}\sigma^3 \w_x^\A\Big)$ for all $x\in X$.

Let $\Psi=\id \times (T^\infty)^{-1}\sigma\times (T^\infty)^{-9}\sigma^3$.
It is easy to check that
$$\Psi: N_\infty(X,\A)\rightarrow N_\infty(X,\A'), \ {\bf x}\mapsto ({\bf x},(T^\infty)^{-1}\sigma{\bf x}, (T^\infty)^{-9}\sigma^3 {\bf x}), \ \forall\ {\bf x}\in N_\infty(X,\A).$$
is an isomorphism from $(N_\infty(X,\A), \langle T^\infty,\sigma \rangle)$ to $(N_\infty(X,\A'), \langle T^\infty,\sigma \rangle)$.
\end{exam}

Similar to Example \ref{exam1}, we have the following general result:

\begin{thm}
Let $(X,T)$ be a t.d.s. and $\A'$ be a family of integral polynomials with $q(0)=0, \forall q\in \A'$. Then there is a family of integral polynomials $\A$ satisfying the condition $(\spadesuit)$ such that
$$(N_\infty(X,\A),\langle T^\infty, {\sigma}\rangle)\cong (N_\infty(X,\A'), \langle T^\infty, \sigma\rangle),$$
and
$$(M_\infty(X,\A),\langle T^\infty, {\widetilde{\sigma}}\rangle)\cong (M_\infty(X,\A'), \langle T^\infty, \widetilde{\sigma}\rangle).$$
\end{thm}


\subsubsection{}

Let $(X,T)$ be a t.d.s. Assume that $\A=\{p_1, \ldots, p_{s}, p_{s+1}, \ldots, p_{d}\}$, where $p_i(n)=a_in, 1\le i\le s$ with
$s\ge 0$, $a_1,a_2,\ldots, a_s\in \Z\setminus\{0\}$ are distinct, and $\deg p_{i}\ge 2, s+1\le i\le d$. Let $\xi_x^\A$ be defined in \eqref{xi}. Let
\begin{equation}\label{}
  W_{x}^\A=\overline{\O}(\xi_{x}^{\A}, \sigma).
\end{equation}
Then $(W_{x}^\A,\sigma)$ is transitive and $\xi_{x}^\A$ is a transitive point.
Similar to the saturation theorems proved in the previous section, we have


\begin{thm}\label{thm-poly-saturate}
Let $(X,T)$ be a minimal t.d.s., and $\pi:X\rightarrow X_\infty$ be the factor map from $X$ to its maximal $\infty$-step pro-nilfactor $X_\infty$. Then, there are minimal t.d.s. $X^*$ and $X_\infty^*$ which are almost one to one
extensions of $X$ and $X_\infty$ respectively, an open factor map $\pi^*$ and a commuting diagram below such that
\[
\begin{CD}
X @<{\varsigma^*}<< X^*\\
@VV{\pi}V      @VV{\pi^*}V\\
X_\infty @<{\varsigma}<< X_\infty^*
\end{CD}
\]
we have the following properties:
\begin{enumerate}
  \item There is a residual subset $X_0^*$ of $X^*$  such that for all $x\in X_0^*$ and each family $\A=\{p_1,\ldots, p_d\}$ satisfying $(\spadesuit)$, $W_{x}^\A$ is ${\pi^*}^{\infty}$-saturated, that is,
$$W_{x}^\A=({\pi^*}^{\infty})^{-1}\Big({\pi^*}^{\infty}(W_{x}^\A)\Big).$$


  \item For each family $\A=\{p_1,\ldots, p_d\}$ satisfying $(\spadesuit)$, $M_\infty(X^*,\A)$ is ${\pi^*}^\infty$-saturated, that is,
  $$M_\infty(X^*,\A)=({\pi^*}^\infty)^{-1}\Big(M_\infty(X^*_\infty,\A)\Big ).$$
\end{enumerate}

If in addition $\pi$ is open, then $X^*=X$, $X_\infty^*=X_\infty$ and $\pi^*=\pi$.
\end{thm}

\begin{proof}
Let $X^*, X_\infty^*, \pi^*$ and $X_0^*$ be as in Theorem \ref{thm-poly-sat1}. When $\pi: X\rightarrow X_\infty$ is open, $X^*=X$, $X_\infty^*=X_\infty$ and $\pi^*=\pi$.

\medskip

Let $x\in X_0^*$.
Recall that
$$\xi_x^\A=\Big((x, \ldots , x), (T^{p_{s+1}(j)}x,\ldots, T^{p_d(j)}x)_{j\in {\Z}}\Big)\in {X^*}^s\times ({X^*}^{d-s})^{\Z},$$
and for $n\in {\Z}$,
\begin{equation}\label{y1}
   {\sigma}^n\xi_x^\A =\Big((T^{a_1n }x, \ldots , T^{a_sn}x), (T^{p_{s+1}(n+j)}x,\ldots, T^{p_d(n+j)}x)_{j\in {\Z}}\Big).
\end{equation}
To show that $W_x^\A=({\pi^*}^{\infty})^{-1}\Big({\pi^*}^{\infty}(W_x^\A)\Big)$, it suffices to prove  that for each
$${\bf x}=\Big( (x_1,\ldots, x_s), (x_j^{(s+1)}, x_j^{(s+2)},\ldots, x_j^{(d)})_{j\in {\Z}}\Big)\in ({\pi^*}^{\infty})^{-1}\Big({\pi^*}^{\infty}(W_x^\A)\Big),$$ and any neighbourhood $\widetilde{U}$ of ${\bf x}$, there is some $n\in {\Z}$ such that $\sigma^n\xi_x^\A\in \widetilde{U}\subseteq {X^*}^s\times ({X^*}^{d-s})^{\Z}$.
Assume that
{\small $$\widetilde{U}=U_1\times\ldots \times U_s\times \Big(\prod_{j={-\infty}}^{-(k+1)} {X^*}^{d-s}\times \prod_{j=-k}^k(U_j^{(s+1)}\times U_j^{(s+2)}\times \ldots \times U_j^{(d)})\times \prod_{j={k+1}}^\infty {X^*}^{d-s} \Big),$$}
where $k\in \N$.
By condition $(\spadesuit)$,
$$\A'=\{a_1n, \ldots, a_sn, p_{i}(n+j): s+1\le i\le n, -k\le j\le k\}$$ is a family consisting of essentially distinct non-constant integral polynomials.
By Theorem \ref{thm-poly-sat1},
$$\overline{\O}_{\A'}(x^{\otimes m})=({\pi^*}^{(m)})^{-1}\Big({\pi^*}^{(m)}(\overline{\O}_{\A'}(x^{\otimes m}))\Big), $$
where $m=|\A'|=s+(d-s)(2k+1)$. Recall that
$$\O_{\A'}(x^{\otimes m})=\Big\{\Big((T^{a_1n }x, \ldots , T^{a_sn}x), (T^{p_{s+1}(n+j)}x,\ldots, T^{p_d(n+j)}x)_{j=-k}^k\Big): n\in {\Z}\Big\}.$$
Combining with \eqref{y1}, the projection of $W_x^\A$ to ${X^*}^s\times \prod_{j=-k}^k {X^*}^{d-s}$ is $\overline{\O}_{\A'}(x^{\otimes m})$. Thus
coordinates of the projection of ${\bf x}$ is
$$\Big( (x_1,\ldots, x_s), (x_j^{(s+1)}, x_j^{(s+2)},\ldots, x_j^{(d)})_{j=-k}^k\Big)\in ({\pi^*}^{(m)})^{-1}\Big({\pi^*}^{(m)}(\overline{\O}_{\A'}(x^{\otimes m}))\Big)=\overline{\O}_{\A'}(x^{\otimes m}).$$
So there is some $n\in {\Z}$ such that
\begin{equation*}
  \begin{split}
     & \quad \Big((T^{a_1n }x, \ldots , T^{a_sn}x), (T^{p_{s+1}(n+j)}x,\ldots, T^{p_d(n+j)}x)_{j=-k}^k\Big)\\
     & \in U_1\times\ldots \times U_s\times \prod_{j=-k}^k(U_j^{(s+1)}\times U_j^{(s+2)}\times \ldots \times U_j^{(d)})
   \end{split}
\end{equation*}
It follows that $\sigma^n\xi_x^\A\in \widetilde{U}$. Thus $W_{x}^\A$ is ${\pi^*}^{\infty}$-saturated.

\medskip

\medskip

With the same proof of (1) using Theorem \ref{thm-poly-sat2} instead of Theorem \ref{thm-poly-sat1}, we can prove (2).
\end{proof}




\subsection{Recurrent properties related to $M_\infty(X,\A)$}\
\medskip

The following theorem reveals the connection of the denseness of minimal points of $M_\infty(X,\A)$ and recurrence times of
polynomials.
\begin{thm}\label{Thm-equivalence1}
Let $(X,T)$ be a minimal t.d.s.
Then the following statements are equivalent:

\begin{enumerate}
\item  For any family $\A=\{p_1, p_2,\cdots, p_d\}$ of integral polynomials satisfying $(\spadesuit)$, \\ $(M_\infty({X},\A), \langle T^\infty, \sigma\rangle)$ is an M-system. 

\item For any integral polynomials $p_1,\ldots, p_d$ with $p_i(0)=0$, $1\le i\le d$, we have that for each $x\in X$ and any neighbourhood $U$ of $x$
$$N^{\Z^2}_{p_1,\ldots,p_k}(x,U)=\{(m,n)\in\Z^2: T^{m+p_1(n)}x\in U, \ldots, T^{m+p_d (n)}x\in U\}$$
is piecewise syndetic in $\Z^2$.
\end{enumerate}
\end{thm}

\begin{proof}
First we show that (1) implies (2).
Without loss of generality, we assume that $p_1,p_2,\cdots, p_d$ are essentially distinct.
Let  $\A=\{p_1, p_2, \ldots, p_s, p_{s+1}, \ldots, p_d\}$ with $p_i(n)=a_in$, $1\le i\le s$ and $\deg (p_j)\ge 2$, $s+1\le j\le d$.

\medskip

If $\A$ satisfies the condition ($\spadesuit$), then
by (1) $(M_\infty({X},\A), \langle T^\infty, \sigma\rangle)$ is an $M$-system, and for each $x\in X$, $\xi_x^\A$ is a transitive point of $(M_\infty({X},\A), \langle T^\infty, \sigma\rangle)$. For a given $x\in X$ and a neighborhood $U$ of $x$, by the definition we have
\begin{equation*}
  \begin{split}
     N^{\Z^2}_{p_1,\ldots,p_k}(x,U)& =\{(m,n)\in\Z^2: T^{m+p_1(n)}x\in U, \ldots, T^{m+p_k(n)}x\in U\}\\
     &= \{(m, n)\in\Z^2: (T^\infty)^m\sigma^n \xi_x^\A\in \widetilde{U}\},
   \end{split}
\end{equation*}
where $\widetilde{U}=U^s\times \Big(\prod_{j=-\infty}^{-1} X^{d-s}\times  U^{d-s} \times \prod_{j=1}^\infty X^{d-s}\Big)\subseteq X^s\times (X^{d-s})^{\Z}$.

Since $(M_\infty({X},\A), \langle T^\infty, \sigma\rangle)$ is an $M$-system, by Proposition \ref{prop-M}, $\xi_x^\A$ is piecewise syndetic recurrent and hence $ N^{\Z^2}_{p_1,\ldots,p_k}(x,U)$ is a piecewise syndetic subset of $\Z^2$.

If $\A$ does not satisfy the condition ($\spadesuit$), then
by Lemma \ref{ww=dem}, there are integral polynomials $\{q_1,\ldots, q_r\}$ satisfying the condition ($\spadesuit$) and $\{p_{s+1},\ldots, p_d\}\subseteq \{q_i^{[j]}=q_i(n+j)-q_i(j): 1\le i\le r, j\in \Z \}.$ Thus there is some $l\in \N$ such that
$$\{p_{s+1},\ldots, p_d\}\subseteq \{q_i^{[j]}=q_i(n+j)-q_i(j): 1\le i\le r, -l\le j\le l\}.$$
Then $\A'=\{a_1n, \ldots, a_sn, q_1(n),\ldots, q_r(n)\}$ satisfies the condition ($\spadesuit$).

By (1), $(M_\infty({X},\A'), \langle T^\infty, \sigma\rangle)$ is an $M$-system, and for each $x\in X$, $\xi_x^{\A'}$
is a transitive point of $(M_\infty({X},\A'), \langle T^\infty, \sigma\rangle)$.
For a given $x\in X$ and a neighborhood $U$ of $x$, let
{\small$$\widetilde{U}=U^s\times \Big(\prod_{j=-\infty}^{-l-1}X^r\times  \prod_{j=-l}^l(T^{q_1(j)}U\times T^{q_2(j)}U\times \cdots \times T^{q_r(j)}U)\times \prod_{j={l+1}}^\infty X^{r}\Big)\subseteq X^s\times (X^{r})^{\Z}.$$}
It follows that
\begin{equation*}
  \begin{split}
     &\quad \{(m, n)\in\Z^2: (T^\infty)^m\sigma^n \xi_x^{\A'}\in \widetilde{U}\}\\
     &= \{(m, n)\in\Z^2: T^{m+a_1n}x,\ldots, T^{m+a_sn}x\in U; T^{m+q_i(n+j)}x\in T^{q_i(j)}U, 1\le i\le r, -l \le j\le l\}\\
     &\subseteq\{(m,n)\in\Z^2: T^{m+p_1(n)}x\in U, \ldots, T^{m+p_d(n)}x\in U\}=N^{\Z^2}_{p_1,\ldots,p_d}(x,U).
   \end{split}
\end{equation*}
Since $\{(m, n)\in\Z^2: (T^\infty)^m\sigma^n \xi_x^{\A'}\in \widetilde{U}\}$ is piecewise syndetic by Proposition \ref{prop-M}, so is  $N^{\Z^2}_{p_1,\ldots, p_d}(x,U)$.

\medskip

Now we show that (2) implies (1). By Proposition \ref{prop-M}, we need to show that for all $x\in X$, $\xi_x^\A$ is piecewise syndetic recurrent, that is, for any neighbourhood $\widetilde{U}$ of $\xi_x^\A$, $\{(m, n)\in\Z^2: (T^\infty)^m\sigma^n \xi_x^{\A}\in \widetilde{U}\}$ is piecewise syndetic in $\Z^2$. Assume that
{\small $$\widetilde{U}=U_1\times\ldots \times U_s\times \Big(\prod_{j={-\infty}}^{-(k+1)} {X}^{d-s}\times \prod_{j=-k}^k(U_j^{(s+1)}\times U_j^{(s+2)}\times \cdots \times U_j^{(d)})\times \prod_{j={k+1}}^\infty {X}^{d-s} \Big),$$}
where $k\in \N$, and $U_i, U_j^{(l)}$ are non-empty open subsets of $X$. Since
$$\xi_x^\A=\Big((x, \ldots , x), (T^{p_{s+1}(j)}x,\ldots, T^{p_d(j)}x)_{j\in {\Z}}\Big)\in \widetilde{U},$$ we have $x\in U_i, T^{p_t(j)}x\in U_j^{(t)}, 1\le i\le s, s+1\le t\le d, -k\le j\le k$. Thus $x\in \bigcap _{i=1}^s U_i \cap \bigcap_{s+1\le t\le d,\atop -k\le j\le k} T^{-p_t(j)}U_j^{(t)}$.
Choose a neighbourhood $U$ of $x$ such that
$$x\in U\subseteq \bigcap _{i=1}^s U_i \cap \bigcap_{s+1\le t\le d,\atop -k\le j\le k} T^{-p_t(j)} U_j^{(t)}.$$
Let $\widetilde{\A}=\{a_in, p_t(n+j)-p_t(j): 1\le i\le s, s+1\le t\le d, -k\le j\le k \}$. Since $\A$ satisfies $(\spadesuit)$, polynomials of $\widetilde{\A}$ are essentially distinct. By (2)
$$N^{\Z^2}_{\widetilde{\A}}(x,U)=\{(m,n)\in\Z^2: T^{m+a_in}x\in U, T^{m+p_t(n+j)-p_t(j)}x\in U, 1\le i\le s, s+1\le t\le d, -k\le j\le k  \}$$
is piecewise syndetic in $\Z^2$. For all $(m,n)\in N^{\Z^2}_{\widetilde{\A}}(x,U)$,
$$T^{m+a_in}x\in U\subset U_i,\quad  T^{m+p_t(n+j)}x\in T^{p_t(j)}U\subseteq U_j^{(t)}, $$
where $1\le i\le s, s+1\le t\le d, -k\le j\le k .$
Hence
\begin{equation*}
   (T^\infty)^m{\sigma}^n\xi_x^\A =\Big((T^{m+a_1n }x, \ldots , T^{m+a_sn}x), (T^{m+p_{s+1}(n+j)}x,\ldots, T^{m+p_d(n+j)}x)_{j\in {\Z}}\Big)\in \widetilde{U}.
\end{equation*}
That is,
$$N^{\Z^2}_{\widetilde{\A}}(x,U)\subseteq \{(m, n)\in\Z^2: (T^\infty)^m\sigma^n \xi_x^{\A}\in \widetilde{U}\},$$
which means $\{(m, n)\in\Z^2: (T^\infty)^m\sigma^n \xi_x^{\A}\in \widetilde{U}\}$  is piecewise syndetic in $\Z^2$.
The proof is complete.
\end{proof}


Similarly, we have

\begin{thm}\label{Thm-equivalence2}
Let $(X,T)$ be a minimal t.d.s.
Then the following statements are equivalent:

\begin{enumerate}
\item  There is a dense $G_\d$ subset $X_0$ of $X$ such that for each $x\in X_0$ and for any family $\A=\{p_1, p_2,\cdots, p_d\}$ of integral polynomials satisfying $(\spadesuit)$,  $(W_x^\A,\sigma)$ is an M-system. 

\item There is a dense $G_\d$ subset $X_0$ of $X$ such that for each $x\in X_0$ and  for any integral polynomials $p_1,\ldots, p_d$ with $p_i(0)=0$, $1\le i\le d$, we have that for any neighbourhood $U$ of $x$
$$N_{p_1,\ldots,p_k}(x,U)=\{n\in\Z: T^{p_1(n)}x\in U, \ldots, T^{p_k(n)}x\in U\}\in \F_{ps}.$$
\end{enumerate}
\end{thm}

\section{$N_\infty(X,\A)$ for nilsystems and some parts of Theorem C}\label{Section-nil}

In this section we study the properties of $N_\infty(X,\A)$ when $(X,T)$ is a nilsystem, and
give the proof of first and second parts of Theorem C.

\subsection{Nilsystems}\
\medskip

In this subsection, we work on nilmanifolds. All results remain true for $\infty$-step pro-nilsystems. The main results of this section are that if $(X, T )$ is a minimal pro-nilsystem, then $(M_\infty(X,\A), \langle T^{\infty},{\sigma} \rangle)$ is a minimal pro-nilsystem, and for each $x\in X$, $(W_x^\A, {\sigma})$ is a minimal pro-nilsystem.

\subsubsection{Polynomial sequences on $G$}

Let $G$ be a nilpotent Lie group, $\Gamma$ be a discrete cocompact subgroup of $G$ and $X$ be the
compact nilmanifold $G/\Gamma$. Recall that $G$ acts on $X$ by left translations: for $a\in G$ and $x=b\Gamma\in X$ one defines $ax=(ab)\Gamma$.
\begin{de}
We will say that a mapping $g: \Z^d\rightarrow G$ is {\em polynomial} if $g$ can be written in the form
$$g({\bf n})=a_1^{p_1({\bf n})}\cdots a_m^{p_m({\bf n})},$$
where $a_1,\ldots, a_m\in G$ and $p_1,\ldots, p_m: \Z^d\rightarrow \Z$ are polynomial
mappings. Such a mapping will also be called a {\em polynomial action of $\Z^d$ on $X$
by translations}.
\end{de}

On the other hand, a {\em homomorphism $\phi: \Z^d\rightarrow G$} will be referred to as
a linear action.

\begin{thm}\cite[Subsection 2.11]{Leibman052}\label{thm-Leibman-poly-nil}
Let $g: \Z^d\rightarrow G$ be a polynomial mapping and let $x\in X=G/\Gamma$.
There exists a nilpotent Lie group $\widehat{G}$ with a discrete cocompact subgroup $\widehat{\Gamma}$, a continuous surjective map $\eta: \widehat{X}=\widehat{G}/\widehat{\Gamma}\rightarrow X$ , a point $\hat{x}\in \widehat{X}$ and a homomorphism $\phi: \Z^d\rightarrow \widehat{G}$ such that
$$g({\bf n})x=\eta (\phi({\bf n})\hat{x}), \ \forall \ {\bf n}\in \Z^d.$$
Moreover, for any open neighbourhood $U$ of $x$,
$$\{{\bf n}\in \mathbb{Z}^d:g({\bf n})x\in U\}=\{{\bf n}\in \mathbb{Z}^d:\phi({\bf n})\hat{x}\in \eta^{-1}(U)\}$$
is a syndetic subset in $\mathbb{Z}^d$.
\end{thm}


The following theorem is a folklore, and we include a proof for completeness. First we need a lemma.
For the definition of $\RP^{[s]}(X,G)$ see \cite{GGY}.
\begin{lem}\label{lift-something}
Let $(X,G)$ be minimal with $G$ being abelian and $s\in\N$. Assume that $\pi: (X,G)\rightarrow (Y,G)$ is a factor map.
Then $(Y,G)$ is an $s$-step pro-nilsystem if and only if $\RP^{[s]}(X,G)\subset R_\pi=\{(x,x'): \pi(x)=\pi(x')\}$.
\end{lem}
\begin{proof} The case when $G=\Z$ was proved in \cite[Theorem 3.11]{SY}. We follow the proof closely.

Assume that $(Y,G)$ is an $s$-step pro-nilsystem, then by \cite[Theorem 1.30]{GMV20} $\RP^{[s]}(Y,G)$ $=\D_x$.
This implies that $\RP^{[s]}(X,G)\subset R_\pi$ by \cite[Theorem 6.1]{GGY}.

Conversely, assume that $\RP^{[s]}(X,G)\subset R_\pi$. If $(Y, G)$ is not an s-step pro-nilsystem, then by \cite[Theorem 1.30]{GMV20}
$\RP^{[s]}(Y,G)\not=\D_Y$. Let $(y_1, y_2) \in \RP^{[s]}(Y,G)\setminus\D_Y$. Now by \cite[Theorem 6.1]{GGY} again, there are $x_1, x_2 \in X$
such that $(x_1, x_2)\in \RP^{[s]}(X,G)$ with $(\pi\times \pi)(x_1, x_2) = (y_1, y_2)$. Since $\pi(x_1) = y_1 \not= y_2= \pi(x_2),$
$(x_1, x_2)\not\in R_\pi$. This means that $\RP^{[s]}(X,G)\not\subset R_\pi$, a contradiction!
\end{proof}

\begin{thm} \label{thm-nil-top-factor}
Let $(X,G)$ be a minimal $s$-step pro-nilsystem with $G$ being abelian and $s\in\N$. Assume that $\pi: (X,G)\rightarrow (Y,G)$ is a factor map.
Then $(Y,G)$ is an $s$-step pro-nilsystem.
\end{thm}
\begin{proof}
\cite[Theorem 1.30]{GMV20} states that a minimal t.d.s. $(X,T)$ is an $s$-step pro-nilsystem if and only if the regionally proximal relation of order $s$ is the diagonal $\D_X$. 
Thus by the assumption, we have $\RP^{[s]}(X,G)=\D_X\subset R_\pi$.
By Lemma \ref{lift-something}, $(Y,G)$ is an $s$-step pro-nilsystem.
\end{proof}


\subsubsection{$N_\infty(X,\A)$ for a minimal nilsystem}\
\medskip

Let $(X = G/\Gamma,  T )$ be a nilsystem with $T$ the translation by the element $\tau \in G$.
Let $x=g\Gamma\in X$ and $d\in \N$. Let $\A= \{p_1,  p_2, \cdots,  p_d \}$ be a family of non-constant essentially distinct integral polynomials with $p_1(0)=\cdots =p_d(0)=0$. Recall that
\begin{equation}\label{}
  \omega_x^\A=\Big((\tau^{p_1(n)}x, \tau^{p_2(n)}x, \ldots, \tau^{p_d(n)}x )\Big)_{n\in \Z}\in (X^d)^{\Z}.
\end{equation}
Note that for $n\in \Z$, we have $(\sigma\omega^\A_x)(n)= (\tau^{p_1(n+1)}x, \tau^{p_2(n+1)}x, \ldots, \tau^{p_d(n+1)}x )$.

Recall that
\begin{equation}\label{}
  N_\infty(X,\A)=\overline{\bigcup\{\sigma^n\omega^\A_x:n\in \Z, x\in X\}}\subseteq (X^d)^{\Z}.
\end{equation}

Now we show

\begin{thm}\label{thm-nil-ergodic}
Let $(X, T)$ be a minimal pro-nilsystem. 
Let $\A= \{p_1,  p_2, \ldots,  p_d \}$ be a family of non-constant essentially distinct integral polynomials with $p_1(0)=\cdots =p_d(0)=0$. Then we have
\begin{enumerate}
  \item The system $(N_\infty(X,\A), \langle T^{\infty},\sigma \rangle)$ is a minimal pro-nilsystem.

  \item For each $x\in X$, the system $(\overline{\O}(\omega_x^\A, \sigma),\sigma)$ is a
minimal pro-nilsystem.
\end{enumerate}
\end{thm}

\begin{proof}
First we assume that $(X = G/\Gamma, T )$ is a minimal nilsystem with $T$ a translation by an element $\tau \in G$.
We divide the proof into several steps.

\medskip
\noindent{\bf Step 1}. Lifting $X$ to $\widehat{X}$ via $\eta$.
\medskip

We fix a point $x_0\in X$.
Let $g: \Z^{d+1}\rightarrow G$ be a polynomial action of $\Z^{d+1}$ on $X$ defined as follows
$$g(n_0,n_1,\ldots,n_d)=\tau^{n_0+p_1(n_1)+p_2(n_2)+\cdots+ p_d(n_d)}.$$
By Theorem \ref{thm-Leibman-poly-nil}, there is a nilpotent Lie group $\widehat{G}$ with a discrete uniform subgroup $\widehat{\Gamma}$, a continuous surjective map $\eta: \widehat{X}=\widehat{G}/\widehat{\Gamma}\rightarrow X$ , a point $\hat{x}_0\in \widehat{X}$ and a homomorphism $\phi: \Z^{d+1}\rightarrow \widehat{G}$ such that
$g({\bf n})x_0=\eta (\phi({\bf n})\hat{x}_0), \ \forall \ {\bf n}\in \Z^{d+1}.$ That is, there are commutative elements $\tau_0,\tau_1, \ldots, \tau_d\in \widehat{G}$ such that
\begin{equation}\label{}
 \tau^{n_0+p_1(n_1)+p_2(n_2)+\cdots+ p_d(n_d)} x_0=\eta (\tau_0^{n_0}\tau_1^{n_1}\tau_2^{n_2}\cdots\tau_d^{n_d} \hat{x}_0), \ \forall (n_0,n_1,\ldots, n_d)\in \Z^{d+1}.
\end{equation}
In particular, we have (see (\ref{ite-n}) for the definition of $\eta^{(d+1)}$)
$$\eta^{(d+1)}\big(\tau_0^n \hat{x}_0, \tau_1^n \hat{x}_0, \tau_2^n \hat{x}_0,\ldots, \tau_d^n \hat{x}_0 \big)=(\tau^n x_0, \tau^{p_1(n)}x_0, \tau^{p_2(n)}x_0, \ldots, \tau^{p_d(n)}x_0 ), \ \forall n\in \Z.$$


\medskip
\noindent {\bf Step 2}. Defining a nilsystem $(\Omega, \langle \Theta_0, \Theta \rangle)$ in $(\widehat{X}^d, \langle \Theta_0, \Theta \rangle)$.
\medskip

Let
$$\widehat{{\bf x}}_0=(\hat{x}_0,\hat{x}_0,\ldots, \hat{x}_0)\in \widehat{X}^d; \theta_0=(\tau_0,\tau_0,\ldots, \tau_0)\ \text{and}\  \theta=(\tau_1, \tau_2,\ldots, \tau_d)\in \widehat{G}^d.$$
And let $\Theta_0, \Theta: \widehat{X}^d\rightarrow \widehat{X}^d$ be defined as the translations by $\theta_0, \theta$ respectively:
$$\Theta_0: {\bf z}\mapsto \theta_0{\bf z}, \quad \Theta: {\bf z}\mapsto \theta{\bf z}, \ \forall \ {\bf z}\in \widehat{X}^d.$$
Then $(\widehat{X}^d, \langle \Theta_0, \Theta \rangle)$ is a t.d.s.
Put
$$\Omega=\overline{\O(\widehat{\bf x}_0, \langle \Theta_0, \Theta \rangle)}=\overline{\{\theta_0^m\theta^n \widehat{\bf x}_0: (m,n)\in \Z^2\}}\subseteq \widehat{X}^d,$$
and by Theorem \ref{thm-ParryLeibman}, $(\Omega, \langle \Theta_0, \Theta \rangle)$ a minimal nilsystem.
Let $$\widetilde{X}=\overline{\O (\widehat{\bf x}_0, \Theta_0)}=\overline{\{\theta_0^n\widehat{\bf x}_0: n\in \Z\}}\subseteq \Omega.$$
Then by Theorem \ref{thm-ParryLeibman}, $(\widetilde{X}, \Theta_0)$
is also a minimal nilsystem, and by the definition,
$\widetilde{X}\subseteq \Delta_d(\widehat{X})=\{\hat{\bf x}=(\hat{x},\hat{x},\ldots, \hat{x})\in \widehat{X}^d: \hat{x}\in \widehat{X}\}$.
Note that by the definition, we have $$\Omega=\overline{\{\Theta^n\widetilde{X}:n\in \Z\}}.$$
Now define $\widetilde{\eta}: \widetilde{X}\rightarrow X$ such that $\widetilde{\eta}(\hat{x},\ldots,\hat{ x})=\eta(\hat{x})$
for each $\hat{\bf x}\in \widetilde{X}$.
Then we have
\begin{equation}\label{}
\widetilde{\eta}(\theta_0^n \widehat{\bf x}_0)=\eta(\tau_0^n\hat{x}_0)= \tau^n x_0, \ \forall n\in \Z.
\end{equation}
Since $(\widetilde{X}, \Theta_0)$, $(X,T)$ are minimal and $\eta$ is continuous, $\widetilde{\eta}: \widetilde{X}\rightarrow X$ is well-defined.
Moreover, it is a continuous surjective map.

\medskip
\noindent{\bf Step 3}. Defining nilsystems $(\widehat{\Omega}_{\hat{x}},\sigma)$ and $\big(\widehat{\Omega}, \langle \Theta_0^\infty, \sigma \rangle\big)$ in $\Omega^\Z$.
\medskip

For each $\hat{x}\in \widehat{X}$ let
$$\zeta_{\hat{x}}=\Big((\tau_1^n \hat{x}, \tau_2^n \hat{x},\ldots, \tau_d^n \hat{x})\Big)_{n\in \Z}=(\Theta^n  \widehat{\bf x})_{n\in \Z}\in \Omega^\Z\subseteq (\widehat{X}^d)^\Z.$$
Put $$\Omega_{\hat{ x}}=\overline{\O(\hat{\bf x}, \Theta)}\subseteq \Omega\ \ \text{and}\ \ \widehat{\Omega}_{\hat{ x}}=\overline{ \O(\zeta_{\hat{x}},\sigma)}\subset \Omega^\Z.$$ Then by Theorem \ref{thm-ParryLeibman} $(\Omega_{\hat{ x}},\Theta)$ is a minimal nilsystem.
As in Subsection \ref{subsection-linear-one}  we define
\begin{equation}\label{}
  \phi_{\hat{x}}:(\Omega_{\hat{ x}},\Theta) \rightarrow (\widehat{\Omega}_{\hat{x}},\sigma), \ {\bf y}\mapsto \big( \Theta^n{\bf y}\big)_{n\in \Z}, \ \forall {\bf y}\in \Omega_{\hat{x}}.
\end{equation}
Since $\Theta$ is continuous, so is $\phi_{\hat{x}}$. It is clear that $\phi_{\hat{x}}$ is one-to-one. Now we show it is surjective. Let $({\bf y}_n)_{n\in \Z}\in \widehat{\Omega}_{\hat{x}}$. Then there is a sequence $\{n_i\}_{i\in \Z}\subseteq \Z$ such that $\sigma^{n_i}\zeta_{\hat{x}}\to ({\bf y}_n)_{n\in \Z}, i\to\infty$. Note that
$$\sigma^{n_i}\zeta_{\hat{x}}=(\Theta^{n+n_i} \widehat{\bf x})_{n\in \Z}=(\ldots, \Theta^{-1}(\theta^{n_i} \widehat{\bf x}), \underset{\bullet}{\Theta^{n_i} \widehat{\bf x}}, \Theta(\Theta^{n_i} \widehat{\bf x}),\ldots).$$
Since $\sigma^{n_i}\zeta_{\hat{x}}\to ({\bf y}_n)_{n\in \Z}, i\to\infty$, at $0^{th}$-coordinate we have that $\Theta^{n_i}\widehat{\bf x}\to {\bf y}_0, i\to\infty$. Thus
$$\sigma^{n_i}\zeta_{\hat{x}}=(\Theta^{n+n_i} \widehat{\bf x})_{n\in \Z}\to (\Theta^n {\bf y}_0)_{n\in \Z}, i\to\infty.$$
So $({\bf y}_n)_{n\in \Z}=(\Theta^n {\bf y}_0)_{n\in \Z}=\phi_{\hat{x}}({\bf y}_0)$. That is, $ \phi_{\hat{x}}:(\Omega_{\hat{ x}},\Theta) \rightarrow (\widehat{\Omega}_{\hat{x}},\sigma)$ is an isomorphism between them. Now $(\widehat{\Omega}_{\hat{x}},\sigma)$ is  isomorphic to $(\Omega_{\hat{ x}},\Theta)$, and so it is a minimal nilsystem.

Set
$$\widehat{\Omega}=\overline{\bigcup_{\hat{{\bf x}} \in \widetilde{X}} \widehat{\Omega}_{\hat{x}}}.$$
Note that $\displaystyle \Omega=\overline{\{\Theta^n\widetilde{X}:n\in \Z\}}=\overline{\bigcup_{\hat{\bf x}\in \widetilde{X}} \Omega_{\hat{x}}}$.

Now we show that $\big(\widehat{\Omega}, \langle \Theta_0^\infty, \sigma \rangle\big)$ is topologically isomorphic to $(\Omega, \langle \Theta_0, \Theta \rangle)$. The proof is similar to the above. Let
\begin{equation}\label{}
  \phi:(\Omega, \langle \Theta_0, \Theta \rangle) \rightarrow \big(\widehat{\Omega}, \langle \Theta_0^\infty, \sigma \rangle\big), \ {\bf y}\mapsto \big( \Theta^n{\bf y}\big)_{n\in \Z}, \ \forall {\bf y}\in \Omega.
\end{equation}
Since $\Theta$ is continuous, so is $\phi$. It is clear that $\phi$ is one-to-one. Now we show it is surjective. Let $({\bf z}_n)_{n\in \Z}\in \widehat{\Omega}$. Then there is a sequence $\big( \Theta^n{\bf y}^i\big)_{n\in \Z}\in  \widehat{\Omega}_{\hat{x_i}} , i\in \N$ such that $\big( \Theta^n{\bf y}^i\big)_{n\in \Z} \to ({\bf z}_n)_{n\in \Z}, i\to\infty$. Note that
$$\big( \Theta^n{\bf y}^i\big)_{n\in \Z} =(\ldots, \Theta^{-1}({\bf y}^i), \underset{\bullet}{{\bf y}^i}, \Theta({\bf y}^i),\ldots).$$
Since $\big( \Theta^n{\bf y}^i\big)_{n\in \Z} \to ({\bf z}_n)_{n\in \Z}, i\to\infty$, at $0^{th}$-coordinate we have that ${\bf y}^i \to {\bf z}_0, i\to\infty$. Thus
$$\big( \Theta^n{\bf y}^i\big)_{n\in \Z} \to (\Theta^n {\bf z}_0)_{n\in \Z}, i\to\infty.$$
So $({\bf z}_n)_{n\in \Z}=(\Theta^n {\bf z}_0)=\phi({\bf z}_0)$. That is, $\phi:(\Omega, \langle \Theta_0, \Theta \rangle) \rightarrow \big(\widehat{\Omega}, \langle \Theta_0^\infty, \sigma \rangle\big)$ is an isomorphism. Thus $\big(\widehat{\Omega}, \langle \Theta_0^\infty, \sigma \rangle\big)$ is isomorphic to $(\Omega, \langle \Theta_0, \Theta \rangle)$, and so it is a minimal nilsystem.

\medskip
\noindent{\bf Step 4}. Showing that $\big(N_{\infty}(X,\A), \langle T^\infty, \sigma\rangle\big)$ is a factor of $\big(\widehat{\Omega}, \langle \Theta_0^\infty, \sigma \rangle\big)$. 
\medskip

Define a map
$$\widehat{\eta}: \big(\widehat{\Omega}, \langle \Theta_0^\infty, \sigma \rangle\big)\rightarrow \big(N_{\infty}(X,\A), \langle T^\infty, \sigma\rangle\big), ({\bf x}_n)_{n\in \Z}\mapsto (\eta^{(d)}{\bf x}_n)_{n\in \Z}, $$
where $({\bf x}_n)_{n\in \Z}=\Big((x_n^{(1)},\ldots, x_n^{(d)})\Big)_{n\in \Z}\in \widehat{\Omega}$ and
$$(\eta^{(d)}{\bf x}_n)_{n\in \Z}=\Big((\eta(x_n^{(1)}),\ldots,\eta( x_n^{(d)}))\Big)_{n\in \Z} \in (X^d)^\Z.$$ Since $\eta$ is continuous, so is $\widehat{\eta}$.
Note that for each $x\in X$ and each $\hat{\bf x}=(\hat{x},\hat{x},\ldots, \hat{x})\in \widetilde{X}$ with $\widetilde{\eta}(\hat{\bf x})=x$
we have
$$\widehat{\eta}(\zeta_{\hat{x}})=\widehat{\eta}\Big(\Big((\tau_1^n \hat{x}, \tau_2^n \hat{x},\ldots, \tau_d^n \hat{x})\Big)_{n\in \Z}\Big)=
\Big((\tau^{p_1(n)}x, \tau^{p_2(n)}x,\ldots, \tau^{p_d(n)}x)\Big)_{n\in \Z}=\omega_x^\A.$$
It follows that $\widehat{\eta}(\widehat{\Omega})=N_\infty(X,\A)$. That is, $\widehat{\eta}: \widehat{\Omega}\rightarrow N_{\infty}(X,\A)$ is a well-defined continuous map.
It is easy to check that $\widehat{\eta}\circ \sigma=\sigma\circ \widehat{\eta}$ and $\widehat{\eta}\circ \Theta_0^\infty=T^\infty\circ \widehat{\eta}$.
Thus $\widehat{\eta}: \big(\widehat{\Omega}, \langle \Theta_0^\infty, \sigma \rangle\big)\rightarrow \big(N_{\infty}(X,\A), \langle T^\infty, \sigma\rangle\big)$ is a factor map.

\medskip

Since $\big(\widehat{\Omega}, \langle \Theta_0^\infty, \sigma \rangle\big)$ is a minimal nilsystem by Step 3,
$\big(N_{\infty}(X,\A), \langle T^\infty, \sigma\rangle\big)$ is a minimal pro-nilsystem by Theorem \ref{thm-nil-top-factor}.
Thus we have shown (1).

\medskip

Let $x\in X$.
When we restrict the map $\widehat{\eta}$ on $\widehat{\Omega}_{\hat{x}}$, we have the  factor map
$$\widehat{\eta}: (\widehat{\Omega}_{\hat{x}},\sigma) \rightarrow (\overline{\O}(\omega_x^\A,\sigma),\sigma).$$
Thus as a factor of  $(\widehat{\Omega}_{\hat{x}},\sigma)$, by Theorem \ref{thm-nil-top-factor},
$(\overline{\O}(\omega^\A_x,\sigma),\sigma)$ is a pro-nilsystem, and we have shown (2).

\medskip
Now we assume that $(X,T)$ is a minimal pro-nilsystem. Then $(X,T)$ is an inverse limit of nilsystems $(Y_i=G_i/\Gamma_i,T_i)$.
It is easy to see that
$$N_\infty(X,\A)=\underset{\longleftarrow}\lim \ N_\infty (Y_i,\A),$$
which implies that $(N_{\infty}(X,\A), \langle T^\infty, \sigma\rangle)$ is a minimal pro-nilsystem by what we just proved.
The same argument applies to $(\overline{\O}(\omega_x^\A, \sigma),\sigma)$. The whole proof is complete.
\end{proof}

\subsubsection{A corollary}

By Subsection \ref{subsection-linear-terms},
$ (W_x^\A, \widetilde{\sigma})$ is isomorphic to $(\overline{\O}(\w_x^\A,\sigma), \sigma),$
and
$ (M_\infty(X,\A),\langle T^\infty, \widetilde{\sigma}\rangle)$ is isomorphic to $(N_\infty(X,\A), \langle T^{\infty}, \sigma \rangle).$
Thus we have that
\begin{cor}\label{cor-nil}
Let $(X, T )$ be a minimal pro-nilsystem. Then so is $(M_\infty(X,\A), \langle T^{\infty},\widetilde{\sigma} \rangle)$, and for each $x\in X$, the system $(W_x^\A, \widetilde{\sigma})$ is a
minimal pro-nilsystem too.
\end{cor}



\subsection{$M_\infty(X,T)$ for distal t.d.s.}\
\medskip

\subsubsection{The statements}

The main result of this subsection is as follows:

\begin{thm}\label{thm-distal}
Let $(X,T)$ be minimal and distal.

\begin{enumerate}
  \item There is residual subset $X_0$ of $X$ such that for all $x\in X_0$ and each family $\A=\{p_1,\ldots, p_d\}$ satisfying $(\spadesuit)$, we have $(W_x^\A,\sigma)$ is an $M$-system. 

  \item For each family $\A=\{p_1,\ldots, p_d\}$ satisfying $(\spadesuit)$,$(M_\infty(X,\A), \langle T^\infty,\sigma \rangle)$ is an M-system. 
\end{enumerate}
\end{thm}

Then by Theorems \ref{Thm-equivalence1} and \ref{thm-distal}, 
we have

\begin{cor}\label{cor-ThmC-2}
Let $(X,T)$ be a minimal distal t.d.s.  Then there is a dense $G_\delta$ set $X_0$ such that for each $x\in X_0$, each neighbourhood $U$ of $x$, and for any non-constant integral polynomials $p_1, \ldots, p_d$ with $p_{i}(0)=0$, $i=1,2,\ldots, d$,
$$N_{\{p_1,\ldots,p_d\}}(x,U)=\{n\in\Z: T^{p_1(n)}x\in U, \ldots, T^{p_d(n)}x\in U\}\in \F_{ps}.$$
\end{cor}


Note that Corollary \ref{cor-ThmC-2} is the second part of Theorem C. 

\medskip

Now we are going to show Theorem \ref{thm-distal}. Let $(X,T)$ be a minimal t.d.s.
Put $\G=\langle T^\infty,\sigma \rangle$. If $(X,T)$ is a pro-nilsystem, then so is
$(M_\infty(X_\infty,\A), \G)$ by Corollary \ref{cor-nil}.
But in general, when $(X,T)$ is distal, $(M_\infty(X,\A),\G)$ may not be distal. So it is not obvious that $(M_\infty(X,\A), \G)$ is an $M$-system.
To prove Theorem \ref{thm-distal}, we need some preparations.

\subsubsection{Distal extensions and group extensions}
Let $(X,T)$ and $(Y,T )$ be t.d.s. and let $\pi: X \to Y$ be a factor map.
One says that $\pi$ is a {\em weak group extension} if there exists a
topological group $K$ such that the following
conditions hold:
\begin{enumerate}
\item $K$ acts continuously on $X$ from the right: the
right action $X\times K\rightarrow X$, $(x,k)\mapsto xk$ is
continuous and $T(xk)=(Tx)k$ for any $x\in X, k\in K$;
\item the fibers of $\pi$ are the $K$-orbits in $X$, i.e.
$\pi^{-1}(\{\pi(x)\})=xK$ for any $x\in X$.
\end{enumerate}
When $K$ is a compact Hausdorff topological group, then we say that $\pi$ is a {\em group extension}.

Note that a group extension is equicontinuous, and an equicontinous extension is distal.

\medskip

Let
${\rm Aut} (X,T)$ be the group of automorphisms of the t.d.s. $(X,T)$, that is, the group of all self-homeomorphisms $\psi$ of $X$ such that $\psi\circ T=T\circ \psi$. For an extension $\pi: (X,T)\rightarrow (Y,T)$, let $${\rm Aut}_\pi(X,T)=\{\chi\in {\rm Aut}(X,T): \pi\circ \chi=\pi\},$$ i.e., the collection of elements of ${\rm Aut}(X,T)$ mapping every fiber of $\pi$ into itself.

Let $\pi: X\rightarrow Y$ be a weak group extension as above. Then for each $k\in K$, $\phi_k: X\rightarrow X, x\mapsto xk$ defines an element of ${\rm Aut}_\pi(X,T)$.

\begin{thm}\cite[Chapter V, (4.21), (5.16)]{Vr} \label{distal-group}
Let $\pi: X \rightarrow Y$ be an extension of minimal t.d.s. $(X,T)$ and $(Y,T)$. Then
\begin{enumerate}
  \item $\pi$ is distal if and only if $\pi$ is a factor of weak group extension $\phi$, i.e. there are minimal t.d.s. $(Z,T)$ and extensions $\pi':Z\rightarrow X,
  \phi: Z\rightarrow Y$ such that $\phi=\pi\circ \pi'$ and $\phi$ is a weak group extension.

  \item $\pi$ is equicontinuous if and only if $\pi$ is a factor of group extension $\phi$.
\end{enumerate}
\end{thm}


\subsubsection{Ideas of the proof of Theorem \ref{thm-distal}}
Recall that we put $\G=\langle T^\infty,\sigma \rangle$. Now we show that for $\A=\{p\}$ with $p(n)=n^2$, $(M_\infty(X,\A), \G)$ is an $M$-system. The proof for the general case follows by the same arguments. 
Moreover, it reveals the ideas when a distal
extension is replaced by a RIC extension in Section \ref{Section-Z2-recurrence}.

\medskip

Let $(X,T)$ be a minimal distal t.d.s.  and let $\pi:X\rightarrow X_\infty$ be the factor map from $X$ to its maximal $\infty$-step pro-nilfactor $X_\infty$. Let $\A=\{p\}$ with $p(n)=n^2$. It is clear 
$M_\infty(X,\A)\subseteq X^{\Z}$. Since $(X,T)$ is distal, $\pi$ is open. By Theorem \ref{thm-poly-saturate}-(2), $M_\infty(X,\A)$ is ${\pi}^\infty$-saturated, that is, $M_\infty(X,\A)=({\pi}^\infty)^{-1}(M_\infty(X_\infty,\A)).$

\medskip

By Theorem \ref{distal-group}, $\pi$ is a factor of a weak group extension $\phi$, i.e. there are minimal t.d.s. $(Z,T)$ and extensions $\pi':Z\rightarrow X, \phi: Z\rightarrow X_\infty$ such that $\phi=\pi\circ \pi'$ and $\phi$ is a weak group extension.
$$
  \xymatrix{
  X \ar[d]_{\pi} & Z \ar[l]_{\pi'} \ar[dl]^{\phi}      \\
  X_\infty }
  $$
Let $W=(\phi^\infty)^{-1}(M_\infty(X_\infty,\A))\subset Z^{\Z}$. Then $(W,\G)$ is a t.d.s. and it is easy to verify that we have the following commuting diagram.
$$\xymatrix{
 M_\infty(X,\A) \ar[d]_{\pi^\infty} & W \ar[l]^{(\pi')^\infty} \ar[dl]^{\phi^\infty}      \\
  M_\infty(X_\infty,\A) }
  $$

To show that $\G$-minimal points are dense in $(M_\infty(X,\A),\G)$, it suffices to show that $\G$-minimal points are dense in $(W,\G)$. For this purpose, we will show that for each  point ${\bf y}\in M_\infty(X_\infty,\A)$, and each ${\bf x}\in (\phi^\infty)^{-1}({\bf y})$, there is a sequence $\{{\bf x}_{k}\}_{k=1}^\infty \subseteq W$ such that each ${\bf x}_k$ is $\G$-minimal and ${\bf x}_k\rightarrow {\bf x}, k\to\infty$.

\medskip

Now fix a point ${\bf y}=(y_i)_{i=-\infty}^\infty \in M_\infty(X_\infty, \A)$ and ${\bf x}=(x_i)_{i=-\infty}^\infty \in (\phi^\infty)^{-1}({\bf y})$. Since ${\bf y}$ is $\G$-minimal by Theorem \ref{thm-nil-ergodic}, there exists a $\G$-minimal point (Lemma \ref{den-minimal})
$${\bf x}'=(x'_i)_{i=-\infty}^\infty \in (\phi^\infty)^{-1}({\bf y}).$$
As $\phi$ is a weak group extension, there exists a
topological group $K$ acts continuously on $Z$ from the right, and the fibers of $\phi$ are the $K$-orbits in $Z$, i.e.
$\phi^{-1}(\{\phi(z)\})=zK$ for any $z\in Z$.
For all $i\in{\Z}$, there is some $k_i\in K$ such that $x_i'k_i=x_i$.
Let $\phi_i: Z\rightarrow Z, z\mapsto zk_i$. Then $\phi_i\in {\rm Aut}_\phi(Z)$ such that ${x}_i=\phi_i(x_i')$. For each $k\in \N$, let
\begin{equation*}
  \begin{split}
    \Phi_k& =(\phi_{-k}\times \phi_{-k+1}\times \cdots \times \phi_k)^\infty\\
    &=\cdots \times (\phi_{-k}\times \phi_{-k+1}\times \cdots \times \phi_k)\times (\phi_{-k}\times \phi_{-k+1}\times \cdots \times \phi_k)\times \cdots.
   \end{split}
\end{equation*}
Then
\begin{equation*}
  \begin{split}
     \Phi_k({\bf x'})&=(\ldots, \phi_k(x_{-k-1}'), {x}_{-k},\ldots,\underset{\bullet}{x_0}, \ldots, {x}_k, \phi_{-k}(x_{k+1}'),\ldots, \phi_k(x_{3k+1}'),\phi_{-k}(x_{3k+2}'),\ldots)\\
     & \rightarrow {\bf x}, \ k\to\infty,
   \end{split}
\end{equation*}

Since  $W=({\phi}^\infty)^{-1}(M_\infty(X_\infty,\A))$,
we have $\Phi_k({\bf x'})\in W$.

It is left to show that $\Phi_k({\bf x'})$ is $\G$-minimal.
For all $k\in \N$, by Lemma \ref{lem-minimal-point}, ${\bf x'}$ is a $\langle T^\infty,\sigma^{2k+1} \rangle$-minimal
point as ${\bf x'}$ is $ \langle T^\infty,\sigma \rangle $-minimal. Note that
$$\sigma^{2k+1}\Phi_k=\Phi_k \sigma^{2k+1},$$
and $T^\infty\Phi_k=\Phi_k T^\infty.$
It follows that $\Phi_k({\bf x'})$ is a $\langle T^\infty,\sigma^{2k+1} \rangle$-minimal point.
Again by Lemma \ref{lem-minimal-point}, $\Phi_k({\bf x'})$ is $\G$-minimal. Thus $(W, \G)$ is an $M$-system, and as its factor, $(M_\infty(X,\A), \G)$ is an $M$-system. This proves Theorem \ref{thm-distal}-(2) for the special case.

\medskip
To show Theorem \ref{thm-distal}-(1) for the special case, we just need to use Theorem \ref{thm-poly-saturate}-(1) and a modification of the above arguments.


\subsection{$M_\infty(X,\A)$ for weakly mixing systems}\

\subsubsection{}

For a weakly mixing minimal t.d.s. $(X,T)$, its $\infty$-step pro-nilfactor $X_\infty$ is trivial and thus by Theorem \ref{thm-poly-saturate} we have

\begin{thm} \label{thm-wm1}
Let $(X,T)$ be minimal and weakly mixing.

\begin{enumerate}
  \item There is a residual subset $X_0$ of $X$ such that for all $x\in X_0$ and each family $\A=\{p_1,\ldots, p_d\}$ satisfying $(\spadesuit)$, we have $W_x^\A=X^s\times (X^{d-s})^{\Z}$ and $(W_x^\A,\sigma)$ is transitive with a transitive point $\xi_x^\A$;

  \item For each family $\A=\{p_1,\ldots, p_d\}$ satisfying $(\spadesuit)$, we have $M_\infty(X,\A)=X^s\times (X^{d-s})^{\Z}$
  with $(M_\infty(X,\A), \G)$ being transitive,  and for each $x\in X$, $\xi_x^\A$ is a transitive point.
\end{enumerate}
\end{thm}

\subsubsection{}

\newpage

For a weakly mixing minimal t.d.s., it is easy to show systems in Theorem \ref{thm-wm1} are $M$-systems.
First we look at a special case when $\A=\{p\}$ with $p(n)=n^2$ which illustrates the ideas how to show that the minimal
points for $\sigma$ (resp. $\G$) are dense.

\medskip
Let $(X,T)$ be a weakly mixing minimal t.d.s. and $\A=\{p\}$ with $p(n)=n^2$.
Then for $x\in X$
$$\xi_x^\A=\w_x^\A=(T^{p(n)}x)_{n\in {\Z}}=(\ldots, T^{p(-1)}x, \underset{\bullet} x, T^{p(1)}x, T^{p(2)}x, \ldots).$$
By Theorem \ref{thm-wm1}, there is residual subset $X_0$ of $X$ such that for all $x\in X_0$,  $(W_x^\A,\sigma)$ is a transitive t.d.s. with a transitive point $\xi_x^\A$ and $(M_\infty(X,\A), \G)$ is a transitive t.d.s.
Since in both cases, the spaces are $X^{\Z}$, it will not be difficult to show minimal points are dense
by the following simple arguments.

For any $x_{-k},x_{-k+1}, \ldots, x_k\in X$, $k\in\N$, it is clear
$$(x_{-k},x_{-k+1},\ldots, x_k)^\infty\triangleq(\ldots, x_{-k},,\ldots, x_k,x_{-k},\ldots, \underset{\bullet} {x_0}, \ldots,x_k, x_{-k},\ldots, x_k,\ldots)$$
is a $\sigma$-periodic point. It follows that $\sigma$-periodic points are dense in $X^{\Z}$. In particular,
for $x\in X_0$, $(T^{p(-k)}x, \ldots, T^{p(-1)}x, x, T^{p(1)}x,\ldots, T^{p(k)}x)^\infty$ is $\sigma$-minimal. Thus, $(W_x^\A,\sigma)$ is an $M$-system.

\medskip
By the minimality of $(X,T)$, for each $k\in \N$, minimal points of $(X^{2k+1},T^{(2k+1)})$ are dense in $X^{2k+1}$ (Lemma \ref{den-minimal}).
Let $x_{-k}, x_{-k+1}, \ldots, x_k\in X$ such that $(x_{-k},\ldots, x_k)$ is $T^{(2k+1)}$-minimal. We claim that $(x_{-k}, \ldots, x_k)^\infty$ is minimal for $\G$. To see this fact, let
$$U=\prod_{i=-\infty}^{-j(2k+1)-k-1} X\times \prod_{i=-j(2k+1)-k}^{-1} U_i\times \underset{\bullet}{U_0} \times  \prod_{i=1}^{j(2k+1)+k} U_i \times \prod_{i=j(2k+1)+k+1}^\infty X\subset X^{\Z}$$ be an open neighborhood of $(x_{-k},\ldots, x_k)^\infty$ with $j\in \N$. Then
$$N_{T^\infty}((x_{-k},\ldots, x_k)^\infty,U)$$ is syndetic in $\Z$ by the minimality of $(x_{-k},\ldots, x_k)$ under $T^{(2k+1)}$. Thus,
$$N_{\langle T^\infty, \sigma^{(2j+1)(2k+1)} \rangle}((x_{-k},\ldots, x_k)^\infty,U)=N_{T^\infty}((x_{-k},\ldots, x_k)^\infty,U)\times \Z$$
is syndetic in $\Z^2$, which implies $N_{\langle T^\infty, \sigma \rangle}((x_{-k},\ldots, x_k)^\infty,U)$ is syndetic in $\Z^2$. Hence,
$(x_{-k},\ldots, x_k)^\infty$ is minimal under $\G$.

Particularly, for each $x\in X$ since $(T^{p(-k)}x,\ldots, T^{p(k)}x)$ is $T^{(2k+1)}$ minimal, we have that
$(T^{p(-k)}x,\ldots, T^{p(k)}x)^\infty\in M_\infty(X,\A)$ is $\G$-minimal.
 Consequently, $(M_\infty(X,\A), \G)$ is an $M$-system.


\begin{rem}
We have the following remarks
\begin{enumerate}
\item 
$M_\infty(X,\{n^2\})$ is not minimal since it contains $\Delta_\infty(X)$ which is invariant under $\langle T^\infty,\sigma \rangle$.
\item If ${\bf x}=(\ldots, x_{-1},x_0,x_1,\ldots)$ is a minimal point of $T^\infty$, $\bf x$ is not necessarily a minimal
point of $\langle T^\infty,\sigma \rangle$, say ${\bf x}=(T^{n^2} x )_{n\in {\Z} }=(\ldots, T^4x, Tx, x,Tx,T^4x,\ldots)$.
\end{enumerate}
\end{rem}

\subsubsection{}

With similar analysis in the example above, we can prove the following result.

\begin{thm} \label{thm-wm2}
Let $(X,T)$ be minimal and weakly mixing.

\begin{enumerate}
  \item There is residual subset $X_0$ of $X$ such that for all $x\in X_0$ and each family $\A=\{p_1,\ldots, p_d\}$ satisfying $(\spadesuit)$, we have $W_x^\A=X^s\times (X^{d-s})^{\Z}$ and $(W_x^\A,\sigma)$ is an $M$-system.

  \item For each family $\A=\{p_1,\ldots, p_d\}$ satisfying $(\spadesuit)$, $M_\infty(X,\A)=X^s\times (X^{d-s})^{\Z}$ and
  $(M_\infty(X,\A), \G)$ is an $M$-system.
\end{enumerate}
\end{thm}

By Theorems \ref{Thm-equivalence1} and \ref{thm-wm2}, 
we have

\begin{cor}\label{cor-ThmC-1}
Let $(X,T)$ be a weakly mixing minimal t.d.s.  Then there is a dense $G_\delta$ subset $X_0$ of $X$ such that for each $x\in X_0$, each neighbourhood $U$ of $x$, and for any non-constant integral polynomials $p_1, \ldots, p_d$ with $p_{i}(0)=0$, $i=1,2,\ldots, d$,
$$N_{\{p_1,\ldots,p_d\}}(x,U)=\{n\in\Z: T^{p_1(n)}x\in U, \ldots, T^{p_d(n)}x\in U\}\in \F_{ps}.$$
\end{cor}


Note that  Corollary \ref{cor-ThmC-1} is the first part of Theorem C. 

\section{The proofs of Theorems C and D}\label{Section-Z-poly-rec}

In this section we finish the proof of Theorem C by showing the last part of Theorem C, and prove Theorem D. 
To finish the proof of Theorem C we need

\subsection{Some lemmas}

\begin{lem}\label{wm-2}
Let $p_1, p_2$ be two essentially distinct integral polynomials of degree $\ge 2$. Then there is a positive constant $L$ depending
on $p_1,p_2$ such that if $k_1,k_2\in \mathbb{Z}$ with $|k_1-k_2|\ge L$ then
the integral polynomials $p_{i_1}(n+k_{j_1})-p_{i_1}(k_{j_1})$ and $p_{i_2}(n+k_{j_2})-p_{i_2}(k_{j_2})$ in one variable $n$  are essentially distinct for each $(i_1,j_1)\neq (i_2,j_2)\in \{1,2\}\times \{1,2\}$.
\end{lem}

\begin{proof} Let $$p_1(n)=\sum_{i=0}^{d_1} a^1_i n^i \text{ and }p_2(n)=\sum_{i=0}^{d_2} a^2_i n^i$$
with $d_1,d_2\ge 2$ and $a^1_{d_1}\neq 0,a^2_{d_2}\neq 0$. Put
$L=(\frac{1}{|a^1_{d_1}|}+\frac{1}{|a^2_{d_2}|}+1)\cdot(|a^1_{d_1-1}|+|a^2_{d_2-1}|+1).$
Then $0<L<\infty$.

Fix $k_1,k_2\in \mathbb{Z}$ with $|k_1-k_2|\ge L$.
Let $(i_1,j_1)\neq (i_2,j_2)\in \{1,2\}\times \{1,2\}$. There are three cases:

\medskip
\noindent{\bf Case 1.} $i_1=i_2=1$. In this case $j_1\neq j_2$. Note that the coefficient of $n^{d_1-1}$ in the polynomial $p_{i_1}(n+k_{j_1})-p_{i_1}(k_{j_1})$ is $a^1_{d_1}d_1 k_{j_1}+a^1_{d_1-1}$, and  the coefficient of $n^{d_1-1}$ in the polynomial $p_{i_2}(n+k_{j_2})-p_{i_2}(k_{j_2})$ is
$a^1_{d_1}d_1 k_{j_2}+a^1_{d_1-1}$. Thus since $|k_1-k_2|\ge L$ and $a^1_{d_1}\neq 0$, one has that $a^1_{d_1}d_1 k_{j_1}+a^1_{d_1-1}\neq a^1_{d_1}d_1 k_{j_2}+a^1_{d_1-1}$ and so
$p_{i_1}(n+k_{j_1})-p_{i_1}(k_{j_1})$ and $p_{i_2}(n+k_{j_2})-p_{i_2}(k_{j_2})$ in one variable $n$  are essentially distinct.

\medskip
\noindent{\bf Case 2.} $i_1=i_2=2$. In this case $j_1\neq j_2$. By the similar arguing as Case 1,
$p_{i_1}(n+k_{j_1})-p_{i_1}(k_{j_1})$ and $p_{i_2}(n+k_{j_2})-p_{i_2}(k_{j_2})$ in one variable $n$  are essentially distinct.

\medskip
\noindent{\bf Case 3.} $i_1\neq i_2$. In this case, when $j_1=j_2$ one has that $k_{j_1}=k_{j_2}$, and so the polynomials $p_{i_1}(n+k_{j_1})-p_{i_1}(k_{j_1})$ and $p_{i_2}(n+k_{j_2})-p_{i_2}(k_{j_2})$ in one variable $n$  are essentially distinct  since $p_{i_1}$ and $p_{i_2}$ are essentially distinct.

Now consider the case when $j_1\neq j_2$. If $\deg(p_1)\neq \deg(p_2)$, then the degrees of $p_{i_1}(n+k_{j_1})-p_{i_1}(k_{j_1})$ and $p_{i_2}(n+k_{j_2})-p_{i_2}(k_{j_2})$ in one variable $n$  are  distinct. Hence they are essentially distinct.

If $\deg(p_1)=\deg(p_2)$, then let $d_1=d_2=d\ge 2$. Note that the coefficient of $n^{d}$ in the polynomial $p_{i_1}(n+k_{j_1})-p_{i_1}(k_{j_1})$ is
$a^{i_1}_{d_{i_1}}$, and  the coefficient of $n^{d}$ in the polynomial $p_{i_2}(n+k_{j_2})-p_{i_2}(k_{j_2})$ is
$a^{i_2}_{d_{i_2}}$. Thus when $a_{d_1}^1\neq a_{d_2}^2$,   the polynomials  $p_{i_1}(n+k_{j_1})-p_{i_1}(k_{j_1})$ and $p_{i_2}(n+k_{j_2})-p_{i_2}(k_{j_2})$ in one variable $n$  are essentially distinct.

When $a_{d_1}^1= a_{d_2}^2$, let $a_{d_1}^1= a_{d_2}^2=a$. Note that the coefficient of $n^{d-1}$ in the polynomial $p_{i_1}(n+k_{j_1})-p_{i_1}(k_{j_1})$ is
$a^{i_1}_{d_{i_1}} d_{i_1} k_{j_1}+a^{i_1}_{d_{i_1}-1}=adk_{j_1}+a^{i_1}_{d_{i_1}-1}$, and  the coefficient of $n^{d-1}$ in the polynomial $p_{i_2}(n+k_{j_2})-p_{i_2}(k_{j_2})$ is
$a^{i_2}_{d_{i_2}} d_{i_2} k_{j_2}+a^{i_2}_{d_{i_2}-1}=a d k_{j_2}+a^{i_2}_{d_{i_2}-1}$.  Thus since $|k_1-k_2|\ge L$ and $i_1\neq i_2$ , one has
\begin{equation*}
  \begin{split}
     &\quad |(a^{i_1}_{d_{i_1}}d_{i_1} k_{j_1}+a^{i_1}_{d_{i_1}-1})-(a^{i_2}_{d_{i_2}} d_{i_2} k_{j_2}+a^{i_2}_{d_{i_2}-1})|\\
&\ge |a|d L-(|a^1_{d_1-1}|+|a^2_{d_2-1}|)\\
&\ge (|a^1_{d_1-1}|+|a^2_{d_2-1}|+1)-(|a^1_{d_1-1}|+|a^2_{d_2-1}|)\\
&=1.
   \end{split}
\end{equation*}
This implies that $a^{i_1}_{d_{i_1}}d_{i_1} k_{j_1}+a^{i_1}_{d_{i_1}-1}\neq a^{i_2}_{d_{i_2}} d_{i_2}k_{j_2}+a^{i_2}_{d_{i_2}-1}$ and so the polynomials $p_{i_1}(n+k_{j_1})-p_{i_1}(k_{j_1})$ and $p_{i_2}(n+k_{j_2})-p_{i_2}(k_{j_2})$ in one variable $n$  are essentially distinct.
\end{proof}

For $n<m$, we denote by $[n,m]$ the interval $\{n,n+1,\ldots, m\}$ of $\Z$. By the definition, it is easy to verify the following lemma.

\begin{lem}\label{lem-pw}
Let $A\subseteq \Z^2$ a piecewise syndetic subset in $\Z^2$ and $B\subseteq \Z^2$. If for any $M\in \N$, there is some ${\bf n}_M\in \Z^2$ such that
$${\bf n}_M+(A\cap [-M,M]^2)\subseteq B,$$
then $B$ is also a piecewise syndetic subset in $\Z^2$.
\end{lem}

\subsection{The proof of Theorem C}\ 
\medskip



In last section, we have proved the first and second part of Theorem C.
Now we finish the proof of Theorem C by proving the last part of it.  Let $(X,T)$ be minimal and $\pi:X\lra X_\infty$ be the factor map. By the discussion in
Section \ref{Section-nil}, we know that the theorem holds for $\infty$-step pro-nilsystems. The idea of our proof is that
we lift the property from $X_\infty$ to $X$ by using the saturation theorem for polynomials when $\pi$ is open. If $\pi$ is not open
then we can modify it by almost one to one extensions. Note that in the middle of the proof we will explain the idea of the proof
for some special case.

\begin{proof}[The proof of the last part of Theorem C] By Theorem \ref{thm-GHSWY}, if $\pi:X\rightarrow X_\infty$ is the factor map to $X_\infty$, then there are minimal t.d.s. $X^*$ and $X_\infty^*$ which are almost one to one extensions of $X$ and $X_\infty$ respectively, and a commuting diagram below such that $X_\infty^*$ is a
$k$-step topological characteristic factor of $X^*$ for all $k\ge 2$.
\[
\begin{CD}
X @<{\varsigma^*}<< X^*\\
@VV{\pi}V      @VV{\pi^*}V\\
X_\infty @<{\varsigma}<< X_\infty^*
\end{CD}
\]

If the result holds for $X^*$, then it also holds for $X$ since $\varsigma^*$ is almost one to one. So without loss of generality, we may assume that $X=X^*$. That is, we may assume the following diagram
$$
\xymatrix@R=0.5cm{
  X \ar[dd]_{\pi} \ar[dr]^{\pi^*}             \\
                & X^*_\infty \ar[dl]_{\varsigma}         \\
  X_\infty                 }
$$
where $\pi^*$ is open and $\varsigma$ is almost one to one.
Let $\Omega=\{x\in X^*_\infty: \varsigma^{-1}(\varsigma(x))=\{x\}\}.$
Then $\Omega$ is a dense $G_\d$ subset of $X^*_\infty$. Let
$$X_0=X_{pol}\cap (\pi^*)^{-1}(\Omega),$$
where $X_{pol}$ is defined in \eqref{cont-poly}.
Then $X_0$ is a dense $G_\d$ subset of $X$, and we will show it is what we are looking for, that is, for essentially distinct integral polynomials  $p_1,\ldots, p_d$
with $p_i(0)=0$ and $\deg (p_i)\ge 2$, $1\le i\le d$,
for each $x\in X_0$ and each neighbourhood $U$ of $x$
$$N_{\{p_1,\ldots,p_d\}}(x,U)=\{n\in\Z: T^{p_1(n)}x\in U, \ldots, T^{p_d(n)}x\in U\}\in \F_{ps}.$$

\medskip
{\em To make the proof easier to understand, we first outline the proof when $d=1$ and $p_1(n)=n^2$.}

\medskip
In this case let $x\in X_0$, $y=\pi^*(x)$ and $V=\pi^*(U)$ which is an open subset of $X_\infty^*$, where $U$ is an open neighborhood of $x$. By Theorem~ \ref{thm-Leibman-poly-nil}
and some discussions we get that
$$\{n\in\Z: T^{n^2}y\in V\}=\{\ldots<s_{-1}<s_0<s_1<\ldots\}$$ is syndetic. It is clear that $T^{-s_j^2}U\cap (\pi^*)^{-1}(y)\not=\emptyset$
for each $j\in\Z$.

Let $k\in \N$. 
We will show that there is $m=m(k)\in\Z$ such that
$$m+s_{-k},\ldots, m+s_k\in \{n\in\Z: T^{n^2}x\in U\}$$ which implies that $\{n\in\Z: T^{n^2}x\in U\}\in \F_{ps}.$

Let $q_i(n)=n^2+2s_in$, $n\in\Z, -k\le i\le k$. Then $q_{-k},\ldots,q_k$ are pairwise distinct polynomials. Applying Theorem \ref{thm-poly-sat1} to $q_{-k},\ldots,q_k$ and the open sets
$T^{-s_{-k}^2}U,\ldots, T^{-s_k^2}U$ we get that there is $m\in\Z$ such that  $T^{q_j(m)}x\in T^{-s_j^2}U$, i.e.
$T^{(m+s_j)^2}x\in U$ for each $-k\le j\le k$.
{\em This finishes the proof for the special case.}

\medskip

Now come back to the proof for the general case.
By Lemma \ref{wm-2},  there is some $L\in \mathbb{N}$ such that if $k_1,k_2\in \mathbb{Z}$ with $|k_1-k_2|\ge L$, then for any $(i_1,j_1)\neq (i_2,j_2)\in \{1,2,\ldots, d\}\times \{1,2\}$, the polynomials  $p_{i_1}(n+k_{j_1})-p_{i_1}(k_{j_1})$ and $p_{i_2}(n+k_{j_2})-p_{i_2}(k_{j_2})$ in one variable $n$  are essentially distinct.

\medskip

Let $z=\pi^*(x)$. Since $x\in X_0$, $z\in \Omega$. Given a neighbourhood $U$ of $x$, let $\ep>0$ such that $B_{2\ep}(x)\subseteq U$,  and let $\widetilde{U}=B_\ep(x)\subseteq U$. Note $\pi^*(\widetilde{U})$ is an open neighbourhood of $z$ as $\pi^*$ is open. Let $z_0=\varsigma(z)$. Since $z\in \Omega$, there exists an open neighbourhood
$W$ of $z_0$ in $X_\infty$ such that $\varsigma^{-1}(W)\subseteq  \pi^*(\widetilde{U})$. Thus
\begin{align}\label{lift-recu-eq-1}
\begin{aligned}
F& \triangleq N_{\{p_1,\ldots,p_d\}}(z,\pi^*(\widetilde{U}))=\{n\in \mathbb{Z}: T^{p_1(n)}z, \ldots , T^{p_d(n)}z \in \pi^*(\widetilde{U})\}\\
&\supseteq N_{\{p_1,\ldots,p_d\}}(z_0,W)=\{n\in \mathbb{Z}: T^{p_1(n)}z_0, \ldots , T^{p_d(n)}z_0 \in W\}.
\end{aligned}
\end{align}

Note that  the minimal $\infty$-step
pro-nilsystem $X_\infty$ is inverse limit of minimal nilsystems. Let $(X_\infty,T)=\displaystyle
\lim_{\longleftarrow}(X_r,T)_{r\in \N}$, where each $(X_r,T)$ is minimal nilsystem and $\pi_{\infty,r}:(X_\infty,T)\rightarrow (X_r,T)$ is the corresponding factor map.
Since
 $$\pi_{\infty,1}^{-1}\big(\pi_{\infty,1}(z_0)\big)\supseteq \pi_{\infty,2}^{-1}\big(\pi_{\infty,2}(z_0)\big)\supseteq \cdots \supseteq \pi_{\infty,r}^{-1}\big(\pi_{\infty,r}(z_0)\big)\supseteq \cdots$$  and
 $$\bigcap_{r=1}^\infty \pi_{\infty,r}^{-1}\big(\pi_{\infty,r}(z_0)\big)=\{z_0\},$$
 we can find $r_*\in \mathbb{N}$ such that $\pi_{\infty,r_*}^{-1}\big(\pi_{\infty,r_*}(z_0)\big)\subseteq W$.
 Let $z_*=\pi_{\infty,r_*}(z_0)$.  Note that $\pi_{\infty,r_*}$ is an open map, we can find an open neighborhood $W_*$ of $z_*$ in $X_{r_*}$ such that $\pi_{\infty,r_*}^{-1}(W_*)\subseteq W$.
 Hence
 \begin{equation}\label{lift-recu-eq-2}
 N_{\{p_1,\ldots,p_d\}}(z_0,W)\supseteq  N_{\{p_1,\ldots,p_d\}}(z_*,W_*)=\{n\in \mathbb{Z}: T^{p_1(n)}z_*, \ldots , T^{p_d(n)}z_* \in W_*\}.
 \end{equation}

Since $(X_{r_*},T)$ is a minimal nilsystem, there exist a nilpotent lie group $G$, a discrete
cocompact subgroup $\Gamma$ of $G$ and $a\in G$ such that $X_{r_*}=G/\Gamma$ and $T(h\Gamma)=ah\Gamma$ for any $h\in G$.
Let $\widetilde{G}=G\times G\times \cdots \times G$ ($d$-times),
 $\widetilde{\Gamma}=\Gamma\times \Gamma\times \cdots \times \Gamma$ ($d$-times) and $g: \mathbb{Z}\rightarrow \widetilde{G}$ with $g(n)=(a^{p_1(n)},\cdots,a^{p_d(n)})$ for any $n\in \mathbb{Z}$. Then $g: \mathbb{Z}\rightarrow \widetilde{G}$ is a polynimal.
 Let $$\widetilde{z_*}=(z_*,z_*,\cdots,z_*)\in \widetilde{G}/\widetilde{\Gamma}=G/\Gamma\times G/\Gamma\times \cdots\times G/\Gamma.$$
 Thus by Theorem  \ref{thm-Leibman-poly-nil},
 we have a continuous surjective map $\eta: \widehat{X}\rightarrow \widetilde{G}/\widetilde{\Gamma}$ of a `larger' nil-manifold $\widehat{X}=\widehat{G}/\widehat{\Gamma}$, a homomorphism $\phi: \mathbb{Z} \rightarrow \widehat{G}$ and $\hat{x}\in \widehat{X}$ such that
\begin{equation}\label{eeeee-eq-1}
g(n)\widetilde{z_*}=\eta(\phi(n)\hat{x}), \ \forall  n\in \mathbb{Z}.
\end{equation}
 Take $\hat{W}=\eta^{-1}(W_*\times W_*\times \cdots \times W_*)$, $\hat{a}=\phi(1)\in \hat{G}$, and $R(\hat{h}\hat{\Gamma})=\hat{a}\hat{h}\hat{\Gamma}$ for $\hat{h}\in \hat{G}$. Then $(\widehat{X},R)$ is nilsystem, $\hat{W}$ is an open neighbourhood of $\hat{x}$ in $\widehat{X}$,  and
 \begin{align*}
 N_{\{p_1,\ldots,p_d\}}(z_*,W_*)&=\{n\in \mathbb{Z}: (T^{p_1(n)}z_*, \ldots , T^{p_d(n)}z_*) \in W_*\times \cdots\times W_*\}\\
 &=\{n\in \mathbb{Z}: g(n)\widetilde{z_*}\in W_*\times W_*\times \cdots \times W_*\}\\
 &\overset{\eqref{eeeee-eq-1}}=\{n\in \mathbb{Z}: \phi(n)\hat{x}\in \eta^{-1}(W_*\times W_*\times \cdots \times W_*)\}\\
 &=\{n\in \mathbb{Z}: R^n\hat{x}\in \hat{W}\}\\
 &=N_R(\hat{x},\hat{W}).
 \end{align*}
Combining this with \eqref{lift-recu-eq-1} and \eqref{lift-recu-eq-2}, one has
$$F=N_{\{p_1,\ldots,p_d\}}(z,\pi^*(\widetilde{U}))\supseteq N_R(\hat{x},\hat{W}).$$
Hence $F\in \F_s$ as $N_R(\hat{x},\hat{W})$ is syndetic.
Let $$F=\{n'_j\}_{j\in \mathbb{Z}}=\{\ldots <n'_{-2}<n'_{-1}<n'_0<n'_1<n'_2<\ldots\}.$$
We take $n_j=n'_{Lj}$  for $j\in \mathbb{Z}$.
Then $\{n_{j}\}_{j\in\mathbb{Z}}\in \F_s$ and  $n_{i+1}-n_{i}\ge L, \forall i\in \mathbb{Z}$.

\medskip

To finish the proof we now show that for each $l\in \mathbb{N}$, there is some $m\in \mathbb{Z}$ such that $m+\{n_j\}_{-l\le j\le l}\subseteq N_{\{p_1,\ldots, p_d\}}(x,U)$ and hence $N_{\{p_1,\ldots, p_d\}}(x,U)\in \F_{ps}$.

Fix $l\in \N$. Let $\A'=\{p_i(n+n_j): -l\le j\le l, 1\le i\le d\}$. Then $\A'$ is a set of essentially distinct integral polynomials, since $n_{i+1}-n_{i}\ge L$. By Theorem \ref{thm-poly-sat1},
 $\overline{\O}_{\A'}(x^{\otimes d'})$ is ${\pi^*}^{(d')}$-saturated, that is,
$$\overline{\O}_{\A'}(x^{\otimes d'})=({\pi^*}^{(d')})^{-1}\Big({\pi^*}^{(d')}(\overline{\O}_{\A'}(x^{\otimes d'}))\Big),$$
where $d'=d(2l+1)$. One has that
$$F_{\A'}(x)\cap U^{d(2l+1)}\neq \emptyset.$$
In particular, there is some $m\in \Z$ such that
$$\Big(T^{p_i(m+n_j)}(x)\Big)_{1\le i\le d, -l\le j\le l}\in {U}^{d(2l+1)}.$$
Thus $$m+\{n_j\}_{-l\le j\le l}=\{m+n_j\}_{-l\le j\le l}\subseteq N_{\{p_1,\ldots,p_d\}}(x,U).$$
The proof is complete.
\end{proof}


By Theorem \ref{Thm-equivalence2}, we have the following corollary of Theorem C:

\begin{cor}\label{}
Let $(X,T)$ be a minimal t.d.s.
There is a dense $G_\d$ subset $X_0$ of $X$ such that for each $x\in X_0$ and for any family $\A=\{p_1, p_2,\cdots, p_d\}$ of integral polynomials satisfying $(\spadesuit)$ and $\deg p_i\ge 2, 1\le i\le d$,  $(W_x^\A,\sigma)$ is an M-system and $\xi_x^\A$ is a transitive point.
\end{cor}

\subsection{The proof of Theorem D}\
\medskip

In this subsection we prove Theorem D.
\subsubsection{A lemma}

\begin{lem}\label{lem-syn}
Let $E$ be a syndetic subset of $\mathbb{Z}^2$. Then there is some $m_*\in \mathbb{Z}$ such that
$$E(m_*):=\{ n\in \mathbb{Z}: (m_*,n)\in E\}\in \F_{ps}.$$
\end{lem}

\begin{proof}
Since $E$ is $\Z^2$-syndetic, there is some $L\in \N$ such that
\begin{equation}\label{hh}
  E+[-L,L]^2=\Z^2,
\end{equation}
where $[-L,L]=\{-L, -L+1,\ldots,L\}$ is an interval of $\Z$.
That is, for each $(m,n)\in \Z^2$, there exits $(i,j)\in [-L,L]^2$ such that $(m +i, n+j)=(m, n)+(i,j)\in E.$

For each $i \in [-L,L]$, let
$$E(i)\triangleq \{ n\in \mathbb{Z}: (i,n)\in E\}.$$

By \eqref{hh}, for each $n\in \Z$, there is some $(i,j)\in [-L,L]^2$ such that
$$(0,n)+(i,j)\in E.$$
That is, $n+j\in E(i)$. It follows that
$$\Z=\bigcup_{-L\le i,j\le L} (E(i)-j).$$
Since $\F_{ps}$ has the Ramsey property (see \cite[Theorem 1.24]{F}), there are some $m_*,j\in [-L,L]$ such that
$E(m_*)-j\in \F_{ps}$. And hence
$$E(m_*)\in \F_{ps},$$
as $\F_{ps}$ is translation invariant.
\end{proof}

\subsubsection{The proof of Theorem D}
Now we are ready to give the proof of Theorem D. Note that the idea of the proof of D is similar to the one of the last part
of Theorem C. Since linear terms may appear in $\A$, this makes the argument longthy.

\begin{proof}[Proof of Theorem D]
By Theorem \ref{thm-poly-sat1}, if $\pi:X\rightarrow X_\infty$ is the factor map to $X_\infty$, then there are minimal t.d.s. $X^*$ and $X_\infty^*$ which are almost one to one extensions of $X$ and $X_\infty$ respectively, and a commuting diagram below.
\[
\begin{CD}
X @<{\varsigma^*}<< X^*\\
@VV{\pi}V      @VV{\pi^*}V\\
X_\infty @<{\varsigma}<< X_\infty^*
\end{CD}
\]

If the result in Theorem D holds for $X^*$, then it also holds for $X$. So without loss of generality, we may assume that $X=X^*$. That is, we may assume
$$
\xymatrix@R=0.5cm{
  X \ar[dd]_{\pi} \ar[dr]^{\pi^*}             \\
                & X^*_\infty \ar[dl]_{\varsigma}         \\
  X_\infty                 }
$$
where $\pi^*$ is open and $\varsigma$ is almost one to one.
Let $\Omega=\{x\in X^*_\infty: \varsigma^{-1}(\varsigma(x))=\{x\}\}.$
Then $\Omega$ is a dense $G_\d$ subset $X^*_\infty$. Let
$$X_0=X_{pol}\cap (\pi^*)^{-1}(\Omega),$$
where $X_{pol}$ is defined in \eqref{cont-poly}.
Then $X_0$ is a dense $G_\d$ subset of $X$, $T(X_0)=X_0$ and for all $x\in X_0$ statements in Theorem \ref{thm-poly-sat1} still hold.

We will show that for integral polynomials  $p_1,\ldots, p_d$
with $p_i(0)=0$, $1\le i\le d$,
for any non-empty open set $U$, we can find some $x\in X_0\cap U$ such that $$N_{\{p_1,\ldots, p_d\}}(x,U)=\{n\in\Z: T^{p_1(n)}x\in U, \ldots, T^{p_d(n)}x\in U\}\in \F_{ps}.$$
Without loss of generality, we may require that  $p_1,\ldots, p_d$
 are essentially distinct non-constant integral polynomials such that $p_i(0)=0$, $1\le i\le d$.

If $\deg(p_i)\ge 2, \forall 1\le i\le d$, then by the third part of Theorem C it is done. So we need to deal with the case when there are linear polynomials in $p_1,\ldots, p_d$. Let
\begin{align}\label{A-def-eq}
\A=\{p_{1}, \ldots, p_s, p_{s+1}(n),\ldots, p_d(n)\},
\end{align}
where $p_i(n)=a_in$, $1\le i\le s$ and $\deg(p_i)\ge 2, s+1\le i\le d$.  Since $p_1,\ldots, p_d$
are essentially distinct non-constant integral polynomials, $a_1,\ldots, a_s\in \mathbb{Z}\setminus\{0\}$ are distinct.

 Let $U$ be a non-empty open set of $X$. Choose an $\ep>0$ and an non-empty open set $\widetilde{U}$ of $X$ such that $B_{2\ep}(\widetilde{U})\subseteq U$. First we need the following claim whose proof will be given in the next subsection.

\medskip
\noindent{\bf Claim:} \ {\em There is some $x\in \widetilde{U}\cap X_0$ such that
$$N_{\{p_1,\ldots, p_s\}}(x,\widetilde{U})\cap N_{\{p_{s+1},\ldots,p_d\}}(\pi^*(x),\pi^*(\widetilde{U}))\in \F_{ps}.$$
}

By Lemma \ref{wm-2}, there is some $L\in \mathbb{N}$ such that if $k_1,k_2\in \mathbb{Z}$ with $|k_1-k_2|\ge L$, then for any
$(i_1,j_1), (i_2,j_2)\in \{s+1,\ldots, d\}\times \{1,2\}$ with $(i_1,j_1)\neq (i_2,j_2)$, the polynomials  $p_{i_1}(n+k_{j_1})-p_{i_1}(k_{j_1})$ and $p_{i_2}(n+k_{j_2})-p_{i_2}(k_{j_2})$  are essentially distinct.

\medskip

Let $z=\pi^*(x)$, where $x\in \widetilde{U}\cap X_0$ is as in the Claim above. Note $\pi^*(\widetilde{U})$ is an open neighbourhood of $z$ as $\pi^*$ is open.
Let
$$F=N_{\{p_1,\ldots,p_s\}}(x,\widetilde{U})\cap N_{\{p_{s+1},\ldots,p_d\}}(\pi^*(x),\pi^*(\widetilde{U})).$$
By Claim, $F=\{n_j\}_{j\in \mathbb{Z}}=\{\ldots < n_{-2}<n_{-1}<n_0<n_1<n_2<\ldots\}$ is a piecewise syndetic set of $\mathbb{Z}$. By taking a piecewise syndetic subset, we may assume that $|n_{i}-n_j|\ge L, \forall i\neq j\in \mathbb{Z}$.

We show that for each $l\in \N$, there is some $m\in \Z$ such that $m+\{n_j\}_{-l\le j\le l}\subseteq N_\A(x,U)$ and hence $N_\A(x,U)\in \F_{ps}$ by Lemma \ref{lem-pw}, where $\A$ is defined in \eqref{A-def-eq}.

\medskip

Fix $l\in \N$. Since $\{n_j\}_{-l\le j\le l}\subseteq N_{\{p_1,\ldots, p_s\}}(x, \widetilde{U})$, we have
$$x\in \widetilde{U}\cap \bigcap_{j=-l}^{l}\left( T^{-a_1n_j}\widetilde{U}\cap T^{-a_2n_j}\widetilde{U}\cap \ldots \cap T^{-a_sn_j}\widetilde{U} \right).$$
Let $0<\ep'<\ep$ such that
$$x\in B_{2\ep'}(x) \subseteq \widetilde{U}\cap \bigcap_{j=-l}^{l}\left( T^{-a_1n_j}\widetilde{U}\cap T^{-a_2n_j}\widetilde{U}\cap \ldots \cap T^{-a_sn_j}\widetilde{U} \right).$$
Thus we have
\begin{equation}\label{b6}
  T^{a_tn_j}B_{2\ep'}(x)\subseteq \widetilde{U}, \quad 1\le t\le s, -l\le j\le l.
\end{equation}

Let $$\A'=\{p_1, \ldots, p_s\}\cup \{p_i^j: -l\le j\le l, s+1\le i\le d\},$$
where $p_i^j(n)=p_i(n+n_j)$.
Then $\A'$ is a set of essentially distinct integral polynomials. As $x\in X_0$,
by Theorem \ref{thm-poly-sat1} $\overline{\O}_{\A'}(x^{\otimes d'})$ is ${\pi^*}^{(d')}$-saturated, that is,
$$\overline{\O}_{\A'}(x^{\otimes d'})=({\pi^*}^{(d')})^{-1}\Big({\pi^*}^{(d')}(\overline{\O}_{\A'}(x^{\otimes d'}))\Big),$$
where $d'=(d-s)(2l+1)+s$. One has that
$$F_{\A'}(x)\cap \Big (B_{2\ep'}(x)^s\times U^{(d-s)(2l+1)}\Big)\neq \emptyset,$$
where $F_{*}$ is defined in Lemma \ref{cont-points}. In particular, there is some $m\in \Z$ such that
$$\left(T^{a_1m}x,\ldots, T^{a_sm}x, \big(T^{p_i(m+n_j)}(x)\big)_{s+1\le i\le d, -l\le j\le l}\right) \in B_{2\ep'}(x)^s\times U^{(d-s)(2l+1)}.$$
Thus
$$T^{p_i(m+n_j)}x\in  U, \quad s+ 1\le i\le d, -l\le j\le l,$$
and by \eqref{b6},
$$T^{a_t(m+n_j)}x\in T^{a_tn_j}B_{2\ep'}(x)\subseteq \widetilde{U}\subseteq U, \quad 1\le t\le s, -l\le j\le l.$$

Thus $$m+\{n_j\}_{-l\le j\le l}=\{m+n_j\}_{-l\le j\le l}\subseteq N_{\{p_1,\ldots,p_d\}}(x,U).$$
The proof is complete.
\end{proof}

\subsubsection{The proof of the Claim}
\begin{proof}[Proof of Claim]
Let $x_0\in \widetilde{U}\cap X_0$ and $z_0=\pi^*(x_0)$. Since $\pi^*$ is open, $\pi^*(\widetilde{U})$ is an open neighbourhood of $z_0$. By the same arguing in the proof of Theorem~ C (the third part),  we can find a nil-manifold $\widehat{X}=\widehat{G}/\widehat{\Gamma}$, a homomorphism $\phi: \mathbb{Z}^2 \rightarrow \widehat{G}$, $\hat{x}\in \widehat{X}$ and an open neighbourhood $\hat{W}$ of $\hat{x}$ in $\widehat{X}$ such that
 \begin{align}\label{key-eeee-1}
 \begin{aligned}
&\hskip0.5cm \{(m,n)\in \mathbb{Z}^2: T^{m+p_{s+1}(n)}z_0\in \pi^*(\widetilde{U}),\ldots, T^{m+p_d(n)}z_0\in \pi^*(\widetilde{U})\}\\
 &\supseteq \{(m,m)\in \mathbb{Z}^2: \phi(m,n)\hat{x}\in \hat{W}\}.
 \end{aligned}
 \end{align}

We consider a $\mathbb{Z}^2$-action on $\widehat{X}=\widehat{G}/\widehat{\Gamma}$ by
$$g^{(m,n)}(\hat{h}\widehat{\Gamma})=\phi(m,n)\hat{h}\widehat{\Gamma}$$
for any $\hat{h}\in \widehat{G}$ and $(n,m)\in \mathbb{Z}^2$. Then
$(\widehat{X},\{g^{(m,n)}\}_{(m,n)\in \mathbb{Z}^2})$ is a distal $\mathbb{Z}^2$-t.d.s.

Then \eqref{key-eeee-1} can be rewritten as
  \begin{align}\label{key-eeee-2}
 \begin{aligned}
&\hskip0.5cm \{(m,n)\in \mathbb{Z}^2: T^{m+p_{s+1}(n)}z_0\in \pi^*(\widetilde{U}),\ldots, T^{m+p_d(n)}z_0\in \pi^*(\widetilde{U})\}\\
 &\supseteq \{(m,n)\in \mathbb{Z}^2: g^{(m,n)}\hat{x}\in \hat{W}\}.
 \end{aligned}
 \end{align}

Take $a=\max_{1\le j\le s}|a_j|$ and $t=2a+1$. Let
$$\sigma_{t}=T\times \ldots\times T \ (t \ \text{times}),\ \text{and}\ \tau=T^{-a}\times \cdots\times T^{-1}\times \id \times T \times \ldots \times T^{a}.$$
By Theorem \ref{thm-Glasner}, $N_{t}(X)=\overline{\O}(x_0^{\otimes t}, \langle \sigma_{t}, \tau\rangle)$ is a $\langle \sigma_{t}, \tau\rangle$-minimal t.d.s., where $x_0^{\otimes t}=(x_0,\ldots, x_0) \in X^t$.
We define a $\Z^2$-action on $X^{t}\times \widehat{X}$ by
$$\big((x_1,\ldots, x_{t}),\hat{h}\widehat{\Gamma}\big)\mapsto (\sigma_{t}^m\tau^n\times g^{(m,n)})\big((x_1,\ldots, x_{t}),\hat{h}\widehat{\Gamma}\big), \forall (m,n)\in \mathbb{Z}^2.$$
We denote the corresponding $\mathbb{Z}^2$-t.d.s. by $(X^{t}\times \widehat{X},\{\sigma_{t}^m\tau^n\times g^{(m,n)}\}_{(m,n)\in \mathbb{Z}^2} )$.

Since $\hat{x}$ is a $\{g^{(m,n)}\}_{(m,n)\in \mathbb{Z}^2}$-distal point and $x_0^{\otimes t}$ is a $\{\sigma_{t}^m\tau^n\}_{(m,n)\in \mathbb{Z}^2}$-minimal point, one has $(x_0^{\otimes t}, \hat{x})$ is a minimal point of $(X^{t}\times \widehat{X},\{\sigma_{t}^m \tau^n\times g^{(m,n)}\}_{(m,n)\in \mathbb{Z}^2} )$. \footnote{Let $(X,T), (Y,S)$ be t.d.s. If $x\in X$ is distal and $y\in Y$ is minimal, then $(x,y)$ is a $T\times S$-minimal point \cite[Theorem 9.11.]{F}. This also holds for $\mathbb{Z}^2$-t.d.s. and more general group actions \cite{AF}.}

Thus for the open neighborhood $\widetilde{U}\times \cdots \times \widetilde{U}\times \hat{W}$ of $(x_0^{\otimes t}, \hat{x})$,
$$E\triangleq \{(m,n)\in \mathbb{Z}^2: \sigma_{t}^m\tau^n\times g^{(m,n)}\big( x_0^{\otimes t}, \hat{x} \big)\in  \widetilde{U}\times \cdots \times \widetilde{U}\times \hat{W}\}$$
is a syndetic subset of $\mathbb{Z}^2$.
Moreover, by Lemma \ref{lem-syn} we can find $m_*\in \mathbb{Z}$ such that
$$E(m_*):=\{ n\in \mathbb{Z}: (m_*,n)\in E\}\in \F_{ps}.$$

For any $n\in E(m_*)$, one has
$$T^{m_*+jn}(x_0)\in \widetilde{U} \text{ for }j=-a,\cdots,-1,0,1,\cdots,a$$ and
$g^{(m_*,n)}\hat{x}\in \widehat{W}$, which implies
$T^{m_*+p_i(n)}z_0\in \pi^*(\widetilde{U})$ for $i=s+1,s+2,\cdots,d$ by \eqref{key-eeee-2}. In particular, $T^{m_*}x_0\in \widetilde{U}$.

Now put $x=T^{m_*}(x_0)$. Then $x\in \widetilde{U}\cap X_0$ as $X_0$ is $T$-invariant. Note that $\pi^*(x)=T^{m_*}(z_0)$. Then by the above discussion for
any $n\in E(m_*)$, one has that
$$T^{a_jn}x=T^{m_*+a_jn}x_0\in \widetilde{U} \text{ for }j=1,\cdots, s,$$
(that is, $n\in N_{\{p_1,\ldots,p_s\}}(x,\widetilde{U})$) and
$$T^{p_i(n)}\pi^*(x)=T^{m_*+p_i(n)}z_0\in \pi^*(\widetilde{U})\text{ for }i=s+1, s+2, \cdots, d,$$
(that is, $n\in  N_{\{p_1,\ldots,p_d\}}(\pi^*(x),\pi^*(\widetilde{U}))$).
Hence $$E(m_*)\subseteq N_{\{p_1,\ldots,p_s\}}(x,\widetilde{U})\cap N_{\{p_{s+1},\ldots,p_d\}}(\pi^*(x),\pi^*(\widetilde{U}))$$ and so
$$N_{\{p_1,\ldots, p_s\}}(x,\widetilde{U})\cap N_{\{p_{s+1},\ldots, p_d\}}(\pi^*(x),\pi^*(\widetilde{U}))\in \mathcal{F}_{ps}.$$
The proof of the Claim is complete.
\end{proof}

\medskip

\section{The proof of Theorem E}\label{Section-Z2-recurrence}


In this section we will show that for any minimal t.d.s. $(X,T)$,  $(M_\infty({X},\A), \langle T^\infty, \sigma\rangle)$ is an $M$-system. And as a corollary, we give a proof of Theorem E.

\medskip

In this section we will use the regularizer and the universal minimal t.d.s. from the theory of minimal flows, which usually deals with group actions on compact Hausdorff spaces.
First we need some facts about minimal t.d.s. which will be used in the proof.


\subsection{Some facts about minimal t.d.s.}\

\subsubsection{The universal minimal t.d.s.}

Let $\beta\Z$ be the Stone-C\v{e}ch
compactification of $\Z$, which is a compact Hausdorff topological
space where $\Z$ is densely and equivariantly embedded. Moreover,
the addition on $\Z$ can be extended to an addition on $\beta\Z$
in such a way that $\beta\Z$ is a compact semigroup with continuous
right translations. 
Let ${\bf M}$ be a minimal left ideal in $\beta\Z$. Then ${\bf M}$ is a closed semigroup with continuous right translations.
By the Ellis-Namakura Theorem the set $J=J({\bf M})$ of idempotents in ${\bf M}$ is non-empty.
Moreover, $\{v{\bf M}:v \in J \}$ is a partition of ${\bf M}$ and every $v{\bf M}$ is a group with unit element $v$.
Let $(X,T)$ be a t.d.s. and $x\in X$. A necessary and sufficient condition for $x$ to be minimal is that $ux=x$ for some $u\in J$.

Let $(X,T)$ be a t.d.s. and then $\Z$ acts on the compact metric space $X$ as follows:
for any $m\in \Z$ and $x \in X$ one has $mx=T^mx$. Then the sets
$\beta\Z$ and ${\bf M}$ also act on $X$ as semigroups and $\beta\Z
x=\{px:p\in \beta\Z\}=\overline{\O}(x,T)$. If $(X,T)$ is minimal, then
${\bf M}x=\overline{\O}(x,T)=X$ for every $x\in X$. Also, as a t.d.s., $({\bf M}, T)$ is an extension of any minimal t.d.s. $(X,T)$. We call $({\bf M}, T)$
the {\em universal minimal} t.d.s. Up to isomorphisms, the universal minimal t.d.s. is unique \cite[Chapter 8]{Au88}.


\subsubsection{Hyperspace t.d.s.}

When $X$ is not metric, neither is $2^X$. And in this case $2^X$ can be
endowed with the Vietoris topology. A basis for this
topology on $2^X$ is given by
$$\langle U_1,\ldots,U_n\rangle=\{A\in 2^X: A\subseteq \bigcup_{i=1}^n U_i\
\text{and $A\cap U_i\neq \emptyset$ for every $i\in \{
1,\ldots,n\}$}\},$$
where each $U_i \subseteq X$ is open. Note that when $X$ is a metric space, the Hausdorff topology consists with the topology induced by Hausdorff metric of $2^X$.

Let $(X,T)$ be a t.d.s. It induces a t.d.s. on $2^X$, which is called the {\em hyperspace t.d.s.}
The
action of $\Z$ on $2^X$ is given by $mA=\{ma:a\in A\}=\{T^m a: a\in A\}$ for each $m \in
\Z$ and $A \in 2^X$. This action induces another one of $\beta\Z$
on $2^X$. To avoid ambiguities one denotes the action of $\beta\Z$
on $2^X$ by the {\em circle operation} as follows: let $p\in
\beta\Z$ and $A\in 2^X$, and define $p \c
A={\lim\limits_{\lambda}} \ m_\lambda A$ for any net $\{m_\lambda\}_{\lambda \in \Lambda}\subseteq \Z$ converging to $p$.
 Moreover
\begin{equation*}
p\c A=\{x\in X: \text{for each $\lambda \in \Lambda$ there is $d_\lambda\in A$ with $x=\lim_\lambda m_\lambda d_\lambda$}\}
\end{equation*}
for any fixed net  $\{m_\lambda\}_{\lambda \in \Lambda}$ converging to $p$. Observe that $pA \subseteq p \circ A$, where $pA=\{pa: a \in A\}$.

\subsubsection{RIC extensions}

Let $\pi: (X,T )\rightarrow (Y,T)$ be a factor map of minimal t.d.s., and $x_0\in X$, $y_0=\pi(x_0)$. Let $u\in J$ such that $ux_0=x_0$. We say that $\pi$
is a {\em RIC} ({\em relatively incontractible}) extension if for every $y
= py_0\in Y$, $p\in\M$,
\begin{equation}\label{cir-op}
\pi^{-1}(y)=p\c \big(u\pi^{-1}(y_0)\big).
\end{equation}
Note that every distal extension is RIC, and every RIC extension is open.

Every factor map between minimal systems can be
lifted to a RIC extension by proximal extensions (see \cite[Theorem 5.13]{EGS} or \cite[Chapter VI]{Vr}).

\begin{thm}\label{RIC}
Given a factor map $\pi:(X,T)\rightarrow (Y,T)$ of minimal systems, there exists a commutative diagram of factor maps (called {\em RIC-diagram}) 
	
	\[
	\begin{CD}
	X @<{\theta '}<< X'\\
	@VV{\pi}V      @VV{\pi'}V\\
	Y @<{\theta}<< Y'
	\end{CD}
	\]
	such that:
	\begin{enumerate}
		\item[(a)] $\theta '$ and $\theta$ are proximal extensions;
		\item[(b)] $\pi '$ is a RIC extension;
		\item[(c)] $X '$ is the unique minimal set in $R_{\pi \theta
		}=\{(x,y)\in X\times Y ': \pi(x)=\theta(y)\}$, and $\theta '$ and
		$\pi '$ are the restrictions to $X '$ of the projections of $X\times
		Y '$ onto $X$ and $Y '$ respectively.
	\end{enumerate}
\end{thm}


\subsubsection{Regular extensions}

\begin{de}
Let $\pi: (X, T)\rightarrow (Y,T)$ be an extension of minimal t.d.s. One says $\pi$ is {\em regular} if for any
minimal point $(x_1,x_2)$ in $R_\pi$ (i.e. $\pi(x_1)=\pi(x_2)$) there exists $\chi\in
{\rm Aut}_\pi (X,T)$ such that $\chi(x_1)=x_2$. It is equivalent to: for any point $(x_1,x_2)\in R_\pi$ there exists  $\chi \in {\rm Aut}_\pi(X,T)$ such
that $(\chi(x_1),x_2)\in {\bf P} (X,T)$.
\end{de}

The notion of regularity was introduced by Auslander \cite{Au66}. Examples of regular extensions are
proximal extensions, group extensions.
For more information about regularity, refer to \cite{Au66, Au88, G92}.

\begin{thm}\cite[Chapter V, 6.3-7]{Vr}\label{thm-reg}
Let $(X,T)$ and $(Y,T)$ be minimal t.d.s. and let $\phi$ be the factor map. Then there is an extension $\theta: (X^*,T)\rightarrow (X,T)$ such that $\phi^*=\phi\circ \theta: (X^*,T)\rightarrow (Y,T)$ is regular. Moreover, $\theta$ is an isomorphism if and only $\phi$ is regular.
\begin{equation*}
  \xymatrix@R=0.5cm{
                &         {X}^* \ar[dl]_{\theta} \ar[dd]^{\phi^*}    \\
  X \ar[dr]_{\phi}      \\
                &         Y          }
\end{equation*}

The extension $\phi^*: X^*\rightarrow Y$ is called a {\em regularizer} of $\phi$.
\end{thm}

\begin{rem}\label{rem-reg}
Let $(Y,T)$ be a minimal t.d.s., and let $\psi: (\M, T) \rightarrow (Y,T)$ be the factor map from the universal minimal t.d.s. $({\bf M}, T)$ to $(Y,T)$. By Theorem \ref{thm-reg}, we have the following diagram.
\begin{equation*}
  \xymatrix@R=0.5cm{
                &         {\M} \ar[dl]_{\theta} \ar[dd]^{\psi^*}    \\
  \M \ar[dr]_{\psi}      \\
                &         Y          }
\end{equation*}
Since $\M$ is the universal minimal t.d.s., $\theta$ is an isomorphism. Thus $\psi$ is regular.
\end{rem}

\subsection{Some lemmas}\ 
\medskip

In the next subsection we show that $(M_\infty({X},\A), \langle T^\infty, \sigma\rangle)$ is an $M$-system.
To do this, we need some preparation. 

The following lemma is well known (for example see \cite[Theorem 5.5-(2)]{AGHSY}). Note that in these lemmas, all spaces in t.d.s. are assumed to be compact and Hausdorff.

\begin{lem} \label{lemma-promimal-finite}
Let $\pi:(X,T)\rightarrow  (Y,T)$ be a proximal extension of minimal t.d.s. $(X,T)$ and $(Y,T)$. Then for each $y\in Y$ and any $x, x_1,x_2,\ldots, x_k\in \pi^{-1}(y)$, there is a net $\{n_i\}\subseteq \N$ such that $T^{n_i}(x_1,x_2,\ldots,x_k)\rightarrow (x,x,\ldots,x), \ i\to\infty$.
\end{lem}

The following lemma is also well known (see e.g. \cite[Chapter V-(2.9)]{Vr}).

\begin{lem}\label{proximal-prod}
Let $\pi_i: (X_i,T)\rightarrow (Y_i,T), i\in I$ be proximal extensions, where $I$ is an index set. Then
$$\prod_{i\in I}\pi_i: \prod_{i\in I} X_i\rightarrow  \prod_{i\in I} Y_i, \ (x_i)_{i\in I}\mapsto (\pi_i(x_i))_{i\in I}$$
is proximal.
\end{lem}

Now we have

\begin{lem} \label{lemma-proximal}
Let $\pi: (X,T)\rightarrow (Y,T)$ be a factor map of minimal t.d.s. and let $\A=\{p_1, p_2,\ldots, p_d\}$
be a set of integral polynomials. If $\pi$ is proximal, then so is $\pi^\infty: (M_\infty(X,\A), \G)\rightarrow (M_\infty(Y,\A),\G)$, where $\G=\langle T^\infty, \sigma\rangle$; and $(M_\infty(X,\A), \G)$ is minimal if and only if $(M_\infty(Y,\A),\G)$ is minimal.
\end{lem}

\begin{proof}
If $\pi$ is proximal, then by Lemma \ref{proximal-prod} $\pi^\infty: (M_\infty(X,\A), T^\infty)\rightarrow (M_\infty(Y,\A),T^\infty)$ is proximal, and hence $\pi^\infty: (M_\infty(X,\A), \G)\rightarrow (M_\infty(Y,\A),\G)$ is also proximal.

It remains to show that $(M_\infty(X), \G)$ is minimal if $(M_\infty(Y),\G)$ is minimal.
By Subsection \ref{subsection-linear-terms},
$ M_\infty(X,\A)$ is isomorphic to $N_\infty(X,\A).$ So it suffices to show that $(N_\infty(X,\A), \G)$ is minimal if $(N_\infty(Y,\A),\G)$ is minimal.

\medskip

Since $N_\infty(X,\A)=\overline{\O}(\w_x^\A,\G)$ for each $ x\in X$, to show that $(N_\infty(X,\A), \G)$ is minimal, we need to show that for any ${\bf x}\in N_\infty (X,\A)$, there is some $x\in X$ such that $\w_x^\A\in \overline{\O}({\bf x}, \G)$.

Let ${\bf x}=({\bf x}_n)_{n\in {\Z}}=\Big((x^{(1)}_n, x^{(2)}_n,\cdots, x^{(d)}_n) \Big)_{n\in {\Z}}\in N_\infty(X,\A).$ Then ${\bf y}=\pi^\infty({\bf x})\in N_\infty(Y,\A)$. By assumption, $(N_\infty(Y,\A), \G)$ is minimal, and there is some $y\in Y$ and a net $\{g_i\}_{i} \subseteq \G$ such that
$$g_i{\bf y}\to \w_y^\A=(\ldots, T^{\vec{p}(-1)}(y^{\otimes d}), \underset{\bullet}{T^{\vec{p}(0)}y^{\otimes d}},T^{\vec{p}(1)}(y^{\otimes d}),T^{\vec{p}(2)}(y^{\otimes d}), \cdots),$$
where $T^{\vec{p}(n)}=T^{p_1(n)}\times T^{p_2(n)}\times\cdots\times T^{p_d(n)}.$ Without loss of generality, assume that
$$g_i{\bf x}\to {\bf \tilde{x}}=({\bf \tilde{x}}_n)_{n\in {\Z}}=\Big((\tilde{x}^{(1)}_n,\tilde{ x}^{(2)}_n,\cdots, \tilde{x_n}^{(d)}) \Big)_{n\in {\Z}}.$$
Then ${\bf \tilde{x}}\in \overline{\O}({\bf x}, \G)$ and $$\pi^\infty({\bf \tilde{x}})=\lim_i \pi^\infty(g_i{\bf x})=\lim_i g_i\pi^\infty({\bf x})=\lim_i g_i {\bf y}=\w_y^\A.$$
Hence
$$ T^{-p_j(n)}\tilde{x_n}^{(j)} \in \pi^{-1}(y), \ \text{for all }\ 1\le j\le d,\ n\in {\Z}. $$
Fix a point $x\in \pi^{-1}(y)$. By Lemma \ref{lemma-promimal-finite}, for any $k\in \N$ there is some net $\{n_i\}_{i}\subseteq \N$ such that
$$\big({T^{((2k+1)d)}}\big)^{n_i}\Big( (T^{\vec{p}(-k)})^{-1}({\bf \tilde{x}}_{-k}), (T^{\vec{p}(-k+1)})^{-1}({\bf \tilde{x}}_{-k+1}),\ldots, (T^{\vec{p}(k)})^{-1}({\bf \tilde{x}}_k)\Big)\to x^{\otimes (2k+1)d}.$$
Since $T^{\vec{p}(-k)}\times \ldots \times T^{\vec{p}(-1)}\times T^{\vec{p}(0)} \times T^{\vec{p}(1)}\times  \ldots \times T^{\vec{p}(k)}$ is continuous, it follows that
$$\big({T^{((2k+1)d)}}\big)^{n_i}\Big({\bf \tilde{x}}_{-k}, {\bf \tilde{x}}_{-k+1}, \ldots, {\bf \tilde{x}}_k\Big)$$
converges to $$(T^{\vec{p}(-k)}x^{\otimes d}, T^{\vec{p}(-k+1)}x^{\otimes d},\ldots, T^{\vec{p}(k)}x^{\otimes d}).$$
We may assume that $(T^\infty)^{n_i}{\bf \tilde{x}}$ converges to
$${\bf \tilde{x}}^{[k]}= \big(({\bf \tilde{x}}_n^k)_{n=-\infty}^{-k-1}, T^{\vec{p}(-k)}x^{\otimes d}, T^{\vec{p}(-k+1)}x^{\otimes d},\ldots, \underset{\bullet}{T^{\vec{p}(0)}x^{\otimes d}}, \ldots, T^{\vec{p}(k)}x^{\otimes d}, ({\bf \tilde{x}}_n^k)_{n=k+1}^\infty\big),$$
where ${\bf \tilde{x}}_n^k\in X^d, |n|\ge k+1 .$
Thus for all $k\in \N$,
$${\bf \tilde{x}}^{[k]} \in \overline{\O}({\bf \tilde{x}}, T^\infty) \subset \overline{\O}({\bf x}, \langle T^\infty, \sigma\rangle).$$
It is clear that $\lim_{k\to\infty} {\bf \tilde{x}}^{[k]}=(T^{\vec{p}(n)}x^{\otimes d})_{n\in \Z}=\w_x^\A $, and we have
$$\w_x^\A\in \overline{\O}({\bf x}, \langle T^\infty, \sigma\rangle).$$
The proof is complete.
\end{proof}

\subsection{The proof of an equivalence form of theorem E}\
\medskip

Now we are ready to prove an equivalence form of Theorem E. In the rest of the paper we put $\G=\langle T^\infty, \sigma\rangle$. By Theorem \ref{Thm-equivalence1}, we show that for a minimal t.d.s., for any family $\A=\{p_1, p_2,\ldots, p_d\}$ of integral polynomials satisfying $(\spadesuit)$, $(M_\infty({X},\A), \G)$ is an $M$-system. Since we will use Theorem \ref{thm-poly-saturate} which requires $X$ being metric, in this Subsection we assume that $X$ is a metric space. But in the proof of the theorem, we will use the universal minimal t.d.s. ${\bf M}$ which is not a metric space.

\subsubsection{Idea of the proof}

Since the proof is long, we describe the idea of the proof first.
Let $(X,T)$ be minimal and $\pi:X\lra X_\infty$ be the factor map. By the discussion in Section \ref{Section-nil},
$(M_\infty({X_\infty},\A), \G)$ is minimal. Using the RIC-diagram (modifications by proximal extensions), we obtain $\pi^*:X^*\lra X_\infty^*$
which is RIC. By Lemma \ref{lemma-proximal}, $(M_\infty({X_\infty^*},\A), \G)$ is also minimal.

The natural idea to show the density of minimal point in $(M_\infty(X^*,\A), \G)$ is that
for a minimal point ${\bf x}$ of $\G$ and a non-empty open set $U\subset M_\infty(X^*,\A)$, one finds some $\phi\in {\rm Aut}(X,T)$ such that $\phi({\bf x})\in U$
is the limit of minimal points. This idea is realized in the proof of Theorem \ref{thm-distal} for the distal case. When we apply
the same idea for the RIC extension, we find that we do not have enough elements in ${\rm Aut}(X,T)$ to realize the idea. Thus, we need to pass our discussion to the universal minimal system to get suitable elements in ${\rm Aut}(X,T)$ for our purpose.

The proof is done in four steps: firstly we construct the extension ({\bf Step 1} below), secondly we transfer the question into
some extension of $(M_\infty(X^*,\A), \G)$ ({\bf Step 2} below);
then we show the set of minimal points is dense in a certain region ({\bf Step 3} below), and finally we use the RIC property to spread minimal points to the whole space ({\bf Step 4} below).

\subsubsection{}
Now we give the main result of this section.

\begin{thm}\label{thm-poly-M-t.d.s}
Let $({X},T)$ be a minimal t.d.s. and let $\A=\{p_1, p_2,\ldots, p_d\}$ be a family of integral polynomials satisfying $(\spadesuit)$. Then $(M_\infty({X},\A), \G)$ is an $M$-system.
\end{thm}

\begin{proof} We split our proof into the following steps.

\medskip
\noindent {\bf Step 1}: The construction of the extensions.

\medskip
Let $X^*,X_\infty^*, \pi^*$ etc. be as in Theorem \ref{thm-poly-saturate}.
Let $\eta': (\M, T)\rightarrow (X^*,T)$ be the factor map from the universal minimal t.d.s. $(\M, T)$ to $(X^*,T)$, and let $\xi'=\pi^*\circ \eta': \M\rightarrow X_\infty^*$. Now using the RIC-diagram for $\xi'$ (Theorem \ref{RIC}), we have the following commutative diagram, where $\theta',\theta$ are proximal and $\psi: \M\rightarrow Y$ is RIC. Also by Remark \ref{rem-reg}, $\psi$ is regular.
\begin{equation*}
\xymatrix{
	X\ar[d]_{\pi}&	{X^*} \ar[l]_{\varsigma^*}\ar[d]_{\pi^*} &  \M \ar[l]_{\eta'}\ar[dl]^{\xi'}  &  \M \ar[l]_{\theta'}\ar[d]^{\psi}  \\
	X_\infty & {X}^*_\infty\ar[l]^{\varsigma} &  & {Y}\ar[ll]^{\theta}	
	}
\end{equation*}
The diagram above can be rewritten as follows:
\begin{equation*}
\xymatrix{
	X\ar[d]_{\pi}&	{X^*} \ar[l]_{\varsigma^*}\ar[d]_{\pi^*} &  \M \ar[l]_{\eta}\ar[dl]^{\xi}  \ar[d]^{\psi}  \\
	X_\infty & {X}^*_\infty\ar[l]^{\varsigma} &   {Y}\ar[l]^{\theta}
	}
\end{equation*}
where $\eta=\eta'\circ\theta'$ and $\xi=\xi'\circ \theta'$. Now let
$$W=(\psi^\infty)^{-1}(M_\infty(Y,\A))\subseteq \M^s\times (\M^{d-s})^{\Z},$$
where $s,d$ are the numbers appearing in the definition of Condition $(\spadesuit)$. Then $W$ is a $\G$-invariant closed subset, that is $(W, \G)$ is a t.d.s. It is clear that $M_\infty(\M,\A)\subseteq W$. We claim that we have the following commutative diagram:
\begin{equation*}
\xymatrix
{
M_\infty(X^*,\A) \ar[d]_{{\pi^*}^\infty}  & W \ar[l]_{\eta^\infty}\ar[d]^{\psi^\infty} \\
M_\infty(X_\infty^*,\A) &  M_\infty(Y,\A)\ar[l]^{\theta^\infty} .
}
\end{equation*}
To see this, note first that $\theta^\infty: (M_\infty(Y,\A), \G)\rightarrow (M_\infty(X_\infty^*,\A), \G)$  is a factor map, and it follows that
\begin{equation*}
  \begin{split}
     W & = (\psi^\infty)^{-1} (M_\infty(Y,\A))\subseteq (\psi^{\infty})^{-1} (\theta^\infty)^{-1}(M_\infty(X_\infty^*,\A)) \\
       & =(\theta^\infty\circ \psi^{\infty})^{-1}(M_\infty(X_\infty^*,\A))=({\pi^*}^\infty\circ \eta^\infty)^{-1}(M_\infty(X_\infty^*,\A))\\
       &= (\eta^\infty)^{-1}\Big(({\pi^*}^\infty)^{-1}(M_\infty(X_\infty^*,\A))\Big).
   \end{split}
\end{equation*}
Thus
$$ M_\infty(X^*,\A)=\eta^{\infty}(M_\infty({\bf M}, \A))\subseteq \eta^\infty (W)\subseteq ({\pi^*}^\infty)^{-1}M_\infty(X_\infty^*,\A).$$
By Theorem \ref{thm-poly-saturate}, $M_\infty(X^*,\A)=({\pi^*}^\infty)^{-1}(M_\infty(X_\infty^*,\A))$. Hence
$$\eta^\infty: (W, \G)\rightarrow (M_\infty(X^*,\A), \G)$$ is a factor map.

\medskip
\noindent {\bf Step 2}: Transitions of the question.

\medskip
By Theorem \ref{thm-poly-saturate}, $(M_\infty(X^*,\A), \G)$ is a transitive t.d.s. To show $(M_\infty(X^*,\A),\G)$ is an $M$-system, it suffices to show that $\G$-minimal points are dense in $(W,\G)$, as $(M_\infty(X^*,\A), \G)$ is a factor of $(W, \G)$.

For convenience, sometimes we denote the point ${\bf x}$ in $\M^s\times (\M^{d-s})^{\Z}$ by
${\bf x}=(x_i)_{i\in I}$ instead of $\small {\bf x}=\Big((x_1,x_2, \ldots, x_s), \big((x^{(1)}_n, x^{(2)}_n,\ldots,
x^{(d-s)}_n)\big)_{n\in {\Z}} \Big)$, and similarly for the points in $Y^s\times (Y^{d-s})^{\Z}$.

\medskip

By Corollary \ref{cor-nil}, $(M_\infty(X_\infty,\A),\G)$ is minimal. Since $\varsigma$ and $\theta$ are proximal,
by Lemma~\ref{lemma-proximal}, $(M_\infty(Y,\A),\G)$ is minimal, and hence $T^\infty$-minimal points are dense
in $M_\infty(Y,\A)$ (Lemma \ref{den-minimal}). Since $\psi$ is open, thus to show that $\G$-minimal points are dense in $(W,\G)$, it
suffices to show for each $T^\infty$-minimal point ${\bf y}\in M_\infty(Y,\A)$, and each ${\bf x}\in (\psi^\infty)^{-1}({\bf y})$,
there is a net $\{{\bf x}_{\a}\}_\a\subseteq W$ such that each ${\bf x}_\a$ is $\G$-minimal and ${\bf x}_\a\rightarrow {\bf x}$.

\medskip

Now fix a $T^\infty$-minimal point ${\bf y}=(y_i)_{i\in I}\in M_\infty(Y, \A)$. Since $(M_\infty(Y,\A),\G)$ is minimal, ${\bf y}$ is also $\G$-minimal. Thus there exists a $\G$-minimal point
${\bf x}'=(x'_i)_{i\in I}\in (\psi^\infty)^{-1}({\bf y})$ by Lemma \ref{den-minimal}. For all $i\in I$, since $x_i'\in {\bf M}$, there is some minimal idempotent $u_i\in \M$ such that $u_ix_i'=x_i'$. Thus
$${\bf x'}\in \prod_{i\in I} u_i\psi^{-1}(y_i),$$ since $x_i'\in \psi^{-1}(y_i)$ for each $i\in I$.
We remark that if $\psi$ is distal, then $u_i\psi^{-1}(y_i)=\psi^{-1}(y_i)$ for each $i\in I$.
\medskip

\noindent {\bf Step 3}: \ {\em For each ${\bf \widetilde{x}}\in \prod_{i\in I} u_i\psi^{-1}(y_i)$, there is a sequence $\{{\bf z}_k\}_{k\in \N}\subseteq \prod_{i\in I} u_i\psi^{-1}(y_i)$ such that ${\bf z}_k$ is $\G$-minimal and ${\bf z}_k\rightarrow {\bf \widetilde{x}}, k\to\infty$.
In particular, $\G$-minimal points are dense in $\prod_{i\in I} u_i\psi^{-1}(y_i)$. }

\medskip

First we show the case for $\A=\{p\}$ with $\deg p\ge 2$. In this case, $I={\Z}$, and the proof is clearer than the general case. Then we give the proof of the general case.

Let ${\bf \widetilde{x}}=(\widetilde{x}_i)_{i=-\infty}^\infty \in \prod_{i=-\infty}^\infty u_i\psi^{-1}(y_i)$. For each $i\in {\Z}$, $x_i',\ \widetilde{x}_i\in u_i\M$. Since $\psi$ is regular, there is some $\phi_i\in {\rm Aut}_\psi(\M)$ such that $\widetilde{x}_i=\phi_i(x_i')$.
For each $k\in \N$, let
\begin{equation*}
  \begin{split}
    \Phi_k& =(\phi_{-k}\times \phi_{-k+1}\times \cdots \times \phi_k)^\infty\\
    &=\cdots \times (\phi_{-k}\times \phi_{-k+1}\times \ldots \times \phi_k)\times (\phi_{-k}\times \phi_{-k+1}\times \ldots \times \phi_k)\times \cdots.
   \end{split}
\end{equation*}
Then
\begin{equation*}
  \begin{split}
     \Phi_k({\bf x'})&=(\ldots, \phi_k(x_{-k-1}'), \widetilde{{x}}_{-k},\ldots,\underset{\bullet}{\widetilde{x_0}}, \ldots, \widetilde{{x}_k}, \phi_{-k}(x_{k+1}'),\ldots, \phi_k(x_{3k+1}'),\phi_{-k}(x_{3k+2}'),\ldots)\\
     & \rightarrow \widetilde{{\bf x}}, \ k\to\infty.
   \end{split}
\end{equation*}
Since $\phi_j\in {\rm Aut}_\psi(\M)$ for each $j\in\Z$, we have that $\phi_j\psi^{-1}(y_i)=\psi^{-1}(y_i)$ and
$$\phi_j(u_i\psi^{-1}(y_i))=u_i\phi_j(\psi^{-1}(y_i))=u_i\psi^{-1}(y_i)$$ for each $i\in \Z$
and hence $\Phi_k({\bf x'})\in \prod_{i=-\infty}^\infty u_i\psi^{-1}(y_i)\subseteq W.$


It is left to show that $\Phi_k({\bf x'})$ is $\G$-minimal.
For all $k\in \N$, by Lemma \ref{lem-minimal-point}, ${\bf x'}$ is a $\langle T^\infty,\sigma^{2k+1} \rangle$-minimal point as ${\bf x'}$ is $\G$-minimal.
Note that
$$\sigma^{2k+1}\Phi_k=\Phi_k \sigma^{2k+1}.$$
It follows that $\Phi_k({\bf x'})$ is a $\langle T^\infty,\sigma^{2k+1} \rangle$-minimal point. Again by Lemma \ref{lem-minimal-point}, it is $\G$-minimal.

\medskip

Now we show the general case. Now
$$\small {\bf y}=(y_i)_{i\in I}=\Big((y_1, \ldots, y_s), \big((y^{(1)}_n, \cdots, y^{(d-s)}_n)\big)_{n\in {\Z}} \Big),$$
$$\small {\bf \widetilde{x}}=(\widetilde{x}_i)_{i\in I}=\Big((\widetilde{x}_1, \ldots, \widetilde{x}_s), \big((\widetilde{x}^{(1)}_n, \cdots, \widetilde{x}^{(d-s)}_n)\big)_{n\in {\Z}} \Big),$$
$$\small {\bf x'}=\Big((x'_1,\ldots, x'_s), \big((x'^{(1)}_n,\cdots, x'^{(d-s)}_n)\big)_{n\in {\Z}} \Big)$$ and
$$\small \prod_{i\in I} u_i\psi^{-1}(y_i)=\prod_{i=1}^su_i\psi^{-1}(y_i)\times \prod_{n=-\infty}^\infty \big(u_n^{(1)}\psi^{-1}(y_n^{(1)})\times \cdots \times u_n^{(d-s)}\psi^{-1}(y_n^{(d-s)})\big).$$
For each $i\in I$, $x_i',\ \widetilde{x}_i\in u_i\M$. Since $\psi$ is regular, there is some $\phi_i\in {\rm Aut}_\psi(\M)$ such that $\widetilde{x}_i=\phi_i(x_i')$. For each $k\in \N$, let
$$\Phi_k=(\phi_1\times \cdots \times \phi_s)\times \Big(\prod_{j=-k}^k \big(\phi_j^{(1)}\times \cdots \times \phi_j^{(d-s)}\big) \Big)^\infty,$$
where
\begin{equation*}
  \begin{split}
\small &\Big(\prod_{j=-k}^k \big(\phi_j^{(1)}\times \cdots \times \phi_j^{(d-s)}\big) \Big)^\infty\\
&=\cdots \times \prod_{j=-k}^k \big(\phi_j^{(1)}\times \cdots \times \phi_j^{(d-s)}\big)\times \prod_{j=-k}^k \big(\phi_j^{(1)}\times \cdots \times \phi_j^{(d-s)}\big)\times \cdots.
 \end{split}
\end{equation*}

Note since $\phi_i\in {\rm Aut}_\psi(\M)$, $\phi_i(u_j\psi^{-1}(y_j))=u_j\psi^{-1}(y_j)$ for all $i,j\in I$. Thus
\begin{equation*}
  \begin{split}
    \Phi_k({\bf x'})&=(\phi_1\times \cdots \times \phi_s)\times \Big(\prod_{j=1}^k \big(\phi_j^{(1)}\times \cdots \times \phi_j^{(d-s)}\big) \Big)^\infty (\bf x')\\
    &= \Big((\widetilde{x}_1, \ldots, \widetilde{x}_s), \big((\phi_j^{(1)}({x_j'}^{(1)}), \ldots, \phi_j^{(d-s)}({x_j'}^{(d-s)}))\big)_{j=-\infty}^{-k-1} \\
     & \quad \quad \quad \quad \big((\widetilde{x}^{(1)}_j, \ldots, \widetilde{x}^{(d-s)}_j)\big)_{j=1}^k, \big((\phi_j^{(1)}({x_j'}^{(1)}), \ldots, \phi_j^{(d-s)}({x_j'}^{(d-s)}))\big)_{j=k+1}^\infty \Big)\\
     &\in u_1\psi^{-1}(y_1)\times \cdots \times u_s\psi^{-1}(y_s)\times \prod_{n=-\infty}^\infty \big(u_n^{(1)}\psi^{-1}(y_n^{(1)})\times \cdots \times u_n^{(d-s)}\psi^{-1}(y_n^{(d-s)})\big)
   \end{split}
\end{equation*}
and hence
$$\Phi_k({\bf x'})\rightarrow {\bf \widetilde{x}}, \ k\to\infty.$$
Now we show that $\Phi_k({\bf x'})$ is $\langle T^\infty,\sigma \rangle$-minimal. This follows by the same arguments as the simple case we just proved.
This ends the {\bf Step 3}.

\medskip

\noindent {\bf Step 4}: The final arguments of the proof using the RIC property.

\medskip
Finally, we show that for each ${\bf x}=(x_i)_{i\in I}\in (\psi^\infty)^{-1}({\bf y})$, there is a net $\{{\bf x}_{\a}\}_\a\subseteq W$ such that each ${\bf x}_\a$ is $\G$-minimal and ${\bf x}_\a\rightarrow {\bf x}$.

Since ${\bf y}=(y_i)_{i\in I}\in M_\infty(Y)$ is $T^\infty$-minimal, there is some minimal idempotent $v\in \M$ such that $v{\bf y}={\bf y}$. Thus $vy_i=y_i$ for all $i\in I$. As $\psi$ is RIC, we have (see (\ref{cir-op}))
$$v\circ\big( u_j\psi^{-1}(y_i)\big)=\psi^{-1}(vy_i)=\psi^{-1}(y_i), \ \forall i\in I.$$
It follows that
\begin{equation}\label{w2}
  v\circ \Big(\prod_{i\in I} u_i\psi^{-1}(y_i)\Big)=\prod_{i\in I} \psi^{-1}(y_i)= (\psi^\infty)^{-1}({\bf y}).
\end{equation}
Since ${\bf x}=(x_i)_{i\in I}\in (\psi^\infty)^{-1}({\bf y})$, by \eqref{w2}, for any net $\{n_\a\}_\a\subseteq \Z$ converging to $v$, there is some ${\bf z}_\a\in  \prod_{i\in I} u_i\psi^{-1}(y_i)$ such that
$$(T^{\infty})^{n_\a} {\bf z}_\a\rightarrow {\bf x}.$$
By the definition it is easy to see $p\circ A=p\circ B, p\in \M$, whenever $A$ is dense in $B$. By {\bf Step~ 3}, we may assume that each ${\bf z}_\a$ is $\G$-minimal. Let ${\bf x}_\a=(T^{\infty})^{n_\a} {\bf z}_\a$. Then ${\bf x}_\a$ is $\G$-minimal and ${\bf x}_\a\rightarrow {\bf x}$. The proof is complete.
\end{proof}

\subsection{The proof of Theorem E}\
\medskip

Theorem E follows from Lemma \ref{ww=dem}, Theorem \ref{Thm-equivalence1} and Theorem \ref{thm-poly-M-t.d.s}. 


\subsection{Proof of Corollary F}\
\medskip
In this subsection we show Corollary F.
\begin{proof}[Proof of Corollary F] We will show the corollary for $N_\infty(X,\A)$, and then we have the same result for
$M_\infty(X,\A)$ by Lemma \ref{M-N-equ}. Without loss of generality, we assume that $\A$  satisfies $(\spadesuit)$.
Recall that for each $x\in X$
\begin{equation*}
  \w_x^\A=(T^{\vec{p}(n)}x^{\otimes d})_{n\in \Z}=\big ((T^{p_1(n)}x, T^{p_2(n)}x,\ldots, T^{p_d(n)}x) \big)_{n\in \Z}
  \in (X^d)^{{\Z}},
\end{equation*}
and
\begin{equation*}
  N_\infty(X,\A)=\overline{\bigcup\{\O(\w_x^\A,\sigma): x\in X\}}\subseteq (X^d)^{{\Z}}.
\end{equation*}


(1) First we assume that $(X,T)$ has  dense $T$-minimal points and $x\in X$. Then there are $T$-minimal points $x_i\in X$ such that $x_i\rightarrow x, i\to\infty$. It is clear that $\omega_{x_i}^\A\rightarrow \omega_x^\A, i\to\infty$ and for each $i\in\N$, $(\overline{\O}(x_i,T),T)$
 is a minimal t.d.s. By Lemma \ref{ww=dem}, Theorem \ref{Thm-equivalence1} and Theorem \ref{thm-poly-M-t.d.s},
$\big(N_\infty(\overline{\O}(x_i,T),\A),\G\big)$ is an $M$-system for each $i\in\Z$, and thus $\omega_{x_i}^\A$ is the limit of
$\G$-minimal points which implies that $\omega_{x}^\A$ is the limit of $\G$-minimal points. Then by the definition,
the set of $\G$-minimal points is dense in $N_\infty(X,\A).$

\medskip
Conversely, assume that $N_\infty(X,\A)$ has dense $\G$-minimal points. Applying Lemma \ref{den-minimal} to the orbit closures of the minimal points,
we get that the $T^\infty$-minimal points are dense in
$N_\infty(X,\A)$. It follows by projection to $0$-th coordinate that the set of $T$-minimal points is dense in $X$.

\medskip
(2) Now we assume that $(X,T)$ is an $M$-system. Then for each transitive point $x$ of $(X,T)$, $\w_x^\A$ is a transitive point of $(N_\infty(X,\A),\G)$. By (1) the set of $\G$-minimal points is dense in $N_\infty(X,\A)$, and it follows that $(N_\infty(X,\A),\G)$ is an $M$-system.
\medskip

Conversely, suppose that $(N_\infty(X,\A),\G)$ is an $M$-system. By (1) the set of $T$-minimal points is dense in $X$.
To prove the transitivity of $(X,T)$ we assume that ${\bf x}$
is a transitive point of $N_\infty(X,\A)$. There are a sequence $\{x_i\}_{i\in\N}$ of $X$ and a sequence $\{n_i\}_{i=1}^\infty$ of $\Z$ such that
$\sigma^{n_i}\omega_{x_i}^\A\rightarrow {\bf x}, i\to\infty.$ 
Then it is easy to verify that for each $\ep>0$, there is $i=i(\ep)$ such that
${\O}(\sigma^{n_i}\omega_{x_i}^\A,\G)$ is $\ep$-dense in $N_\infty(X,\A)$.\footnote{Recall the definition of $\ep$-dense subsets: if $B$ is a subset of a metric space $S$, then $B$  is $\ep$-dense in $S$ for given $\ep>0$ if for any $s\in S$, there is $b\in B$ such that the distance between $s$ and $b$ is less than $\ep$.} 
Hence, ${\O}(x_i,T)$ is $\ep$-dense in $X$ since
the $0$-th coordinate of $\sigma^n(T^\infty)^m\big(\sigma^{n_i}\omega_{x_i}^\A\big)$ is $(T^{m+p_1(n+n_i)}x_i,\ldots, T^{m+p_1(n+n_i)}x_i)$
for each $n,m\in\Z$.

To sum up, for each $\ep>0$, there is some $x_i\in X$ such that ${\O}(x_i,T)$ is $\ep$-dense in $X$, which implies that $(X,T)$ is transitive.
\end{proof}

\section{Combinatoric consequences}\label{Section-combinatorial}

In this section, we give some combinatoric consequences. Theorem A is a corollary of the following theorem.

\begin{thm}\label{thm-comb1}
Let $S\subseteq \Z $ and $d\in \N$, and let $p_i$ be an integral polynomial with $p_i(0)=0$, $1\le i\le d$. If $S\in \F_{ps}(\Z)$, then
there is $A\in\F_{ps}(\Z^2)$ such that for any $N\in \N$, there is some ${\bf m}_N\in \Z^2$ satisfying
$${\bf m}_N+(A\cap [-N,N]^2)\subseteq \{(m,n)\in\Z^2: m+p_1(n), m+p_2(n),\ldots, m+p_d(n)\in S\}.$$

In particular, $\{(m,n)\in\Z^2: m+p_1(n), m+p_2(n),\ldots,m+p_d(n)\in S\}\in \F_{ps}(\Z^2).$
\end{thm}

\begin{proof}
Let $\Sigma_2=\{0,1\}^\Z$ and $\sigma: \Sigma_2\rightarrow \Sigma_2$ be the shift, that is, for ${\bf x}=(x_n)_{n\in \Z}\in \Sigma_2$,
$(\sigma{\bf x})_n=x_{n+1}, \forall n\in \Z$. Let the metric $\rho$ on $\Sigma_2$ be defined as follows: for ${\bf x}=(x_n)_{n\in \Z}, {\bf y}=(y_n)_{n\in \Z}$,
$$\rho({\bf x}, {\bf y})=0\ \text{if}\ {\bf x}={\bf y};\ \  \rho({\bf x}, {\bf y})=\frac{1}{k+1}, \ \text{if}\ {\bf x}\not={\bf y}\ \text{and}\ k=\min\{|i|: x_i\neq y_i\}.$$

\medskip

Let ${\bf x}=1_S\in \Sigma_2$ be the indicator function, i.e., $x_n=1$ if and only if $n\in S$.
Let $Z=\overline{\O}({\bf x},\sigma)$. Since $S$ is piecewise syndetic, $Z$ contains a point $1_{S'}$ with $S'$ actually syndetic. Since $S'$ is syndetic, no limit point of translates of $1_{S'}$ is equal to 
${\bf 0}=(\ldots 000\ldots)$. Let $X$ be a minimal subset in orbit closure of $1_{S'}$. Then for any ${\bf y}\in X$, ${\bf y}\neq {\bf 0}$.

Let $U=\{{\bf z}=(z_n)_{n\in \Z}\in X: z_0=1\}$. By the choice of $X$, $U$ is a non-empty open subset of $X$. By Theorem E, there is some ${\bf y}\in U$ such that
$$A=N^{\Z^2}_{\{p_1,\ldots,p_d\}}({\bf y},U)=\{(m, n)\in\Z^2: \sigma^{m+p_1(n)}{\bf y}\in U, \ldots, \sigma^{m+p_d(n)}{\bf y}\in U\}$$
is piecewise syndetic in $\Z^2$. That is,
$$A=\{(m,n)\in \Z^2: y_0= y_{m+p_1(n)}= y_{m+p_2(n)}=\ldots=y_{m+p_d(n)}=1 \}\in \F_{ps}(\Z^2).$$

Let $$B=\{(m , n)\in\Z^2: m+p_1(n), m+p_2(n),\ldots, m+p_d(n)\in S\}.$$
We show that for any $N\in \N$, there is some ${\bf m}_N\in \Z^2$ such that
$${\bf m}_N+(A\cap [-N,N]^2)\subseteq B,$$
then by Lemma \ref{lem-pw}, $B$ is also a piecewise syndetic subset in $\Z^2$.

\medskip

Let $N\in \N$ and set
$$A\cap [-N,N]^2=\{(m_i, n_i)\}_{i=1}^{l_N}.$$
Then
$$y_{m_i+ p_j(n_i)}=y_0=1, \ \forall 1\le j\le d, \ 1\le i\le l_N.$$
Note that ${\bf y}\in X\subseteq Z=\overline{\O}({\bf x},\sigma)$.
Choose $m \in \Z$ such that $\sigma^{m}{\bf x}$ is close to ${\bf y}$ such that $\rho(\sigma^{m}{\bf x}, {\bf y})<\frac{1}{\max\{|m_i|+|p_j(n_i)|: 1\le j\le d, 1\le i\le l_N\}}$.
It follows that
$$x_{m+k}=y_k, \ \forall |k|\le \max\{|m_i|+|p_j(n_i)|: 1\le j\le d, 1\le i\le l_N\}.$$
In particular,
$$x_{m+ m_i+p_j(n_i)}= y_{m_i+p_j(n_i)}=y_0=1, \ \forall 1\le j\le d, \ 1\le i\le l_N.$$
So we have
$$m+ m_i+ p_j(n_i) \in S, \ \forall 1\le j\le d, \ 1\le i\le l_N ,$$
which implies that
$${\bf m}_N+(A\cap [-N,N]^2)\subseteq B , $$
where ${\bf m}_N=(m,0)$.
Thus $B$ is piecewise syndetic in $\Z^2$. The proof is complete.
\end{proof}

As a corollary, we have Furstenberg-Glasner's theorem which was mentioned in the introduction.
With the almost same proof of Theorem \ref{thm-comb1} using Theorem D instead of Theorem E, we have the following result.

\begin{thm}\label{thm-9.3}
Let $S\subseteq \N$ and $d\in \N$, and let $p_i$ be an integral polynomial with $p_i(0)=0$ for each $1\le i\le d$. If $S\in \F_{ps}(\Z)$, then there is $A\in\F_{ps}(\Z)$
such that for any $N\in \N$, there is some ${m}_N\in \Z$ satisfying
$$A\cap [-N,N]\subseteq \{n\in\Z:  m_N+p_1(n), m_N+p_2(n),\ldots,m_N+p_d(n)\in S\}.$$
\end{thm}


\section{Some open questions}

We have some natural questions for the further investigation.

\begin{ques}
Assume that $k,d\in \N$, $S\subset \Z^k$ is piecewise syndetic, and $p_{i,j}$ is an integral polynomial with $p_{i,j}(0)=0$ for each $1\le i\le d$
and $1\le j\le k$. Consider the set

\begin{align}\label{mpoly}
\begin{aligned}&\{(m_1,\ldots,m_k,n)\in \Z^{k+1}: (m_1+p_{1,1}(n),\ldots,m_k+p_{1,k}(n))\in S,\\
& \quad \quad \quad \ldots, (m_1+p_{d,1}(n),\ldots,m_k+p_{d,k}(n))\in S\}.
\end{aligned}
\end{align}
Is it true that the (\ref{mpoly}) is piecewise syndetic in $\Z^{k+1}$?
\end{ques}

The corresponding dynamical version is the following

\begin{ques}Let $k,d\in\N$ and $(X,\langle T_1,\ldots,T_k\rangle)$ be a minimal t.d.s., where $T_i:X\rightarrow X$ is a homeomorphism and
$T_i\circ T_j=T_j\circ T_i$ for $1\le i,j\le k$. Let $p_{i,j}$ be an integral polynomial with $p_{i,j}(0)=0$ for each $1\le i\le d$
and $1\le j\le k$.

Is it true that for each $x\in X$ and each neighbourhood $U$ of $x$
\begin{align}\label{mpolytopo}
\begin{aligned}
&\{(m_1,\ldots,m_k,n)\in\Z^{d+1}: T_1^{m_1+p_{1,1}(n)}\cdots T_k^{m_k+p_{1,k}(n)}x\in U,\\
 &\quad \quad \quad \quad \quad \quad \quad \quad  \quad \quad \quad  \ldots, T_1^{m_1+p_{d,1}(n)}\cdots T_k^{m_k+p_{d,k}(n)}x\in U\}
\end{aligned}
\end{align}
is piecewise syndetic in $\Z^{k+1}$?
\end{ques}

The above two questions can also be asked for 
a nilpotent group.

\medskip
We remark that for $k=1$ the two questions have been answered in the current paper.
Moreover, it seems that the tools in the current paper can not be applied directly to the above questions, namely
one needs to work out some kind ``saturation theorem" for commuting homeomorphisms.


\begin{thebibliography}{SS}

\bibitem{Ak97} E. Akin, \textit{ Recurrence in topological dynamical systems: Furstenberg families and Ellis actions}, Plenum Press, New York, 1997.

\bibitem{AGHSY} E. Akin, E. Glasner, W. Huang, S. Shao and X. Ye, \textit{Sufficient conditions under which a transitive system is chaotic},
 Ergodic Theory Dynam. Systems, {\bf 30} (2010), 1277--1310.



\bibitem{Au66} J. Auslander, \textit{Regular minimal sets. I.}, Trans. Amer. Math. Soc., {\bf 123} (1966), 469--479.

\bibitem{Au88} J. Auslander, \textit{Minimal flows and their extensions}, North-Holland Mathematical Studies, vol.153, North-Holland, Amsterdam.

\bibitem{AF} J. Auslander and H. Furstenberg, \textit{Product recurrence and
distal points}, Trans. Amer. Math. Soc., \textbf{343} (1994),
221--232.

\bibitem{Beiglbock} M. Beiglb\"{o}ck, \textit{Arithmetic progression in abundance by combinatorial tools}, Proc. Amer. Math. Soc. {\bf 137} (2009), 3981--3983.




\bibitem{BH01} V. Bergelson and N. Hindman, \textit{Partition Regular Structures Contained in
Large Sets Are Abundant},  J. Combin. Theory Ser. A,  {\bf 93}(2001), 18--36.




\bibitem{BL96} V. Bergelson and A. Leibman, \textit{Polynomial extensions of van der Waerden's and
Szemer\'edi's theorems},  J. Amer. Math. Soc., {\bf 9} (1996), 725--753.


\bibitem{D-Y} P. Dong, S. Donoso, A. Maass, S. Shao and X. Ye, \textit{Infinite-step nilsystems, independence andcomplexity},   Ergodic Theory Dynam. Systems, {\bf 33}(2013), 118--143.

\bibitem{EGS} R. Ellis, S. Glasner and L. Shapiro, \textit{Proximal-Isometric Flows},  Adv. Math., {\bfseries 17} (1975), 213--260.

\bibitem{Fort} M. K. Fort, Jr., \textit{Points of continuity of semi-continuous functions}, Publ. Math. Debrecen,
{\bf 2} (1951), 100--102.







\bibitem{F77} H. Furstenberg, \textit{Ergodic behavior of diagonal measures and a theorem of Szemer\'edi on arithmetic progressions}.  J. Anal. Math., {\bf 31} (1977), 204--256.

\bibitem{F} H. Furstenberg, \textit{Recurrence in ergodic theory and combinatorial number theory}, M. B. Porter Lectures. Princeton University Press, Princeton, N.J., 1981.

\bibitem{F81} H. Furstenberg, \textit{Poincar\'e recurrence and number theory}, \text{ Bull. Amer. Math. Soc.},  {\bf 5} (1981), 211--234.


\bibitem{FG98} H. Furstenberg and  E. Glasner, \textit{Subset dynamics and van der Waerden's theorem},
Contemp. Math., {\bf 215} (1998), 197--203.



\bibitem{FW} H. Furstenberg and B. Weiss, \textit{Topological dynamics and combinatorial number theory},  J. Anal. Math., {\bf 34} (1978), 61--85.


\bibitem{G92} E. Glasner, \textit{Regular PI metric flows are equicontinuous}, Proc.
Amer. Math. Soc., {\bf 114} (1992), 269--277.

\bibitem{G94} E. Glasner, \textit{Topological ergodic decompositions and applications to products of powers of a minimal transformation}, J. Anal. Math., {\bf 64} (1994), 241--262.

\bibitem{GGY} E. Glasner, Y. Gutman and X. Ye, \textit{ Higher order regionally proximal equivalence relations for general minimal group actions}, Adv. Math., {\bf 333} (2018), 1004-1041.


\bibitem{GHSWY} E. Glasner, W. Huang, S. Shao, B. Weiss and X. Ye, \textit{Topological characteristic factors and nilsystems}, arXiv:2006.12385, to appear in  J. Eur. Math. Soc.



\bibitem{GH} W. Gottschalk and G. Hedlund, \textit{Topological dynamics}, American Mathematical Society Colloquium Publications, Vol. 36. American Mathematical Society, Providence, R. I., 1955. vii+151 pp.

\bibitem{GMV20}  Y.  Gutman, F. Manners and P. Varj$\acute{\text{u}}$,  \textit{The structure theory of nilspaces III: Inverse limit representations and topological dynamics}, Adv. Math., {\bf 365} (2020), 107059, 53 pp.








\bibitem{HK18} B. Host and B. Kra, \textit{Nilpotent Structures in Ergodic Theory},
Mathematical Surveys and Monographs {\bf 236}, AMS, 2018.

\bibitem{HKM} B. Host, B. Kra and A. Maass, \textit{Nilsequences and a structure theory for topological dynamical systems},  Adv. Math., {\bf 224} (2010) 103--129.

\bibitem{HY05} W. Huang and X. Ye, \textit{Dynamical systems disjoint from any minimal system}, {Trans. Amer. Math. Soc.}, {\bf 357} (2005), 669--694.

\bibitem{HSY-19-1} W. Huang, S. Shao and X. Ye, \textit{Topological correspondence of multiple ergodic averages of nilpotent actions}, J. Anal. Math.,  {\bf 138} (2019), 687--715.

\bibitem{HSY-21} W. Huang, S. Shao and X. Ye, \textit{Multiply minimal points for the product of iterates}, arXiv:2103.16759, to appear in Israel J. Math.

\bibitem{HSY-new} W. Huang, S. Shao and X. Ye, \textit{Polynomial Furstenberg joinings and its applications}, preprint, 2022.



\bibitem{Kura1} K. Kuratowski, \textit{Topology, Vol. I}, Acad. Press, New York, N.Y., 1966.

\bibitem{Kura2} K. Kuratowski, \textit{Topology, Vol. II}, Acad. Press, New York, N.Y., 1968.

\bibitem{Leibman051} A. Leibman, \textit{ Pointwise convergence of ergodic averages for polynomial sequences of translations on a nilmanifold}, Ergodic Theory Dynam. Systems, {\bf 25} (2005), 201--213.

\bibitem{Leibman052} A. Leibman, \textit{Pointwise convergence of ergodic averages for polynomial actions of ${\mathbb Z}^d$ by translations on a nilmanifold}, Ergodic Theory Dynam. Systems, {\bf 25} (2005),  215--225.


\bibitem{P} W. Parry, \textit{Ergodic properties of affine transformations and flows on nilmanifolds}, Amer. J. Math., {\bf 91} (1969), 757--771.


\bibitem{Qiu} J. Qiu, \textit{Polynomial orbits for totally minimal systems}, arXiv:2202.08782, 2022.





\bibitem{SY} S. Shao and X. Ye, \textit{Regionally proximal relation of order $d$ is an equivalence one for minimal systems and a combinatorial consequence},  Adv. Math., {\bf 231}(2012), 1786--1817.




\bibitem{Veech} W. A. Veech, \textit{Point-distal flows.} Amer. J. Math.,  {\bf 92} (1970), 205--242.

\bibitem{Vr} J. de Vries, \textit{Elements of Topological Dynamics}, Kluwer Academic Publishers, Dordrecht, 1993.



\bibitem{WXY} Q. Wu, H. Xu and X. Ye, \textit{On structure theorems and non-saturated examples,}
arXiv:2201.00152, to appear in Commun. Math. Stat.

\end{thebibliography}
\end{document}